\documentclass[11pt]{amsart}
\usepackage{amsmath}
\allowdisplaybreaks
\usepackage{amssymb,latexsym}
\usepackage[matrix,arrow]{xy}
\usepackage[usenames]{color}
\usepackage{slashed}
\usepackage{parskip}
\usepackage{amsmath,amssymb,graphics}
\usepackage{mathrsfs}
\usepackage{graphicx}

\newcommand{\mathsym}[1]{{}}

\newtheorem{thm}{Theorem}[section]
\newtheorem{cor}[thm]{Corollary}
\newtheorem{lem}[thm]{Lemma}
\newtheorem{prop}[thm]{Proposition}

\theoremstyle{definition}
\newtheorem{defn}{Definition}[section]
\numberwithin{equation}{section}
\theoremstyle{remark}

\theoremstyle{example}
\newtheorem{exa}{Example}[section]

\newcommand{\md }{\mathrm{d}}

\newcommand{\be}{\begin{equation}}
\newcommand{\ee}{\end{equation}}
\newcommand{\bag}{\begin{eqnarray}}
\newcommand{\eag}{\end{eqnarray}}
\newcommand{\ban}{\begin{eqnarray*}}
\newcommand{\ean}{\end{eqnarray*}}
\newcommand{\ba}{\begin{aligned}}
\newcommand{\ea}{\end{aligned}}

\newcommand{\bpf}{\begin{proof} }
\newcommand{\epf}{\end{proof} }

\newcommand{\ric}{\mathrm{Ric}}
\newcommand{\tr}{\mathrm{tr}}

\newcommand{\dbar}{\overline{\partial}}

\begin{document}

\title{The Dirichlet Problem of Fully Nonlinear Equations on Hermitian Manifolds}
\author{Ke Feng}
\address{School of Mathematical Sciences, Peking University, Beijing 100871, China}
\email{kefeng@math.pku.edu.cn}
\author{Huabin Ge$^*$}
\address{School of Mathematics, Renmin University of China, Beijing, 100872, P.R. China}
\email{hbge@bjtu.edu.cn}
\author{Tao Zheng}
\address{School of Mathematics and Statistics, Beijing Institute of Technology, Beijing 100081, China}
\email{zhengtao08@amss.ac.cn}
\thanks {$^*$Supported by National Natural Science Foundation of China (NSFC) grant No. 11871094}
\begin{abstract}
We study the Dirichlet problem of a class of fully nonlinear elliptic equations on Hermitian manifolds and derive a priori $C^2$ estimates which depend on the initial data on manifolds, the admissible subsolutions and the upper bound of the gradients of the solutions. In some special cases, we  obtain the gradient estimates, and hence we can solve the corresponding Dirichlet problem with admissible subsolutions. We also study the Hessian quotient equations and $(m-1,m-1)$-Hessian quotient equations on compact Hermitian manifolds without boundary.
\end{abstract}
\maketitle

\section{Introduction}\label{secintro}
Let $(M,J,g)$ be a compact Hermitian manifold with smooth boundary, $\dim_{\mathbb{C}}M=m$, and the canonical complex structure $J$, where $g$ is the Hermitian metric, i.e., $g(JX,JY)=g(X,Y)$ for any vector fields $X,\,Y\in \mathfrak{X}(M)$. Then we can define a real $(1,1)$ form $\omega$  by
\begin{equation*}
\omega(X, Y):=g(JX, Y),\quad \forall\;X,\,Y\in \mathfrak{X}(M).
\end{equation*}
This form $\omega$ is determined uniquely by $g$ and vice versa, and hence we will not distinguish the terms in the following.

Fix a real $(1,1)$ form $\chi$ which is not necessarily positive definite.  Let $W (\mathrm{d}u)$ be a real $(1,1)$ form which depends linearly on $\mathrm{d}u$. Then for any $u\in C^2(M,\mathbb{R}),$ we
define a new real $(1,1)$ form $\vartheta$ by
\begin{equation}
 \label{defnvartheta}
 \vartheta := \chi  +\sqrt{-1}\partial\overline{\partial} u  + W(\mathrm{d}u).
 \end{equation}
Note that we do not assume that $\vartheta$ is positive definite, and the form $\vartheta$ defines an endomorphism $\vartheta^{\flat}$ of $T^{1,0}M$ which is Hermitian with respect to $\omega,$ i.e.,
\begin{equation}
\label{flat}
\omega\left(X,  \overline{\vartheta^{\flat}(Y)}\right)=\omega(\vartheta^{\flat}(X), \bar Y)=\vartheta(X,\bar Y), \quad\forall\;X,\,Y\in \Gamma(T^{1,0}M),
\end{equation}
where $\Gamma(\bullet)$ denotes the set of smooth sections of vector bundle $\bullet.$

In the following, we denote by $\lambda(\vartheta^\flat)$ the $m$-tuple of eigenvalues of $\vartheta^\flat$ (i.e., the eigenvalues of $\vartheta$ with respect to the Hermitian metric $\omega$), and use \eqref{flat} as the definition of the operator ${}^{\flat}$. Note that for any real symmetric section $B$ of $T^*M\otimes T^*M$, one infers that
\begin{equation}
\label{realflat}
g(B^\flat(X),Y):=B(X,Y),\quad \forall\;X,\,Y\in\mathfrak{X}(M).
\end{equation}
For convenience, we use the notation
\begin{equation}
 \label{varthetav}
\vartheta_v:=\chi +\sqrt{-1}\partial\bar\partial v+W (\mathrm{d} v),\quad \forall\; v\in C^2(M,\mathbb{R}).
\end{equation}

Given $h\in C^\infty(M,\mathbb{R})$ and $\varphi\in C^\infty(\partial M,\mathbb{R}),$ we study the Dirichlet problem for $\vartheta_u$ given by
\begin{align}
\label{cma}
F (\vartheta_u^{\flat})
=f (\lambda (\vartheta_u^{\flat} ) )=&h,  \quad\text{on}\;\; M, \\
\label{cmabv}
u=&\varphi,  \quad \text{on}\;\; \partial M,
\end{align}
where $W(\mathrm{d}u)$ has a special structure for later applications, and $f$ is a smooth symmetric function of the eigenvalues of $\vartheta_u^{\flat}.$

The Dirichlet problem has been extensively studied since the work of Ivochkina \cite{ivochkina1980} and Caffarelli, Nirenberg  \&  Spruck \cite{cnsacta}; see for example \cite{chouwang2001,guan1998,guan2014,trudinger1995,wang1994,guanli2013advance,guanarxiv2014,lisongying2004,collinspicardarxiv1909,yuanarxiv2020}. We refer to \cite{phongsongsturm2012} for recent progress and further references on this  subject.

We suppose that $f$ is defined on an open symmetric cone $\Gamma\subsetneqq\mathbb{R}^m$, with vertex at the origin $\mathbf{0}$, and that the cone $\Gamma$ satisfies that $\Gamma\supset\Gamma_m:=\left\{(\lambda_1,\dots,\lambda_m)\in\mathbb{R}^m:\;\lambda_j>0,\;1\leq j\leq m\right\}.$
 For instance, we take (see \cite{spruck}) $\Gamma$ as the standard $k$-positive cone $\Gamma_k\subset\mathbb{R}^m$ given by
\begin{equation*}
\Gamma_k:=\{\lambda\in\mathbb{R}^m:\;\sigma_j(\lambda)>0,\;j=1,\cdots,k\},\quad 1\leq k\leq m,
\end{equation*}
where $\sigma_j$ is the $j^{\mathrm{th}}$ elementary symmetric polynomial defined on $\mathbb{R}^m$  given by
\begin{equation*}
\sigma_j(\lambda)=\sum_{1\leq i_1<\cdots<i_j\leq m} \lambda_{i_1}\cdots \lambda_{i_j},\quad
\forall\;\lambda=(\lambda_1,\dots,\lambda_m)\in\mathbb{R}^m,\quad 1\leq j\leq m.
\end{equation*}
The assumptions on the cone $\Gamma$ also yield that (see \cite{cnsacta})
\begin{equation}
\label{cone}
\Gamma\subset\Gamma_1=\left\{\lambda=(\lambda_1,\cdots,\lambda_m)\in\mathbb{R}^m:\,\sum_{j=1}^m\lambda_j>0\right\}.
\end{equation}
In addition, $f$ satisfies
\begin{enumerate}
\item \label{assum1}$f$ is a concave function and $f_j:=\partial f/\partial \lambda_j>0$ for any $j=1,\cdots,m;$
\item \label{assum2}there holds $\sup_{\partial\Gamma}f<\inf_M h,$ where
\begin{equation*}
\sup_{\partial\Gamma}f:=\sup_{\lambda'\in\partial\Gamma}\limsup_{\Gamma\ni\lambda\to\lambda'}f(\lambda);
    \end{equation*}
\item \label{assum3}for any $\sigma$ with $\sigma<\sup_\Gamma f$ and $\lambda\in\Gamma$, we have
\begin{equation*}
    \lim_{t\to+\infty}f(t\lambda)>\sigma.
\end{equation*}
\end{enumerate}
Given Assumption \eqref{assum3}, the concavity of $f$ yields that (see for example \cite{cnsacta})
\begin{equation}
 \label{filambdai}
\sum_{i=1}^mf_i\lambda_i\geq 0.
\end{equation}
 The complex setup is very different from the real analogy (see for example \cite{guanjiao2014}) because of two different types of Hessian of functions. For general $W(\mathrm{d}u),$ it seems hard to control $|\nabla_p\nabla_qu|_g^2$ (see Section \ref{secpreliminary} for the notations) for the second order estimate in the complex setup. To overcome this difficulty, we deal with some special $W(\mathrm{d}u).$ For this aim, we define a bijection
$$
T:\,\mathbb{R}^m\to\mathbb{R}^m,\quad \lambda=(\lambda_1,\cdots,\lambda_m)\mapsto
\mu=(\mu_1,\cdots,\mu_m),
$$
where $\mu_i=T_i(\lambda)=\frac{1}{m-1}\sum_{k\neq i}\lambda_k,\; 1\leq i\leq m.$ If we can rewrite $f$ as
\begin{equation}
\label{stwacta2017equ2.8}
f(\lambda)=\tilde f(T(\lambda)),
\end{equation}
where $\tilde f$ defined on an open symmetric convex cone $\tilde \Gamma  \subsetneq\mathbb{R}^m$ containing $\Gamma_m$ satisfies
\begin{enumerate}
\item  $\tilde f$ is a concave function and $\tilde f_j:=\partial \tilde f/\partial \mu_j>0$ for each $j=1,\cdots,m;$
\item  there holds $\sup_{\partial\tilde\Gamma}\tilde f<\inf_M h,$ where
\begin{equation*}
\sup_{\partial\Gamma}\tilde f:=\sup_{\mu'\in\partial\tilde\Gamma}\limsup_{\tilde\Gamma\ni\mu\to\mu'}\tilde f(\mu);
    \end{equation*}
\item  for any $\sigma$ with $\sigma<\sup_{\tilde\Gamma} \tilde f$ and $\mu\in\tilde\Gamma$, we have
\begin{equation*}
    \lim_{t\to+\infty}\tilde f(t\mu)>\sigma,
\end{equation*}
\end{enumerate}
then it follows from \cite{stw1503} that  $f$ given by \eqref{stwacta2017equ2.8} satisfies Assumption \eqref{assum1}, \eqref{assum2} and \eqref{assum3}.

In this case, we define an operator depending only on the fixed metric $\omega$ by
\begin{equation}
 \label{defnpomega} \tilde{\vartheta}_u := P_\omega(\vartheta_u) = \frac{1}{m-1}\Big( (\mathrm{tr}_\omega\vartheta_u)
  \omega - \vartheta_u\Big).
  \end{equation}
If $\omega$ is the Euclidean metric on $\mathbb{C}^m,$ then the condition $P_{\omega}(\sqrt{-1}\partial\overline{\partial}u) \geq 0$ is equivalent to saying that $u$ is $(m-1)$-plurisubharmonic, in the sense of Harvey \& Lawson \cite{harveylawson2011}.

For any function $u\in C^2(M,\mathbb{R}),$ we define
 \begin{equation*}
 \Delta u:=\frac{m\sqrt{-1}\partial\bar\partial u\wedge\omega^{m-1}}{\omega^m}.
 \end{equation*}
 Then one infers that
\begin{equation}
 \label{eq:u6}\tilde{\vartheta}_u= \tilde{\chi}  + \frac{1}{m-1}\Big(
  (\Delta u)\omega -  \sqrt{-1}\partial\bar\partial u\Big) + Z(\mathrm{d}u),
  \end{equation}
where
\begin{align}
\tilde{\chi}  =& P_\omega(\chi),\nonumber\\
 \label{eq:c1} Z(\mathrm{d}u):=& P_\omega(W(\mathrm{d}u))= \frac{1}{m-1}\Big( (\mathrm{tr}_\omega
  W(\mathrm{d}u))\omega - W(\mathrm{d}u)\Big).
 \end{align}
Note that we can also write $W(\mathrm{d}u)$ explicitly in terms of $Z(\mathrm{d}u)$
 \begin{equation}
 \label{eq:cz} W(\mathrm{d}u) = \big(\mathrm{tr}_{\omega} Z(\mathrm{d}u)\big) \omega - (m-1) Z(\mathrm{d}u).
 \end{equation}
If $f$ can be rewritten as \eqref{stwacta2017equ2.8}, then we assume that  the form $Z(\mathrm{d}u)$ satisfies
 \begin{enumerate}
 \item\label{zu1}
 in the local holomorphic coordinate system $(U;z_1,\dots,z_m)$ near any point,  one has
 \begin{equation}\label{ZZZ}
Z_{i\bar{j}}=Z_{i\bar{j}}^p u_p +\overline{Z_{j\bar{i}}^p}u_{\bar p},
\end{equation}
for some tensor $Z_{i\bar{j}}^p$, independent of $u;$
 \item\label{zu2} in orthonormal holomorphic coordinate system for $\omega$ at any given point, the component
$Z_{i\bar j}$ is independent of $u_{\bar i}$ and $u_j$ (i.e., $Z_{i\bar{j}}^j=0$ for all $i,j$), and
$\nabla_{i}Z_{i\bar{i}}$ is independent of $u_{\bar{i}}$ (i.e.,  $\nabla_{\bar{i}} Z_{i\bar{i}}^i=0$ for all $i$).
 \end{enumerate}
The structure of the gradient term in \eqref{eq:c1} first appears in  Sz\'ekelyhidi, Tosatti \& Weinkove \cite{stw1503} to solve the Gauduchon conjecture on compact Hermitian manifolds without boundary in which the structure plays  a key role in the estimate of the complex Hessian of the solutions to \eqref{cma}.

If $f$ cannot be rewritten as \eqref{stwacta2017equ2.8}, then we consider
another kind of  $W(\mathrm{d}u)$ given by
$$
W_{i\bar j}(\mathrm{d}u)=a_i\nabla_{\bar j}u+a_{\bar j}\nabla_iu,
$$
where $a_{\bar i}=\overline{a_i}.$
This kind of  $W(\mathrm{d}u)$ is considered by Yuan \cite{yuancjm18} for $\dim_{\mathbb{C}}M=2$ and by Tosatti \& Weinkove \cite{twarxiv1906} for general dimension of $M$.
In this case, we have
\begin{equation}
\label{secondwdu}
W_{i\bar j}^p=a_{\bar j}\delta_{ip},\quad Z_{i\bar j}^p=\frac{1}{m-1}\left(a_{\bar p}\delta_{ij}-a_{\bar j}\delta_{ip}\right),
\end{equation}
satisfying Condition \eqref{zu1} and Condition \eqref{zu2} of $Z(u)$. If $\chi$ satisfies $\partial\bar\partial \chi=0,$ then we can choose $a$ to be a holomorphic $1$ form such that $\vartheta_u$ lies in the same Aeppli cohomology class of $\chi $ (see \cite{twarxiv1906} for more details).

\begin{defn}
  \label{csubsol}
Let $(M,J,g)$ be a compact Hermitian manifold with smooth boundary, $\dim_{\mathbb{C}}M=m$, and the canonical complex structure $J$, where $g$ is the Hermitian metric. Then a function $\underline{u}\in C^2(M,\mathbb{R})$ is called a $\mathcal{C}$-subsolution of \eqref{cma} if at each point $\mathbf{p}$, the set
\begin{equation*}
\left(\lambda (\vartheta_{\underline{u}}^\flat)+\Gamma_m\right)\cap \partial \Gamma^{h(\mathbf{p})}
\end{equation*}
is bounded. Here and hereafter,  $\Gamma^\sigma$ is a convex set given by $$\Gamma^\sigma:=\{\lambda\in\Gamma:\;f(\lambda)>\sigma\}.$$
A  function $\underline{u}\in C^2(M,\mathbb{R})$ is called  admissible if $$\lambda (\vartheta_{\underline{u}}^\flat )\in\Gamma,\quad\text{for any}\;x\in M.$$

A  function $\underline{u}\in C^2(M,\mathbb{R})$ is called  an admissible subsolution to the Dirichlet problem \eqref{cma}-\eqref{cmabv} 
if
\begin{align*}
          F(\vartheta_{\underline{u}}^\flat)
           =f(\lambda(\vartheta_{\underline{u}}^\flat))\geq& h,  \;\;\text{on}\;\; M, \\
         \underline{u}=&\varphi,  \;\;\text{on}\; \;\partial M.
\end{align*}
A solution $($resp. admissible subsolution$)$ $u$ $($resp. $\underline{u}$$)$ is called geometric
solution $($resp. geometric admissible subsolution$)$ if $\lambda(\vartheta_{u}^{\flat})\in \overline{\Gamma_m}$
$($resp. $\lambda(\vartheta_{\underline{u}}^\flat)\in \overline{\Gamma_m}$$).$
\end{defn}
Note that any admissible subsolution is a $\mathcal{C}$-subsolution. For the notion of geometric solution see for example Zhang \cite{zhangimrn2010}.
\begin{thm}
 \label{mainthm}
Let $(M,J,g)$ be a compact Hermitian manifold with smooth boundary and $\dim_{\mathbb{C}}M=m,$ where $g$ is the Hermitian metric with respect to the complex structure $J,$ and let $\underline{u}\in C^4(M,\mathbb{R})$ be an admissible  subsolution of the Dirichlet problem \eqref{cma}-\eqref{cmabv} with $f$ satisfying Assumption \eqref{assum1}, \eqref{assum2} and \eqref{assum3}. Then  there holds a priori estimate for the solution $u\in C^4(M,\mathbb{R})$
\begin{equation}
\label{equuc2alpha}
\|u\|_{C^{2}(M,\mathbb{R})}\leq C,
\end{equation}
where $C>0$ is a constant depending on the initial data of $M$, $\underline{u}$ and the upper bound of $|\partial u|_g$.
\end{thm}
In the following, by saying a uniform constant $C$, we mean that the constant $C$ depends only on the background data and the fixed adapted data (see Section \ref{secboundary}) which will change from line to line.

Let us give some examples of the equations in Theorem \ref{mainthm}.
\begin{exa}[generalized Hessian equations]
 If $f=\log\sigma_k$ with $\Gamma=\Gamma_k$ and $1\leq k\leq m,$ then this is the generalized Hessian equation  $($since there exist terms of first order derivative of $u$$)$ defined by
\begin{equation}
\label{hessiancma}
\vartheta_u^k\wedge\omega^{m-k}=e^{h}\omega^m.
\end{equation}
\end{exa}
If $\chi=\omega,\, W(\mathrm{d}u)\equiv0,\,k=m,$ and $M$ is a K\"ahler manifold, then \eqref{hessiancma} is the complex Monge-Amp\`ere equation and the result belongs to  Boucksom \cite{boucksombd} which contains as special cases Caffarelli,  Kohn,  Nirenberg \& Spruck \cite{ckns1985}, and Chen \cite{chen2000} (see also \cite{yau1978,guan1998,donaldson2012,blocki09,blocki2012,ps09}).

If $\partial M=\emptyset,$ and $\chi=\omega,\, W(\mathrm{d}u)\equiv0$ and $M$ is K\"ahler, then the solutions to \eqref{hessiancma}  belong to Yau \cite{yau1978} with $k=m$ to solve the Calabi conjecture, to Dinew \&  Ko{\l}odziej \cite{dkajm} (cf. \cite{houmawu}) for $1<k<m.$

If $\partial M=\emptyset,$ and $\chi=\omega,\, W(\mathrm{d}u)\equiv0$ and $M$ is Hermitian, then the solutions to \eqref{hessiancma} belong to Cherrier \cite{ch1} with $k=m=2$ and  Tosatti \& Weinkove \cite{twjams10} with $k=m$ for general $m$ (cf. \cite{ha,twasian,gl}), and to Sz\'ekelyhidi \cite{gaborjdg} and Zhang \cite{zhangpjm} for $1<k<m$ independently.
\begin{exa}[generalized Hessian quotient equations]
 If $f=\left(\sigma_k/\sigma_\ell\right)^{1/(k-\ell)}$ with $\Gamma=\Gamma_k$ and $1\leq \ell<k\leq m,$ then this is the generalized Hessian quotient equation $($since there exist terms of first order derivative of $u$$)$ defined by
 \begin{equation}
 \label{hessianquotient}
 \vartheta_u^\ell\wedge\omega^{m-\ell}=h \vartheta_u^k\wedge\omega^{m-k},\quad 0\leq \ell<k\leq m,\quad 0<h\in C^\infty(M,\mathbb{R}).
 \end{equation}
\end{exa}
If $\partial M=\emptyset, \,W(\mathrm{d}u)\equiv0$ and $M$ is K\"ahler,  then $h=\frac{\int_M\chi^\ell\wedge\omega^{m-\ell}}{\int_M\chi^k\wedge\omega^{m-k}}$ is a constant and the solution to  \eqref{hessianquotient} is obtained by Song \& Weinkove \cite{songweinkove2008} for $\ell=m-1,\,k=m$ and this solution is the critical point of the $J$-flow introduced by Donaldson \cite{donald1999asian} from the point of view of moment amps, as well as Chen \cite{chen2000imrn,chen2004cag} in his study of the Mabuchi energy, by Fang, Lai \& Ma \cite{fanglaima2011} for $1\leq \ell<m,\,k=m$, and by Sz\'ekelyhidi \cite{gaborjdg} for $1\leq \ell<k\leq m$.

If $M$ is a compact Hermitian manifold and $W(\mathrm{d}u)\equiv0$, then the solution to \eqref{hessianquotient} with $k=m,\,1\leq \ell<m$  belongs to Sun \cite{sun2017} with $\partial M=\emptyset$ and to Guan \& Sun \cite{guansun2015} with $\partial M\not=\emptyset$ (i.e., the Dirichlet problem).
\begin{exa}
The equations with
\begin{equation}
\label{tklcma}
f(\lambda):=\log\frac{\sigma_k(T(\lambda))}{\sigma_\ell(T(\lambda))},\quad \lambda\in T^{-1}(\Gamma_k), \quad 0\leq \ell<k\leq m.
\end{equation}
\end{exa}
This equation is also called $(m-1,m-1)$-Hessian equation (cf.\cite{collinsarxiv1903}).
Let us give some examples for \eqref{tklcma}. Fu, Wang \& Wu \cite{fuwangwu1,fuwangwu2} study the form type Monge-Amp\`ere equations to find the Calabi-Yau type theorem for the balanced metric (i.e., Hermitian metric $\omega$ with $\mathrm{d}(\omega^{m-1})=0$) in the Bott-Chern cohomology group $H_{\mathrm{BC}}^{m-1,m-1}(M,\mathbb{R})$ on compact Hermitian manifolds without boundary. That is, given any representative $\Phi\in c_1(M)\in H^{1,1}_{\mathrm{BC}}(M,\mathbb{R})$ and any balanced metric $\omega$ with $\omega^{m-1}\in H_{\mathrm{BC}}^{m-1,m-1}(M,\mathbb{R}),$ we hope to find a new balanced metric $\omega_u$ such that $\ric(\omega_u)=\Phi$, where
\begin{equation*}
\omega_u^{m-1}=\omega^{m-1}+\sqrt{-1}\partial\bar\partial\left(u\omega^{m-2}\right)>0,\quad u\in C^{\infty}(M,\mathbb{R}).
\end{equation*}
Fu, Wang \& Wu \cite{fuwangwu2} solve  this question on compact K\"ahler manifolds without boundary which admits nonnegative orthogonal
bisectional curvature without boundary, i.e., they try to solve the equation
\begin{equation}
\label{n1pshcma}
\det\left(\omega^{m-1}+\sqrt{-1}\partial\bar\partial u\wedge\eta^{m-2} \right)=e^{F+b}\det\omega^{m-1},\quad \sup_M u=0,\quad (u,b)\in C^\infty(M,\mathbb{R})\times \mathbb{R},
\end{equation}
where $\omega$ is a balanced metric and $\eta$ is a K\"ahler metric with nonnegative orthogonal bisectional curvature.

Tosatti \& Weinkove \cite{twjams} observe  that \eqref{n1pshcma} is related to the $(m-1)$-plurisubharmonic ($(m-1)$-psh for short) function which is introduced by Harvey \& Lawson \cite{harveylawson2011}, and solve a kind of Monge-Amp\`ere type equations. As a corollary, they give an affirmative answer to \eqref{n1pshcma} on compact K\"ahler manifolds without boundary and later they generalize their result on compact Hermitian manifolds without boundary \cite{twcrelle}.
 Note that \eqref{n1pshcma} is in the form of \eqref{cma} with
 $$
 f(\lambda)=\log\sigma_m(T(\lambda)),\quad \lambda\in T^{-1}(\Gamma_m),\quad W(\mathrm{d}u)\equiv0.
 $$
The equation \eqref{tklcma} is proposed by Tosatti, Wang, Weinkove \& Yang \cite{twwycvpde} with $W(\mathrm{d}u)\equiv0$ which is solved by Sz\'ekelyhidi \cite{gaborjdg} with $\ell=0,1\leq k\leq m $ on compact Hermitian manifolds without boundary.

Another kind of equation in the form of \eqref{tklcma} is related to the Gauduchon conjecture \cite{gauduchon2} on compact Hermitian manifolds without boundary. A Hermitian metric $\omega$ is called  \emph{Gauduchon}  (see \cite{gauduchon1}) if
$
\partial\dbar(\alpha^{m-1})=0,
$
and \emph{strongly Gauduchon} (see \cite{popovici}) if
$
\bar\partial(\alpha^{m-1})\;\text{is}\;\partial\textrm-\text{exact}.
$
This conjecture can be deduced from the equation (see \cite{popovici,twcrelle,stw1503})
\begin{equation}
\label{gaucma}
\det \omega_u^{m-1}=e^{F+b}\det\omega^{m-1},\quad (u,b)\in C^\infty(M,\mathbb{R})\times \mathbb{R},
\end{equation}
where
\begin{equation*}
\omega_u^{m-1}=\omega_0^{m-1}+\sqrt{-1}\partial\bar\partial u\wedge\omega^{m-2}+\Re\left(\sqrt{-1}\partial u\wedge \bar\partial(\omega^{m-2})\right)>0,
\end{equation*}
where $\omega$ is the Gauduchon metric and $\omega_0$ is another Hermitian metric.

Sz\'ekelyhidi, Tosatti \& Weinkove \cite{stw1503} solve \eqref{gaucma} on compact Hermitian manifolds without boundary and hence give an affirmative answer to the Gauduchon conjecture.  Note that \eqref{gaucma} is in the form of \eqref{cma} with
$f=\log\sigma_m(T(\lambda))$ on $\Gamma =T^{-1}(\Gamma_m)$ which is exactly solved by Sz\'ekelyhidi, Tosatti \& Weinkove \cite{stw1503}  under  Assumptions \eqref{zu1} and Assumption \eqref{zu2} of  $Z(\mathrm{d}u)$ on compact Hermitian manifolds without boundary (cf.\cite{zhengimrn}). Actually, their method works for $f=\log\sigma_k(T(\lambda))$ on $\Gamma =T^{-1}(\Gamma_k)$ with $1\leq k\leq m.$

In order to solve the Dirichlet problem \eqref{cma}-\eqref{cmabv}, one need deduce a priori $C^2$ estimates up to the boundary. One of the most difficult steps is possibly the second order estimates on the boundary (see \cite{guansun2015}). Other ingredients are the Evans-Krylov theorem (see \cite{twwycvpde}), the Schauder estimates and the continuity method arguments (see for example \cite{gt1998}), which are all well understood and we will omit them.

Given Theorem \ref{mainthm}, it remains to derive an upper bound for the gradients of the solutions.  Compared with the Riemannian setup \cite{guanarxiv2014}, it seems not easy to get the upper bound for the gradients of the solutions under our general setup in complex manifolds. Indeed, deducing the first order estimates for fully nonlinear equations in complex manifolds is a rather challenging and mostly open question (see \cite{guansun2015}). To our knowledge, the existing estimates for the gradients of the solutions to the Dirichlet problem \eqref{cma}-\eqref{cmabv} seem to need that $\Gamma$ is equal to $\Gamma_m$ (see \cite{blocki09,guansun2015,yuancjm2018})  or some special domain which comes from the analytic aspect (see \cite{guanqiuyuan2019}), or that there exists strict subsolution (see \cite{li1994}) or geometric solutions (see \cite{zhangimrn2010}).

In this paper, we 
prove the uniform upper bound of the gradients of the solutions to
the Dirichlet problem \eqref{hessianquotient}-\eqref{cmabv} with $1\leq \ell \leq m-1,\,k=m.$
 \begin{thm}
 \label{thmhessianquotient}
 Let $(M,J,g)$ be a compact Hermitian manifold with smooth boundary and $\dim_{\mathbb{C}}M=m,$ where $g$ is the Hermitian metric with respect to the complex structure $J,$ and $ f=\left(\sigma_m/\sigma_\ell\right)^{1/(m-\ell)}$ with $\Gamma=\Gamma_m$ and $1\leq \ell<m,$
and let  $\underline{u}\in C^\infty(M,\mathbb{R})$ be an admissible  subsolution of the Dirichlet problem \eqref{cma}-\eqref{cmabv}. That is, $\underline{u}$ satisfies
 $$
 h\vartheta_{\underline{u}}^m\geq \vartheta_{\underline{u}}^{m-\ell}\wedge\omega^\ell.
 $$
 Then there exists a unique smooth solution to the Dirichlet problem \eqref{cma}-\eqref{cmabv}.
 \end{thm}
If $W(\mathrm{d}u)\equiv0,$ then Theorem \ref{thmhessianquotient} is obtained by Guan \& Sun \cite{guansun2015}. Here our methods to prove the first and second order estimates are different from the ones in  \cite{guansun2015}.

It follows from \cite{zhangimrn2010} that there hold similar a priori gradient estimates of the geometric solutions to the Dirichlet problem \eqref{cma}-\eqref{cmabv} of generalized Hessian equations with a geometric admissible subsolution. If $\Gamma=\Gamma_m,$ then (admissible) solutions are geometric (admissible) solutions automatically. This, together with our second order estimates, yields that
 \begin{thm}
 \label{thmcmagradient}
 Let $(M,J,g)$ be a compact Hermitian manifold with smooth boundary and $\dim_{\mathbb{C}}M=m,$ where $g$ is the Hermitian metric with respect to the complex structure $J,$ and $ f=\log\sigma_m$ with $\Gamma=\Gamma_m,$
and let  $\underline{u}\in C^\infty(M,\mathbb{R})$ be an admissible  subsolution of the Dirichlet problem \eqref{cma}-\eqref{cmabv}.
 Then there exists a unique smooth solution to the Dirichlet problem \eqref{cma}-\eqref{cmabv} in this case.
 \end{thm}
Since the Monge-Amp\`ere equations here contain gradient terms, this result generalizes Boucksom \cite{boucksombd} for the Dirichlet problem of Monge-Amp\`ere equations on K\"ahler  manifolds which contains as special cases Caffarelli,  Kohn,  Nirenberg \& Spruck \cite{ckns1985}, and Chen \cite{chen2000} (cf. \cite{yau1978,guan1998,donaldson2012,blocki09,blocki2012,ps09,chuarxiv1807}).

Yuan \cite{yuanarxiv2020} also consider the Dirichlet problem in Theorem \ref{thmhessianquotient} with some different conditions and Theorem \ref{thmcmagradient} using different methods for the first and second order estimates.

We also study Equation \eqref{cma} on compact Hermitian manifolds without boundary using our second order interior estimates and some other analytic techniques. We prove
\begin{thm}
\label{thmkellhessianquotient}
Let $(M,J,g)$ be a compact Hermitian manifold without boundary and $\dim_{\mathbb{C}}M=m,$ where $g$ is the Hermitian metric with respect to the complex structure $J.$

If there exists a function $\underline{u}\in C^\infty(M,\mathbb{R})$  such that $\chi_{\underline{u}}$ is a a $k$-positive form $($i.e., the $m$-tuple of their eigenvalues with respect to $\omega$ belong to $\Gamma_{k}$$)$ with
\begin{align}
\label{kellconhermitian1}
 &\frac{\vartheta_{\underline{u}}^\ell\wedge\omega^{m-\ell}}{\vartheta_{\underline{u}}^k\wedge\omega^{m-k}}\geq h,\quad\mbox{on}\quad M,\\
 \label{kellconhermitian2}
 &k h  \vartheta_{\underline{u}}^{k-1}\wedge\omega^{m-k}- \ell\vartheta_{\underline{u}}^{\ell-1}\wedge\omega^{m-\ell}>0
\end{align}
for $(m-1,m-1)$ form,
then   the $(k,\ell)$-Hessian quotient equation \eqref{hessianquotient} $1\leq \ell<k\leq m$ has a smooth solution.

If there exists a function $\underline{u}\in C^\infty(M,\mathbb{R})$  such that $P_\omega(\chi_{\underline{u}})$ is a a $k$-positive form with
\begin{align}
\label{kellm-1conhermitian1}
 &\frac{P_\omega(\chi_{\underline{u}})^\ell\wedge\omega^{m-\ell}}
 {P_\omega(\chi_{\underline{u}})^k\wedge\omega^{m-k}}\geq h,\quad\mbox{on}\quad M,\\
 \label{kellm-1conhermitian2}
 &k h  P_\omega(\chi)^{k-1}\wedge\omega^{m-k}- \ell P_\omega(\chi)^{\ell-1}\wedge\omega^{m-\ell}>0
\end{align}
for $(m-1,m-1)$ form,
then the $(k,\ell)$-$(m-1,m-1)$-Hessian quotient equation \eqref{tklcma} with $1\leq \ell<k\leq m$ has a smooth solution.
\end{thm}
When $(k,\ell)=(m,\ell)$ with $1\leq \ell<m$ and $W_{i\bar j}^p\equiv0,$ Equation \eqref{hessianquotient} is solved by Sun \cite{sun2017}. For general   with $1\leq \ell<k\leq m$, our result partially generalizes \cite[Corollary 3]{gaborjdg} with $W_{i
\bar j}^p\equiv0$ on K\"ahler manifolds without boundary.

Note that Equation \eqref{tklcma} with $(k,\ell)=(k,0)$ and $1\leq k\leq m$ on the compact Hermitian manifolds without boundary is solved by \cite{stw1503} (see    \cite{gaborjdg} for the case  $1\leq k < m$ and $W_{i\bar j}^p\equiv0$).   Equation \eqref{hessiancma} on the compact Hermitian manifolds without boundary is solved by \cite{twarxiv1906,yuancjm18,yuanarxiv2020} which generalized the results mentioned above. Here we give a slightly different second order estimates of these equations motivated by \cite{twarxiv1906} for the   Monge-Amp\`ere equations with gradient terms.

The paper is organized as follows. In section \ref{secpreliminary}, we collect some preliminaries such as Hermitian manifolds with boundary and $\mathcal{C}$-/admissible subsolutions which will be used in the following.  In Section \ref{seczerofirstbd}, we deduce the zero order estimates of the solutions to the Dirichlet problem \eqref{cma}-\eqref{cmabv} on the whole manifold $M$ and the gradient estimates of the solutions to the Dirichlet problem \eqref{cma}-\eqref{cmabv} on the boundary $\partial M$. In Section \ref{sec2orderbd} and Section \ref{sec2orderin}, we give the second  order estimates  of the solutions to the Dirichlet problem \eqref{cma}-\eqref{cmabv} on the boundary and on the whole manifolds respectively, and complete the proof of Theorem \ref{mainthm}. In Section \ref{sec1st}, we obtain the gradient estimates for some special case and prove  Theorem \ref{thmhessianquotient}. In Section \ref{secwubian}, we study Equation \eqref{cma} on compact Hermitian manifolds without boundary and mainly prove Theorem \ref{thmkellhessianquotient}.

\noindent {\bf Acknowledgements}
The authors thank Professor Weisong Dong, Bo Guan, Xinan Ma, and Zhenan Sui for helpful conversation.
The third-named author is the corresponding author and thanks Professor Jean-Pierre Demailly, Valentino Tosatti and Ben Weinkove for their invaluable directions. This paper was almost completed when the third-named author was a post-doc in Institut Fourier supported by the European Research Council (ERC) grant No. 670846 (ALKAGE), hence he thanks the institution for hospitality.

\section{Preliminaries}\label{secpreliminary}
In this section, we collect some preliminaries which will be used in the following (see for example \cite{boucksombd,demaillybook1,guan2014,gaborjdg,trudinger1995}). Throughout the paper, Greek and Latin indices run from $1$ to $2m$ and $1$ to $m$ respectively, and we use subscripts  $x_{\alpha}$ for the partial derivative $\partial/\partial x_{\alpha},$ unless otherwise indicated.
\subsection{The Levi Form of Boundary}Let $\Omega\subset\mathbb{R}^m$ be a bounded open set with $C^k,\,k\in \mathbb{N}^*\cup\{\infty\},$  boundary, i.e., for any $\mathbf{a}\in \partial\Omega$, there exists a  $C^k$ function $\rho$ defined on an open neighborhood $V$ of $\mathbf{a}$ such that
\begin{equation}
\label{localbddefn}
\rho_{\upharpoonright\Omega\cap V}<0,\quad
\rho_{\upharpoonright(\partial\Omega)\cap V}=0,\quad
(\mathrm{d}\rho)_{\upharpoonright(\partial\Omega)\cap V}\not=0.
\end{equation}
Then for another $C^k$ function $\varrho$ defined on $W\ni \mathbf{a}$ with $\varrho_{\upharpoonright W\cap (\partial\Omega)}\equiv0$, there exists a $C^k$ function $\psi$ defined on $W\cap V$ such that
\begin{equation}
\label{varrhofrho}
\varrho=\psi\rho, \quad\text{on}\quad W\cap V,
\end{equation}
and
\begin{equation}
 \label{fbd}
\psi=\frac{\nu\cdot \varrho}{\nu\cdot \rho}=\frac{|\nabla\varrho|}{|\nabla\rho|},\quad\text{on}\quad (\partial\Omega)\cap W\cap V,
\end{equation}
where $\nu$ is the unit outward normal vector on $(\partial\Omega)\cap V\cap W.$


Let $\Omega\subset\mathbb{C}^m$ be a bounded open set with $C^k$ boundary, $ k\in \mathbb{N}^*\cup\{\infty\}$.
Then the holomorphic tangent space to $\partial\Omega$ is by definition the largest complex subspace which
is contained in the tangent space $T_{\partial\Omega}$ to the boundary:
$$
{}^hT_{\partial\Omega}:=T_{\partial\Omega}\cap JT_{\partial\Omega}.
$$
For a local definition function $\rho$ of boundary near $\mathbf{z}$, we claim that ${}^hT_{\partial\Omega,\mathbf{z}}$ is the complex hyperplane in $\mathbb{C}^m$ given by
\begin{equation}
\label{thpartialomega}
{}^hT_{\partial\Omega,\mathbf{z}}:=\left\{\xi\in\mathbb{C}^m:\;\sum\limits_{i=1}^m\frac{\partial\rho}{\partial z_i}(\mathbf{z})\xi_i=0\right\},
\end{equation}
and \eqref{thpartialomega} is independent of the choice of the definition function $\rho$ of the boundary.

Indeed, from the definition of $T_{\partial\Omega}$, it follows that
\begin{equation}
\label{tpartialomega}
T_{\partial\Omega,\mathbf{z}}=\left\{X\in\mathbb{R}^{2m}:\;\left(\mathrm{d}\rho\right)_{\upharpoonright \mathbf{z}}(X)=0\right\}
\end{equation}
and
\begin{align}
\label{jtpartialomega}
JT_{\partial\Omega,\mathbf{z}}
=&\left\{JX\in\mathbb{R}^{2m}:\;\left(\mathrm{d}\rho\right)_{\upharpoonright \mathbf{z}}(X)=0\right\}\\
=&\left\{Y\in\mathbb{R}^{2m}:\;\left(\mathrm{d}\rho\right)_{\upharpoonright \mathbf{z}}(-JY)=0\right\},\nonumber
\end{align}
where we extend $J$ to the $p$ form $\vartheta$ by
\begin{equation*}
(J\vartheta)(\cdot,\cdots,\cdot)=(-1)^p\vartheta(J\cdot,\cdots,J\cdot).
\end{equation*}
From \eqref{tpartialomega} and \eqref{jtpartialomega}, we get
\begin{equation}
\label{tcapjt}
T_{\partial\Omega,\mathbf{z}}\cap JT_{\partial\Omega,\mathbf{z}}
=\left\{X\in\mathbb{R}^{2m}:\;\left(\mathrm{d}\rho\right)_{\upharpoonright \mathbf{z}}(X-\sqrt{-1}JX)=0\right\}.
\end{equation}
Since $J(X-\sqrt{-1}JX)=\sqrt{-1}(X-\sqrt{-1}JX)$, equality \eqref{thpartialomega} follows from \eqref{tcapjt}.

Now we assume that $\varrho$ is another definition function of $\partial\Omega$ near $\mathbf{z}$.
Since $\mathrm{d}\varrho(\mathbf{z})\not=0$, it follows from \eqref{varrhofrho} and \eqref{fbd} that $\psi(\mathbf{z})\not=0.$
This yields that
\begin{equation*}
{}^hT_{\partial\Omega,\mathbf{z}}=\left\{\xi\in\mathbb{C}^m:\;\left(\mathrm{d}\varrho\right)_{\upharpoonright \mathbf{z}}(\xi)=0\right\}
=\left\{\xi\in\mathbb{C}^m:\:\left(\mathrm{d}\rho\right)_{\upharpoonright \mathbf{z}}(\xi)=0\right\},
\end{equation*}
as desired.

The Levi form on ${}^hT_{\partial\Omega}$ is defined at every point $\mathbf{z}\in\partial\Omega$ by
\begin{align*}
L_{\partial\Omega,\mathbf{z}}(\xi,\eta):=&\left(\frac{\mathrm{d}J\mathrm{d}\rho}{2|\mathrm{d} \rho|}\right)_{\upharpoonright \partial\Omega,\mathbf{z}}(\xi,\bar\eta)\\
=&\frac{1}{|\mathrm{d} \rho|(\mathbf{z})}\sum\limits_{i,j=1}^m
\frac{\partial^2\rho}{\partial z_i\partial\overline{z}_j}(\mathbf{z})\xi_i\overline{\eta}_j,
\quad\xi,\eta\in {}^hT_{\partial\Omega,\mathbf{z}}.\nonumber
\end{align*}
The Levi form does not depend on the particular choice of $\rho$. Indeed, this follows from  \eqref{localbddefn}, \eqref{varrhofrho} and \eqref{fbd} directly.

\begin{lem}\label{lemomegabd}
Let $\Omega\subset\mathbb{C}^m$ be a bounded open set with $C^2$ boundary.
\begin{enumerate}
\item\label{ex1-8.12a}
Let $\mathbf{a}\in\partial\Omega$ be a given point. Let $e_m$ be the outward normal vector to $T_{\partial\Omega,\mathbf{a}},\,(e_1,\dots,e_{m-1})$ an orthonormal basis of ${}^hT_{\partial\Omega,\mathbf{a}}$ in which the Levi form is diagonal and $(z_1,\dots,z_m)$ the associated
linear coordinates centered at $\mathbf{a}.$ Then there is a neighborhood $V$ of $\mathbf{a}$ such that $\partial\Omega\cap V$ is
the graph $\Re z_m=\varphi (z_1, \dots, z_{m-1}, \Im z_m)$ of a function $\varphi$ such that $\varphi(\mathbf{z})=O(|\mathbf{z}|^2)$ and the matrix
$\left(\partial^2\varphi/\partial z_i\partial\bar z_j(\mathbf{a})\right)_{1\leq i,j\leq m-1}=\mathrm{diag}\{\lambda_1,\dots,\lambda_{m-1}\},$ where $\lambda_1,\dots,\lambda_{m-1}$ are the eigenvalues of the Levi form $L_{\partial\Omega,\mathbf{a}}.$
\item\label{ex1-8.12b}
There exists  a local coordinate given by
\begin{align*}
z_m=&w_m+\frac{1}{2}\sum_{1\leq j,\,k\leq m}d_{jk}w_jw_k, \quad \text{with}\quad d_{jk}=d_{kj},\\
z_i=&w_i,\quad 1\leq i\leq m-1,
\end{align*}
on a neighborhood $V'$   of $\mathbf{a} =\mathbf{0}$ such that
 \begin{equation*}
 \Omega\cap V'=V'\cap\left\{-\Re w_m+\sum_{1\leq j\leq m}\lambda_j|w_j|^2+O(|\mathbf{w}|^3)<0\right\},
 \end{equation*}
where $\lambda_1,\cdots,\lambda_{m-1}$ are the eigenvalues of the Levi form $L_{\partial\Omega,\mathbf{0}}$ and $\lambda_m\in\mathbb{R}$ can be assigned to any given value by a suitable choice of the coordinates.
\end{enumerate}
\end{lem}
\begin{proof}See, e.g., \cite[Exercise \rm{I}-8.12]{demaillybook1}.
\end{proof}
\begin{lem}
\label{lembou7.12}
Let $r$ be a smooth function defined near $\mathbf{0}\in\mathbb{R}^m$ with coordinates $x_1,\cdots,x_m$ such that $r_{x_m}(\mathbf{0})=-1$ and $r(\mathbf{0})=r_{x_{i}}(\mathbf{0})=0,\,1\leq i\leq m-1$, and let $N$ denote the hypersurface defined by $\{r=0\}$ which is smooth near $0$. Then $(x_{i})_{\upharpoonright N},\,1\leq i\leq m-1,$ are local coordinates of $N$ and for any smooth function $v$ near $\mathbf{0}$ there holds
\begin{align*}
  \partial_{(x_{i})_{\upharpoonright N}}\left(v_{\upharpoonright N}\right)(\mathbf{0})
  =& v_{x_i}(\mathbf{0})+ v_{x_m}(\mathbf{0}) r_{x_i}(\mathbf{0}),\\
 \partial_{(x_{i})_{\upharpoonright N}} \partial_{(x_j)_{\upharpoonright N}}\left(v_{\upharpoonright N}\right)(\mathbf{0})
=& v_{x_ix_j}(\mathbf{0})+ v_{x_m}(\mathbf{0}) r_{x_ix_j}(\mathbf{0}).
\end{align*}
\end{lem}
\begin{proof}
See, e.g.,  \cite[Lemma 7.2]{boucksombd}.
\end{proof}
\subsection{Complex Manifolds with Boundary}\label{secboundary} In this subsection, we set $B_R(\mathbf{z})\subset \mathbb{C}^m$ $($resp. $\bar B_R(\mathbf{z})\subset \mathbb{C}^m $$)$  the open $($resp. closed$)$ ball centered at the point $\mathbf{z}$ with radius $R.$
A complex manifold with smooth boundary $M$ of $\dim_{\mathbb{C}}M=m$ is a smooth manifold with boundary
equipped with a system of coordinate patches
\begin{equation*}
 \phi_j:\;U_j\to \left\{\mathbf{z}\in B_R(\mathbf{0}):\; \,r_j(\mathbf{z})\leq0\right\},\quad j\in\text{the index set}\;\mathcal{J}
\end{equation*}
such that
$
\phi_j\circ \phi_i^{-1}
$ is a biholomorphic on $\phi_{i}(U_i\cap U_j)\cap \left\{r_i<0\right\},$ where $r_j$'s are the local definition functions, i.e., smooth functions defined on the neighborhood of  $\bar B_R(\mathbf{0})$ with $\mathrm{d}r_j\not=0$ along $\{r_j=0\}.$

The holomorphic tangent bundle ${}^h T_{\partial M}$ of $\partial M$ is defined as the largest complex subbundle of $T_M$ which is contained $T_{\partial M},$ i.e.,
\begin{equation*}
 {}^h T_{\partial M}:=T_{\partial M}\cap J T_{\partial M},
\end{equation*}
where $J:\;T_M\to T_M$ is the complex structure.

Let $(M,J,g)$ be a complex manifold with smooth boundary of $\dim_{\mathbb{C}}M=m,$ and let $\overrightarrow{\nu}$ be the unit outward  normal vector field on $\partial M.$ Then the Levi form $L_{\partial M,\overrightarrow{\nu}}$ of $\partial M$ with respect to $\overrightarrow{\nu}$ is locally given by
\begin{align}
\label{leviformm}
L_{\partial M,\overrightarrow{\nu}}(\xi,\eta):=&\left(\frac{\mathrm{d}J\mathrm{d}(r_j\circ \phi_j)}{2|\mathrm{d} (r_j\circ \phi_j)|_g}\right)_{\upharpoonright \partial M,\mathbf{z}}(\xi,\bar\eta)\\
=&\frac{1}{2\left(\overrightarrow{\nu}(r_j\circ \phi_j)\right)(\mathbf{z})}\left(\mathrm{d}J\mathrm{d}(r_j\circ \phi_j)\right)(\xi,\bar\eta),
\quad\xi,\eta\in {}^hT_{\partial M,\mathbf{z}},\quad\mathbf{z}\in U_j\cap \partial M.\nonumber
\end{align}
It is easy to check that the expression in \eqref{leviformm} is well defined. The boundary $\partial M$ is called
weakly (resp. strictly) pseudoconcave if $L_{\partial M,\overrightarrow{\nu}}\leq 0$ (resp. $<0$), and  weakly (resp. strictly) pseudoconvex if $L_{\partial M,\overrightarrow{\nu}}\geq 0$ (resp. $>0$).

Throughout this paper, we fix a covering of $\partial M$ consisting of finite open sets $\{U_{i}\}_{i\in\mathcal{J}}$ such that
\begin{equation*}
\phi_i:\;U_{i}\to B_{2,i}(\mathbf{0}):=\left\{\mathbf{z}\in B_2(\mathbf{0}):\;r_i(\mathbf{z})\leq 0\right\}
\end{equation*}
is diffeomorphism,
\begin{equation*}
\phi_{j}\circ\phi_{i}^{-1}:\;\phi_i(U_i\cap U_j)\cap\left\{r_i<0\right\}
\to \phi_j(U_i\cap U_j)\cap\left\{r_j<0\right\}
\end{equation*}
is biholomorphic, and that the family of finite open subsets
\begin{equation*}
\left\{V_i:=\phi_{i}^{-1}(B_{1,i}(\mathbf{0})),\,i\in\mathcal{J}\right\}
\end{equation*}
still covers the boundary $\partial M$.

We denote the local coordinates and the definition function of $\partial M$ on $U_i$ by
\begin{equation*}
\mathbf{w}^{(i)}:=\mathbf{w}\circ\phi_{i},\quad \rho_i:=r_i\circ\phi_{i},\quad \forall\;i\in\mathcal{J}.
\end{equation*}
Fix a index $i\in\mathcal{J}$. For any point $\mathbf{p}\in V_i\cap \partial M$, there exists a biholomorphic map (see for example \cite[Formula (1.6)]{chengyau1980})
\begin{equation*}
\psi_i:\;B_{2}(\mathbf{0})\to B_{2}(\mathbf{0}),\quad \mathbf{w}\mapsto \mathbf{z}
\end{equation*}
such that $\psi_i\circ \phi_i(\mathbf{p})=\mathbf{0},$ and that
\begin{equation}
\label{rhocoor}
\rho_i\circ\phi_i^{-1}\circ\psi_i^{-1}(\mathbf{z})=r_i\circ\psi_i^{-1}(\mathbf{z})=
-\Re z_m+ \sum_{1\leq j\leq m}\lambda_j|z_j|^2  +O(|\mathbf{z}|^3),
\end{equation}
where $\lambda_1,\cdots,\lambda_{m-1}$ are the eigenvalues of the Levi form $L_{\partial M,\nu}$ with respect to $g$ and $\lambda_m\in\mathbb{R}$ can be assigned to any given value by a suitable choice of the coordinates.

In the later use, we will not distinguish $r_i,$ $r_i\circ\phi_i$ and $r_i \circ\psi_i^{-1}$ for convenience. We will say that for any point $\mathbf{0}\in \partial M,$ we want to study our questions on the adapted data $(B,r,\mathbf{z})$ where $B=B_{1}(\mathbf{0}),$ $r$ is a definition function of $\partial M,$ and $\mathbf{z}$ is
the coordinate on $B$ centered at $\mathbf{0}$ such that $r$ satisfies \eqref{rhocoor}.

Let $e_1,\cdots,e_m$ be a basis of local frame fields of $T^{1,0}_M$ with dual $\theta^1,\cdots, \theta^m $ which are $(1,0)$ forms. Then we extend the Riemannian metric $g$ to $T_{M}\otimes_{\mathbb{R}}\mathbb{C}$ to obtain
\begin{equation*}
 \omega=\sqrt{-1}\sum_{i,j=1}^mg_{i\bar j}\theta^i\wedge\bar \theta^j,\quad g_{i\bar j}:=g(e_i,\bar e_j),\quad\overline{g_{i\bar j}}=g_{j\bar i}.
\end{equation*}
Let us denote by $\nabla$ the Chern connection of the Hermitian metric $g.$ Then we fix some notations.
\begin{align*}
& \nabla_{i}:=\nabla_{e_i},\quad \nabla_{\bar j}:=\nabla_{\bar e_j},\quad  \nabla_ie_j=:\Gamma_{ij}^ke_k,\\
&\nabla_ie_j-\nabla_je_i-[e_i,e_j]=T(e_i,e_j)=:T_{ij}^ke_k,\\
&\left(\nabla_i\nabla_{\bar j}-\nabla_{\bar j}\nabla_i-\nabla_{[e_i,\bar e_j]}\right)e_k=:R_{i\bar jk}{}^\ell e_\ell,\quad R_{i\bar j k\bar \ell}:=R_{i\bar jk}{}^pg_{p\bar\ell}.
\end{align*}
Denote by $\Delta_g$ the Laplace-Beltrami operator of the Riemannian metric $g$. There holds (see for example \cite[Lemma 3.2]{tosatticag2007})
\begin{align}
\label{2laplace}
\Delta_g \varphi=2\Delta \varphi+\tau(\md \varphi),\quad \forall\;\varphi\in C^2(M,\mathbb{R}),
\end{align}
where
$$
\tau(\md \varphi)=2\Re\left(T_{pj}^jg^{\overline{q}p}\nabla_{\bar q}\varphi\right).
$$
Given a function $u\in C^4(M,\mathbb{R}),$ it follows from the Ricci identity and the first Bianchi identity that (see for example \cite{twcrelle})
\begin{align}
 \label{ricciid0}
 \nabla_i\nabla_{j}u
 =&\nabla_{j}\nabla_iu-T_{ij}^k\nabla_ku,\\
 \label{ricciid1}\nabla_i\nabla_{\bar j}\nabla_k u
=&\nabla_k\nabla_i\nabla_{\bar j}u-T_{ik}^p\nabla_p\nabla_{\bar j}u,\\
\label{ricciid2}
\nabla_i\nabla_{\bar j}\nabla_{\bar k}u
=&\nabla_{\bar k}\nabla_i\nabla_{\bar j}u
-\overline{T_{jk}^q}\nabla_{\bar q}\nabla_iu
+\overline{R_{j\bar ik}{}^q}\nabla_{\bar q}u,\\
\label{ricciid2times}
\nabla_{\bar \ell}\nabla_k\nabla_{\bar j}\nabla_iu
=&\nabla_{\bar j}\nabla_i\nabla_{\bar \ell}\nabla_ku
+R_{k\bar\ell i}{}^p\nabla_{\bar j}\nabla_pu
-R_{i\bar j k}{}^p\nabla_{\bar\ell}\nabla_pu\\
&-T_{ki}^p\nabla_{\bar\ell}\nabla_{\bar j}\nabla_p u
-\overline{T_{\ell j}^q}\nabla_{k}\nabla_{\bar q}\nabla_i u
-T_{ki}^p\overline{T_{\ell j}^q}\nabla_{\bar q}\nabla_p u.\nonumber
\end{align}
\subsection{Subsolutions}\label{secsubsolution}
In this subsection, we recall some preliminaries from \cite{gaborjdg} (cf. \cite{trudinger1995,guan2014}). Given any $\sigma\in\left(\sup_{\partial\Gamma}f,\sup_\Gamma f\right)$, the set $\Gamma^\sigma=\{\lambda\in\Gamma:\;f(\lambda)>\sigma\}$ is open and convex 
and $\partial\Gamma^\sigma=f^{-1}(\sigma)$ is a smooth hypersurface. 
We denote, by $\mathbf{n}(\lambda),$  the inward pointing unit normal vector, i.e.,
\begin{equation*}
\mathbf{n}(\lambda):=\frac{\nabla f}{|\nabla f|}(\lambda),\quad \forall \;\lambda\in\partial\Gamma^\sigma.
\end{equation*}
We set $\mathcal{F}(\lambda):=\sum_{k=1}^mf_k(\lambda).$ The Cauchy-Schwarz inequality yields that $|\nabla f|\leq \mathcal{F}\leq \sqrt{m}|\nabla f|$.

Following \cite{trudinger1995}, we set
\begin{equation*}
\Gamma_\infty:=\left\{(\lambda_1,\cdots,\lambda_{m-1}):\;(\lambda_1,\cdots,\lambda_m)\in\Gamma\;\text{for some}\;\lambda_m\right\}.
\end{equation*}
For any $\mu\in\mathbb{R}^m$, the set $\left(\mu+\Gamma_m\right)\cap \partial\Gamma^{\sigma}$ is bounded, if and only if
\begin{equation*}
\lim_{t\to+\infty}f(\mu+t\mathbf{e}_i)>\sigma,\quad\forall\;1\leq i\leq m,
\end{equation*}
where $\mathbf{e}_i$ denotes the $i^{\mathrm{th}}$ standard basis vector.  This limit is well defined as long as any $(m-1)$ tuple $\mu'$ in $\mu$ satisfies $\mu'\in\Gamma_\infty$, i.e., on the set $\tilde \Gamma$ defined by
\begin{equation*}
\tilde\Gamma:=\left\{\mu\in\mathbb{R}^m:\;\exists\;t>0\;\text{such that}\;\mu+t\mathbf{e}_i\in\Gamma\;\forall\;i\right\}.
\end{equation*}
For any $\lambda'=(\lambda_1,\cdots,\lambda_{m-1})\in\Gamma_\infty,$  the concavity of $f$ implies that the limit
\begin{equation*}
\lim_{\lambda_m\to+\infty}f(\lambda_1,\cdots,\lambda_{m-1},\lambda_m)
\end{equation*}
is either finite for all $\lambda'$ or infinite for all $\lambda'$ (see \cite{trudinger1995}).

If the limit is infinite, then $\left(\mu+\Gamma_m\right)\cap \partial\Gamma^\sigma$ is bounded for all $\sigma$ and $\mu\in\tilde\Gamma$. In particular, any admissible $\underline{u}$ is a  $\mathcal{C}$-subsolution, not vice versa.

If the limit is finite, then we define the function $f_\infty$ on $\Gamma_\infty$ by
\begin{equation*}
f_\infty(\lambda_1,\cdots,\lambda_{m-1})=\lim_{t\to+\infty}f(\lambda_1,\cdots,\lambda_{m-1},t).
\end{equation*}
In this case, for $\mu\in\tilde\Gamma,$ the set $\left(\mu+\Gamma_m\right)\cap\partial\Gamma^\sigma$ is bounded if and only if $f_\infty(\mu')>\sigma$, where $\mu'\in\Gamma_\infty$ is any $(m-1)$ tuple of entries of $\mu.$
\begin{prop}[Sz\'ekelyhidi \cite{gaborjdg}]
\label{gaborprop5}
Given $\delta,\,R>0,$ if $\mu\in\mathbb{R}^m$ such that
\begin{equation*}
\left(\mu-2\delta\mathbf{1}+\Gamma_m\right)\cap \partial\Gamma^\sigma\subset B_R(\mathbf{0}),
\end{equation*}
where $B_R(\mathbf{0})\subset\mathbb{R}^m$ is the ball centered at $\mathbf{0}$ with radius $R$,
then there exists a constant $\kappa>0$ depending only on $\delta$ and $\mathbf{n}$ on $\partial\Gamma^\sigma$ such that for any $\lambda\in\partial\Gamma^\sigma$ with $|\lambda|>R$, there holds either
\begin{equation*}
\sum_{j=1}^mf_j(\lambda)(\mu_j-\lambda_j)>\kappa\mathcal{F}(\lambda),
\end{equation*}
or
\begin{equation*}
f_i(\lambda)>\kappa \mathcal{F}(\lambda),\quad \forall\;1\leq i\leq m.
\end{equation*}
\end{prop}

\begin{lem}[Sz\'ekelyhidi \cite{gaborjdg}]
\label{lem9}
Let $f$ be a smooth symmetric function defined on $\Gamma$ satisfying Assumption \eqref{assum1}, \eqref{assum2} and \eqref{assum3} in the introduction. Then  $\forall\;\sigma\in(\sup_{\partial\Gamma}f,\sup_\Gamma f)$, one infers that
\begin{enumerate}
\item \label{gaborlem91}there exists an $N>0$ depending only on $\sigma$ such that $(\Gamma+N\mathbf{1})\subset \Gamma^\sigma;$
\item \label{gaborlem92}there is a $\tau>0$ depending only on $\sigma$ such that $\mathcal{F}(\lambda)>\tau,\;\forall\,\lambda\in\partial\Gamma^\sigma.$
\end{enumerate}
\end{lem}
We need some formulae for the derivatives of eigenvalues (see for example \cite{spruck}).
\begin{lem}[Spruck \cite{spruck}]
\label{lemspruck}
The first and second order derivatives of the eigenvalue $\lambda_i$ at a diagonal matrix $(A_{ij})$ $($consider it as a Hermitian matrix$)$ with distinct eigenvalues are
\begin{align}
\label{eigenvalue1st}\lambda_{i}^{pq}=&\delta_{pi}\delta_{qi},\\
\label{eigenvalue2nd}\lambda_{i}^{pq,rs}=&(1-\delta_{ip})\frac{\delta_{iq}\delta_{ir}\delta_{ps}}{\lambda_i-\lambda_p}
+(1-\delta_{ir})\frac{\delta_{is}\delta_{ip}\delta_{rq}}{\lambda_i-\lambda_r},
\end{align}
where
$$
\lambda_{i}^{pq}=\frac{\partial\lambda_i}{\partial  A_{pq}}, \quad
\lambda_{i}^{pq,rs}=\frac{\partial^2\lambda_i}{\partial  A_{pq}\partial  A_{rs}}.
$$
If we consider $A=(A_{ij})$ as a symmetric matrix, then the right side of \eqref{eigenvalue2nd} should be multiplied by $2.$
\end{lem}
\begin{lem}[Gerhardt \cite{gerhardt}]
\label{lemgerhardt}
If $F(A)=f(\lambda_1,\cdots, \lambda_m)$ in terms of a smooth symmetric funtion of the eigenvalues, then at a diagonal matrix $(A_{ij})$ $($consider it as a Hermitian matrix$)$ with distinct eigenvalues there hold
\begin{align}
\label{f1stdaoshu}F^{ij}=&\delta_{ij} f_i,\\
\label{f2nddaoshu}F^{ij,rs}=&f_{ir}\delta_{ij}\delta_{rs}+\frac{f_i-f_j}{\lambda_i-\lambda_j}(1-\delta_{ij})\delta_{is}\delta_{jr},
\end{align}
where
$$
F^{ij}=\frac{\partial F}{\partial  A_{ij}}, \quad F^{pq,rs}=\frac{\partial^2F}{\partial  A_{ij}\partial  A_{rs}}.
$$
If we consider $A=(A_{ij})$ as a symmetric matrix, then the second term in the right side of \eqref{f2nddaoshu} should be multiplied by $2.$
\end{lem}
These formulae make sense even if the eigenvalues are not distinct. Indeed, if $f$ is smooth and symmetric, then $f$ is a smooth function of elementary symmetric polynomials which are smooth on the space of matrices by Vieta's formulas and hence $F$ is a smooth function on the space of matrices.
In particular, we have $f_i\longrightarrow f_j$ as $\lambda_i\longrightarrow \lambda_j$. If $f$ is concave and symmetric, then we have that $\frac{f_i- f_j}{\lambda_i-\lambda_j}\leq 0$ (see \cite{spruck} or \cite[Lemma 2]{eh89}). In particular, if $\lambda_i\leq \lambda_j$, then we have $f_i\geq f_j$.

In the local coordinates $(U;z_1,\cdots,z_m),$ let $A=A_i{}^j\mathrm{d}z_i\otimes \partial_j\in \mathrm{End}(T^{1,0}M)$ be a Hermitian map with respect to $g.$ Then we set $A_{i\bar j}:=A_i{}^qg_{q\bar j}$ satisfying $\overline{A_{i\bar j}}=A_{j\bar i}.$
We define a strictly elliptic operator  $L$ by
\begin{align}
\label{defnl}
L(u)=&F^{ij}g^{\overline{q}j}\left(\partial_i\partial_{\overline{q}}u+W_{i\bar q}(\mathrm{d} u)\right)\\
=&F^{ij}g^{\overline{q}j}\left(\nabla_i\nabla_{\overline{q}}u
+W_{i\bar q}^p(\nabla_p u)+\overline{W_{q\bar i}^p}(\nabla_{\bar p} u)\right),\quad \forall\;u\in C^2(M,\mathbb{R}).\nonumber
\end{align}
It is easy to see that $L$ is the linearized operator of $F$ given in \eqref{cma}.
We also use the notation
$F^{i\bar q}:=F^{i j}g^{\bar q j}$ such that $(F^{i\bar q})$ is a positive definite Hermitian matrix.  Indeed, without loss of generality, we set $\lambda_1(A)>\cdots>\lambda_m(A)$ and the general case follows from the continuity arguments.
Let $\xi_p=\xi_p{}^q\partial_q$ be the unit complex eigenvector of $A$ with eigenvalue $\lambda_p,$ i.e.,
we have
\begin{align}
\label{gij}
\sum_{i,j=1}^mg_{i\bar j}\xi_p{}^i\overline{\xi_q{}^j}=&\delta_{pq},\quad p,q=1,\cdots,m,\\
\label{duijiaoa1}
\sum_{r=1}^m\xi_p{}^rA_{r}{}^s=&\lambda_p\xi_{p}{}^s,\quad p,q=1,\cdots,m.
\end{align}
It follows from \eqref{gij} that
\begin{equation}
\label{gbarellk}
g^{\bar \ell k}=\sum_{p=1}^m\overline{\xi_p{}^\ell}\xi_p{}^k,\quad g_{k\bar \ell}=\sum_{p=1}^m\zeta_k{}^p\overline{\zeta_\ell{}^p},
\end{equation}
where $\zeta=(\zeta_i{}^j)$  is the inverse matrix of $\xi,$ i.e., there holds
\begin{equation}
\label{inversxizeta}
\sum_{q=1}^m\xi_i{}^q\zeta_q{}^j=\delta_i{}^j,\quad\quad 1\leq i,j\leq m.
\end{equation}
From \eqref{duijiaoa1} and \eqref{gbarellk}, we get
\begin{align}
\label{duijiaoa2}
 &\sum_{r,s=1}^m\xi_p{}^rA_{r}{}^s\zeta_s{}^q
=\sum_{r,s=1}^m\xi_p{}^rA_{r\bar s}\overline{\xi_{q}{}^s}
=\delta_{pq}\lambda_p,\\
\label{aibarj}
&A_{i\bar j}=\zeta_i{}^p\lambda_p\overline{\zeta_j{}^p}.
\end{align}
We observe from \eqref{aibarj} that $\lambda=(\lambda_1,\cdots,\lambda_m)$ are  the eigenvalues of the Hermitian metric $(A_{i\bar j})$ if $(\xi_i{}^j)$ satisfies $\sum_p\xi_i{}^p\overline{\xi_j{}^p}=\delta_{ij}.$  Using this observation and Lemma \ref{lemspruck}, we can calculate the derivatives of the eigenvalues of a map in $\mathrm{End}(T^{1,0}M).$
Indeed, we set $\theta^i=\zeta_k{}^i\mathrm{d}z^k.$ Then from \eqref{duijiaoa1} one can deduce that
\begin{equation*}
 A=A_k{}^\ell\mathrm{d}z^k\otimes \partial_\ell
=\zeta_k{}^p\lambda_p\xi_p{}^\ell \mathrm{d}z^k\otimes \partial_\ell
=\lambda_p\theta^p\otimes\xi_p
= :\tilde{ A}_i{}^j\theta^i\otimes \xi_j,
\end{equation*}
with $(\tilde A_i{}^j)=(\tilde A_{i\bar j})=(\lambda_i\delta_{ij})$ is a Hermitian matrix and
\begin{equation}
\label{tildeaa}
 \tilde{ A}_i{}^j=\xi_i{}^kA_k{}^\ell \zeta_{\ell}{}^j.
\end{equation}
If $\lambda_k$ is smooth at $(A_i{}^j)$, then it follows from \eqref{eigenvalue1st}, \eqref{eigenvalue2nd}, \eqref{gbarellk}, \eqref{inversxizeta} and \eqref{tildeaa} that
\begin{align}
\label{lambdakij}
\frac{\partial \lambda_k}{\partial A_{i}{}^j}
=&\frac{\partial \lambda_i}{\partial \tilde A_{p}{}^q}
\frac{\partial \tilde A_{p}{}^q}{\partial   A_{i}{}^j}
=\xi_k{}^i\zeta_{j}{}^k=\xi_k{}^ig_{j\bar q}\overline{\xi_k{}^q},\\
\label{lambdakijpq}
\lambda_{k}^{ij,pq}=&
\frac{\partial^2\lambda_k}{\partial \tilde A_r{}^s\partial \tilde A_u{}^v}
\frac{\partial \tilde A_r{}^s}{\partial A_i{}^j}
\frac{\partial \tilde A_u{}^v}{\partial A_p{}^q}\\
=&\sum_{r\not=k}\frac{\xi_r{}^i\zeta_j{}^k\xi_k{}^p\zeta_q{}^r+\xi_k{}^i\zeta_j{}^r\xi_r{}^p\zeta_q{}^k}
{\lambda_k-\lambda_r}\nonumber\\
=&\sum_{r\not=k}\frac{\xi_r{}^ig_{j\bar q}\overline{\xi_k{}^q}\xi_k{}^pg_{q\bar u}\overline{\xi_r{}^u} +\xi_k{}^ig_{j\bar v}\overline{\xi_r{}^v}\xi_r{}^pg_{q\bar s}\overline{\xi_k{}^s}}
{\lambda_k-\lambda_r}.\nonumber
\end{align}
If $(M,g)$ is a Riemannian manifold and $A\in \mathrm{End}(TM),$  then using the local coordinate $(U;x_1,\cdots,x_m),$ we write $A=A_i{}^j\mathrm{d}x^i\otimes \partial_{x_j}$ and $g=g_{ij}\mathrm{d}x_i\otimes \mathrm{d}x_j.$ If $A$ is symmetric with respect to $g,$ then we have
 $A_{ij}:=A_i{}^pg_{pj}$ satisfying $A_{ij}=A_{ji}.$ Let $\xi_i=\xi_i{}^j\partial_{x_j}$ be the eigenvector of $A$ with respect to $\lambda_i,$ i.e.,
\begin{equation*}
 A\xi_i=\lambda_i\xi_i,\quad 1\leq i\leq m.
\end{equation*}
A similar argument yields that $\lambda_i$'s are the eigenvalues of $(A_{ij})$ if $(\xi_i{}^j)$ is orthonormal matrix (see \cite[Lemma 5.2]{ctwjems}).

If all the eigenvalues are smooth at $(A_i{}^j),$ then we can obtain  from \eqref{gbarellk}, \eqref{lambdakij} and \eqref{lambdakijpq} that
\begin{equation}
\label{fijexpression}
F^{ij}
=\sum_{p=1}^m\zeta_j{}^pf_p\xi_p{}^i
=\sum_{p,r=1}^mf_p\xi_p{}^ig_{j\bar r}\overline{\xi_p{}^r},\quad
F^{i\bar q}
=\sum_{p=1}^mf_p\xi_p{}^i\overline{\xi_p{}^q},
\end{equation}
and
\begin{align*}
 F^{ij,pq}
=&f_{k\ell} \xi_k{}^i\zeta_j{}^k\xi_\ell{}^p\zeta_q{}^\ell
+f_{k}\sum_{r\not=k}\frac{\xi_r{}^i\zeta_j{}^k\xi_k{}^p\zeta_q{}^r+\xi_k{}^i\zeta_j{}^r\xi_r{}^p\zeta_q{}^k}
{\lambda_k-\lambda_r}\\
=&f_{k\ell}\xi_k{}^ig_{j\bar u}\overline{\xi_k{}^u}\xi_\ell{}^pg_{q\bar v}\overline{\xi_\ell{}^v}
+f_k\sum_{r\not=k}\frac{\xi_r{}^ig_{j\bar q}\overline{\xi_k{}^q}\xi_k{}^pg_{q\bar u}\overline{\xi_r{}^u} +\xi_k{}^ig_{j\bar v}\overline{\xi_r{}^v}\xi_r{}^pg_{q\bar s}\overline{\xi_k{}^s}}
{\lambda_k-\lambda_r}.\nonumber
\end{align*}
These formulae are sight generalization of Lemma \ref{lemgerhardt}.
 Thanks to \eqref{duijiaoa2} and \eqref{fijexpression}, we deduce that
the matrices $(F^{ij})$ and $(A_i{}^j)$
(hence $(F^{i\bar j})$ and $(A_{i\bar j})$)
can be diagonalized at the same time with
\begin{align}
\label{ijfijaij}
 \sum_{i,j=1}^mF^{ij}A_{i}{}^j
=\sum_{i,q=1}^mF^{i\bar q}A_{i\bar q}
=\sum_{k=1}^mf_k\lambda_k,\\
\label{ijfijaipapj}
\sum_{i,j,p=1}^mF^{ij}A_{i}{}^pA_p{}^j
=\sum_{i,j,p,q=1}^mF^{i\bar j}g^{\bar q p}A_{i\bar q} A_{p\bar j}
=\sum_{k=1}^mf_k\lambda_k^2.
\end{align}
\begin{lem}
\label{guanprop2.19}
Let $(F^{i\bar j})$ and $ (A_{i\bar j})$ be  $m\times m$ Hermitian matrices both of which can be diagonalized at the same time using one unitary matrix, and let $(f_1,\cdots,f_m)\in \Gamma_m$ and $(\lambda_1,\cdots,\lambda_m)\in \mathbb{R}^m$ be the eigenvalues of $(F^{i\bar j})$ and $ (A_{i\bar j})$ respectively. Then  there exists an index $r$ such that
\begin{equation*}
\sum_{\ell=1}^{m-1}F^{i\bar j} A_{i\bar \ell}A_{\ell \bar j}\geq \frac{1}{2}\sum_{k\not=r}f_k\lambda_k^2.
\end{equation*}
\end{lem}
\begin{proof} This is a Hermitian version of \cite[Proposition 2.19]{guan2014}.
We use the notations in the above paragraph with $g_{i\bar j}=\delta_{ij},$ and obtain
\begin{equation*}
\mathbb{R}\ni\sum_{\ell=1}^{m-1}F^{i\bar j} A_{i\bar \ell}A_{\ell \bar j}=\sum_{p=1}^mf_p\lambda_p^2(1-\zeta_m{}^p\xi_p{}^m)
=\sum_{p=1}^mf_p\lambda_p^2(1-\overline{\xi_p{}^m}\xi_p{}^m).
\end{equation*}
Note that$\sum_{p=1}^m\overline{\xi_p{}^m}\xi_p{}^m=1.$  Suppose that there exists some $r$ such that $\overline{\xi_r{}^m}\xi_r{}^m>1/2;$ otherwise we are done. One infers that
 $$
\sum_{p\not=r}\overline{\xi_p{}^m}\xi_p{}^m<1/2,
$$
and hence
\begin{equation*}
 \sum_{\ell=1}^{m-1}F^{i\bar j} A_{i\bar \ell}A_{\ell \bar j}\geq \sum_{p\not=r}f_p\lambda_p^2(1-\zeta_m{}^p\xi_p{}^m)
>\frac{1}{2}\sum_{p\not=r} f_p\lambda_p^2.
\end{equation*}
\end{proof}
\subsection{Existence of Admissible Subsolutions}
As pointed out in \cite{gaborjdg}, it is meaningful to find geometric conditions under which the admissible subsolution exists. If $M\subset\mathbb{R}^m$ is a bounded open set, then the authors \cite{cnsacta,trudinger1995} prove that the subsolutions exist under suitable convexity type condition on the boundary. Li \cite{lisongying2004} proves corresponding results for bounded open set $M\subset\mathbb{C}^m.$
\section{A Preliminary Estimate}\label{seczerofirstbd}
\begin{thm}
Let $(M,J,g)$ be a compact Hermitian manifold with smooth boundary, $\dim_{\mathbb{C}}M=m$, and the canonical complex structure $J$, where $g$ is the Hermitian metric. Suppose that $\underline{u}\in C^2(M,\mathbb{R})$ is an admissible  subsolution to \eqref{cma}-\eqref{cmabv} and that $u\in C^2(M,\mathbb{R})$ is a solution to \eqref{cma}-\eqref{cmabv}. There exists a uniform constant $C$ depending only on background data $(M,J,g),\,\varphi$ and $\underline{u}$ such that
\begin{equation}
\label{equzeroorder}
\sup_{M}|u|+\sup_{\partial M}|\partial u|_g\leq C.
\end{equation}
\end{thm}
\begin{proof}
It follows from \eqref{fijexpression} that $(F^{i\bar j}(A))$ has eigenvalues $f_1(\lambda),\cdots,f_m(\lambda)$ and hence is positive definite. Then we have
\begin{align*}
 F (\vartheta_{\underline{u}}^\flat )
-F (\vartheta_{u}^\flat )
=&\int_0^1\frac{\mathrm{d}}{\mathrm{d}t}F (\vartheta_{t\underline{u}+(1-t)u}^\flat  )\nonumber\\
=&\left(\int_0^1F^{i\bar q}\mathrm{d}t\right)  \Big( (\underline{u}-u)_{i\bar q}+W_{i\bar q}(\mathrm{d} (\underline{u}-u)\Big)\geq 0.\nonumber
\end{align*}
The maximum principle yields that
 \begin{equation}
 \label{underlineuu}
 \underline{u}\leq u,\quad \text{on}\; M.
 \end{equation}
 On the other hand, the definition of $\Gamma$ implies that $u$ satisfies
 \begin{equation*}
 \left\{\begin{array}{rl}
g^{\bar j i}\left(\chi_{i\bar j}+\partial_i\partial_{\bar j}u+W_{i\bar j}(\mathrm{d}u)\right)>0,&\quad \text{in} \quad M, \\
 u=\varphi,&\quad \text{on}\quad\partial M.
 \end{array}\right.
 \end{equation*}
 Hence one deduces that
 \begin{equation}
 \label{uh}
   u\leq \tilde\varphi,\quad\text{on}\;M
 \end{equation}
 by the maximum principle, where $\tilde\varphi$ is the solution to the Dirichlet problem
 \begin{equation*}
 \left\{\begin{array}{rl}
g^{\bar j i}\left(\chi_{i\bar j}+\partial_i\partial_{\bar j}\tilde\varphi+W_{i\bar j}(\mathrm{d} \tilde\varphi)\right)=0,&\quad \text{in} \quad M, \\
 \tilde\varphi=\varphi,&\quad \text{on}\quad\partial M.
 \end{array}\right.
 \end{equation*}
Now \eqref{equzeroorder} follows from \eqref{underlineuu} and \eqref{uh}.
\end{proof}
\section{Second Order Estimate on the Boundary}\label{sec2orderbd}
In this section, we prove the second order estimates of the solution $u$ to \eqref{cma} on the boundary.
\begin{thm}
\label{thm2ndonbd}
Let $(M,J,g)$ be a compact Hermitian manifold with smooth boundary, $\dim_{\mathbb{C}}M=m$, and the canonical complex structure $J$, where $g$ is the Hermitian metric. Suppose that $\underline{u}\in C^4(M,\mathbb{R})$ is an admissible  subsolution to \eqref{cma}-\eqref{cmabv} and that $u\in C^4(M,\mathbb{R})$ is a solution to \eqref{cma}-\eqref{cmabv}. There exists a uniform constant $C_K$ depending only on background data  and $K$ such that
\begin{equation*}
\sup_{\partial M}|\mathrm{Hess}_gu|_g \leq C_K,
\end{equation*}
where $\mathrm{Hess}_gu$ is the Hessian of $u$ with respect to the Levi-Civita connection of $g$ and $K:=1+\sup_{M}|\partial u|_g^2.$
\end{thm}
Let us recall some preliminaries for the proof of Theorem \ref{thm2ndonbd} from \cite{ckns1985,guan2014,boucksombd} and references therein.
For any point $\mathbf{0}\in\partial M,$  we use the adapted data $B,r,\mathbf{z}$ in Section \ref{secboundary} with
$$z_i=x_{2i-1}+\sqrt{-1}x_{2i},\quad 1\leq i\leq m.$$
Let $D_1,\cdots,D_{2m}$ be the dual vector fields of
\begin{equation*}
\mathrm{d}x_{\alpha},\;\mathrm{d}r,\quad \alpha\not=2m-1
\end{equation*}
given by
\begin{equation*}
D_\alpha:=\frac{\partial}{\partial x_\alpha}-\frac{r_{x_\alpha} }{r_{ x_{2m-1}}}\frac{\partial}{\partial x_{2m-1}},\quad \alpha \not=2m-1,
\end{equation*}
and
$$
D_{2m-1}:=-\frac{1 }{r_{x_{2m-1}}}\frac{\partial}{\partial x_{2m-1}}.
$$
To obtain Theorem \ref{thm2ndonbd}, it is sufficient to prove
\begin{align}
 \label{ttestimate}
&|D_\alpha D_\beta u|\leq CK,\quad \alpha,\beta\not=2m-1,\\
\label{tnestimate}
&|D_\beta D_{2m-1}u|\leq CK,\quad \beta\not=2m-1,\\
\label{nnestimate}
&|D_{2m-1} D_{2m-1}u|\leq C_K.
\end{align}
Let us define the distance function $\rho(\mathbf{x})$ by
\begin{equation*}
\rho(\mathbf{x}):=\mathrm{dist}_g(\mathbf{x},\mathbf{0}),\quad \forall\;\mathbf{x}\in M.
\end{equation*}
We set $M_\delta:=\{\mathbf{x}\in M:\;\rho(\mathbf{x})\leq \delta\}.$ Since $\sqrt{-1}\partial\bar\partial\rho^2(\mathbf{0})=\omega(\mathbf{0}),$ we may assume that
\begin{equation}
\label{ddbarrho2}
\frac{1}{2}\omega\leq \sqrt{-1}\partial\bar\partial\rho^2\leq 2\omega,\quad\text{on}\;M_\delta,
\end{equation}
provided that $\delta>0$ is chosen small enough.

We consider another distance function $d$ given by
\begin{equation*}
d(\mathbf{x}):=\mathrm{dist}_g(\mathbf{x},\partial M),\quad \forall\,\mathbf{x}\in M.
\end{equation*}
Since $\partial M$ is smooth, it follows from \cite[Lemma 14.16]{gt1998} that there exists a constant $\delta>0$ such that the distance function $d$ is smooth on $\{x\in M:\;d(x)\leq \delta\} $ and hence on $M_\delta.$
\begin{lem}
There exist uniform positive numbers $t,\,\delta,\,\varepsilon$ small enough and $N$ with $N\gg1$ such that the function
$$
v:=(u-\underline{u})+td-\frac{1}{2}Nd^2
$$
satisfies
\begin{equation}
\label{lv}
\left\{
\begin{array}{rl}
 L(v)\leq-\varepsilon(1+\mathcal{F}),&\quad \text{in}\;M_\delta; \\
  v\geq0,&\quad \text{on}\;\partial M_\delta,
\end{array}
\right.
\end{equation}
where the operator $L$ is given by \eqref{defnl}.
\end{lem}
\begin{proof}This is a Hermitian version of \cite[Lemma 4.1]{guan2014}, and we use the ideas modified from there based on \cite{gaborjdg} (see Section \ref{secsubsolution}).
Thank to \eqref{underlineuu}, we require $\delta\leq 2t/N$ in order to obtain $v\geq 0$ on $M_\delta$ after $t$ and $N$ being bounded.

A direct calculation yields that
\begin{align}
\label{lvesti}
L(v)=& F^{i\bar q}\left((u-\underline{u})_{i\bar q}+W_{i\bar q}(\mathrm{d}(u-\underline{u}))\right)\\
&+(t-Nd)F^{i\bar q} d_{i\bar q}-NF^{i\bar q}d_i d_{\bar q}
 +(t-Nd)F^{i\bar q} W_{i\bar q}(\mathrm{d} d) \nonumber\\
\leq &C_1(t-Nd)\mathcal{F}+F^{i\bar q} \left( (u-\underline{u})_{i\bar q}+W_{i\bar q}(\mathrm{d}(u-\underline{u}))\right)-NF^{i\bar q} d_i d_{\bar q}.\nonumber
\end{align}
Fix $\varepsilon>0$ sufficiently small. Since $\underline{u}$  is an admissible subsolution (and hence $\mathcal{C}$-subsolution), we can find $\epsilon_0>0$ small and $R>0$ large such that
\begin{equation*}
\left(\lambda(\vartheta_{\underline{u}}^\flat)-2\epsilon_0\mathbf{1}+\Gamma_m\right)\cap \partial\Gamma^{h(\mathbf{x})}\subset B_{R}(\mathbf{0}),\quad \forall\,\mathbf{x}\in M_\delta.
\end{equation*}
Let $\lambda_1(\vartheta_{u}^\flat)\geq \cdots\geq \lambda_m(\vartheta_{u}^\flat)$ and $\lambda_1(\vartheta_{\underline{u}}^\flat)\leq \cdots\leq \lambda_m(\vartheta_{\underline{u}}^\flat)$. Then the concavity of $f$ yields that $f_1\leq \cdots\leq f_m$ (see \cite{spruck} or \cite[Lemma 2]{eh89}). Then we can deduce
\begin{equation}
\label{fijabag}
F^{ij}(\vartheta_{u}^\flat)\left((\vartheta_{\underline{u}}^\flat)_i{}^j-(\vartheta_{u}^\flat)_i{}^j\right)\geq f_i(\lambda(\vartheta_{u}^\flat))
\left(\lambda_i(\vartheta_{u}^\flat)-\lambda_i(\vartheta_{\underline{u}}^\flat)\right).
\end{equation}
from these two inequalities and the theorem in \cite{richter} which states that for any $n\times n$ Hermitian matrices $A$ and $B$ with eigenvalues $\gamma_1\geq\cdots\geq\gamma_n$ and $\delta_1\geq\cdots\geq\delta_n$ respectively, there holds  the sharp estimate
\begin{equation*}
\sum\limits_{i}\gamma_i\delta_{n+1-i}\leq \tr(AB)\leq \sum\limits_{i}\gamma_i\delta_{i}.
\end{equation*}
One infers from \eqref{fijabag}  and the concavity of $f$  that
\begin{align}
\label{fabneg}
&F^{i\bar q} (\vartheta_{u}^\flat)\left( (u-\underline{u})_{i\bar q}+W_{i\bar q}(\mathrm{d}(u-\underline{u}))\right)\\
=&F^{ij}(\vartheta_{u}^\flat)\left((\vartheta_{u}^\flat)_i{}^j-(\vartheta_{\underline{u}}^\flat)_i{}^j\right)
\leq  f_i(\lambda(\vartheta_{u}^\flat))\left(\lambda_i(\vartheta_{u}^\flat)-\lambda_i(\vartheta_{\underline{u}}^\flat)\right)
\leq f(\lambda(\vartheta_{u}^\flat))- f(\lambda(\vartheta_{\underline{u}}^\flat))\leq 0,\nonumber
\end{align}
since $\underline{u}$ is an admissible subsolution to \eqref{cma}.

The following argument splits into two cases.

\textbf{Case 1:} $|\lambda(A)|\leq R.$ One can deduce that
\begin{equation*}
\left\{\lambda\in\Gamma:\;f(\lambda)\geq \inf_{M}h>\sup_{\partial \Gamma}f\right\}\cap \overline{B_R(\mathbf{0})}\subset\Gamma
\end{equation*}
is a compact set, and hence there exists a constant $C_2$ depending on the background data such that
\begin{equation*}
C_2\geq f_i\geq C_2^{-1}>0,\quad1\leq j\leq m,\quad\text{on}\;M_\delta.
\end{equation*}
This yields that
\begin{equation}
\label{fiqdidq}
  F^{i\bar q}d_i d_{\bar q} \geq 1/(2C_2),
\end{equation}
since $(d_i d_{\bar q})$ is a non-negative Hermitian matrix with respect to $g$ and $2g^{\bar q j}d_i d_{\bar q}=1.$ Then by \eqref{lvesti}, \eqref{fabneg} and \eqref{fiqdidq}, we can fix $N$ sufficiently large so that \eqref{lv} holds for $t,\varepsilon\in (0,1/2]$ provided that the positive number $\delta$ is small enough.

\textbf{Case 2:} $|\lambda|>R.$ Thanks to Assertion \eqref{gaborlem92} of Lemma \ref{lem9}, we deduce that
\begin{equation}
\label{ftau}
\mathcal{F}\geq \tau>0.
\end{equation}

Thanks to \eqref{fijabag} and Proposition \ref{gaborprop5},  one can deduce that either
\begin{equation}
\label{fijaba}
F^{ij}(\vartheta_{u}^\flat)\left((\vartheta_{\underline{u}}^\flat)_i{}^j-(\vartheta_{u}^\flat)_i{}^j\right)\geq \kappa \mathcal{F}
\end{equation}
or $f_i\geq \kappa \mathcal{F}$ for all $1\leq i\leq m.$

If \eqref{fijaba} occurs, then we have
 \begin{equation}
 \label{fijaab}
 F^{i\bar q} \left((u-\underline{u})_{i\bar q}+W_{i\bar q}(\mathrm{d}(u-\underline{u}))\right)=F^{ij}(\vartheta_{u}^\flat)
 \left((\vartheta_{u}^\flat)_i{}^j-(\vartheta_{\underline{u}}^\flat)_i{}^j\right)
 \leq -\kappa \mathcal{F}.
 \end{equation}
 Then \eqref{lv} follows from \eqref{lvesti}, \eqref{ftau} and \eqref{fijaab} provided $t$ and $\delta$ sufficiently small.

If $f_i\geq \kappa \mathcal{F},\,1\leq i\leq m$ occurs, then we can deduce
\begin{equation}
\label{nfg}
-NF^{i\bar q}  d_i d_{\bar q} \leq -c_2N\mathcal{F}
\end{equation}
with $c_2>0$ sufficiently small since $2g^{\bar q j}d_i d_{\bar q}=1$ and $(d_i d_{\bar q})$ is a non-negative Hermitian matrix with respect to $g.$

Therefore, \eqref{lv} follows from \eqref{fabneg}, \eqref{ftau} and \eqref{nfg} provided the positive constants $t,\,\delta$ and $\varepsilon$ sufficiently small.
\end{proof}
We will prove Theorem \ref{thm2ndonbd} by the ideas modified from the ones of \cite{ckns1985,trudinger1995,guan2014} in the local case and the Riemannian setup.
We use $D_\beta,\,1\leq \beta \leq 2m$ from \cite{ckns1985} where the complex Monge-Amp\`ere equation in $\Omega\subset \mathbb{C}^m$ was studied (see also \cite{guan1998,boucksombd} and references therein). For this aim, we need write Lemma \ref{guanprop2.19}  in a slightly different way, i.e., there exists a constant $c_0$ depending on $(M,J,g),\,\partial M$ and the adapted data $(B,r,\mathbf{z})$ such that
\begin{equation}
\label{p1m1ij1}
 \sum_{p=1}^{m-1}\sum_{i,j=1}^mF^{i\bar j}(\vartheta_u)_{i\bar p}(\vartheta_u)_{p\bar j}
\geq c_0 \sum_{r\not =r_0}f_r\lambda_r^2,
\end{equation}
for some $r_0$ with $1\leq r_0\leq m.$
Indeed, replacing $A_{i\bar j} $ by $(\vartheta_u)_{i\bar j}$, the calculation at the end of Section \ref{secsubsolution} yields that
\begin{equation*}
 \sum_{p=1}^{m-1}\sum_{i,j=1}^mF^{i\bar j}(\vartheta_u)_{i\bar p}(\vartheta_u)_{p\bar j}
=\sum_{p=1}^{m-1}\sum_{r=1}^mf_r\lambda_r^2\zeta_p{}^r\overline{\zeta_p{}^r}.
\end{equation*}
Since $(\zeta_i{}^j)$ is invertible depending only on the background data and the adapted data, there exists at most one index, say $r_0,$ such that $\sum_{p=1}^{m-1}\zeta_p{}^{r_0}\overline{\zeta_p{}^{r_0}}=0,$ as desired.
Although $c_0$ depending on the adapted data $(B,r,\mathbf{z}),$ it is still `uniform' for our estimates because we have fixed a family of finite adapted data $(B,r,\mathbf{z})$'s throughout this paper, and for the same reason, we can estimate the quantity like
$$
\sum_{p=1}^{m-1}\sum_{i,j=1}^mF^{i\bar j}(\vartheta_u)_{i\bar p}(\vartheta_u)_{p\bar j}
$$
which is not globally defined on the manifold.
\begin{proof}
[Proof of Theorem \ref{thm2ndonbd}]
A direct calculation yields that
all the $D_\beta$'s commute, and it follows Lemma \ref{lembou7.12} that $D_\beta,\beta\not=2m-1$ are tangent to $\partial M.$
Hence the tangent-tangent estimate \eqref{ttestimate} follows from
\begin{equation*}
 D_\alpha D_\beta(u-\underline{u})=0,\quad \alpha,\,\beta\not=2m-1.
\end{equation*}
Let us prove the normal-tangent estimates \eqref{tnestimate}.
For this aim, we consider the function
\begin{equation*}
Q:=Q_1\pm D_\alpha(u-\underline{u})
\end{equation*}
on $M_\delta$ with $$Q_1:=A_1Kv+A_2K\rho^2-\frac{1}{K}\sum_{\beta\not=2m-1}\left((u-\underline{u})_{x_{\beta}}\right)^2$$ and $\alpha\not=2m-1$ fixed, where $A_1$ and $A_2$ will be determined later.

For convenience, we write $a:=-r_{x_\beta}/r_{x_{2m-1}}$ and a direct calculation gives
\begin{align}
\label{dbetafbarq}
 \left(D_\beta(u-\underline{u})\right)_{\bar q}
=&D_{\beta}( (u-\underline{u})_{\bar q}) +  a_{\bar q}(u-\underline{u})_{x_{2m-1}} ,\\
 \label{dbetafbarqi}
 \left(D_\beta(u-\underline{u})\right)_{i\bar q}
=&D_\beta((u-\underline{u})_{i\bar q})
+a_i((u-\underline{u})_{\bar q})_{x_{2m-1}}\\
&+a_{i\bar q}(u-\underline{u})_{x_{2m-1}}+a_{\bar q}((u-\underline{u})_{i})_{x_{2m-1}}.\nonumber
\end{align}
Since
\begin{equation*}
\frac{\partial}{\partial x_{2m-1}}=2\frac{\partial}{\partial z_m}+\sqrt{-1}\frac{\partial}{\partial x_{2m}}
=2\frac{\partial}{\partial \bar z_m}-\sqrt{-1}\frac{\partial}{\partial x_{2m}},
\end{equation*}
we get
\begin{align}
\label{aiuunderlineu}
&a_i((u-\underline{u})_{\bar q})_{x_{2m-1}}+a_{\bar q}((u-\underline{u})_{i})_{x_{2m-1}}\\
=&2a_i(u-\underline{u})_{m\bar q}+2a_{\bar q}(u-\underline{u})_{i\bar m}\nonumber\\
&+\sqrt{-1}(a_i((u-\underline{u})_{\bar q })_{x_{2m}}-a_{\bar q}((u-\underline{u})_{i})_{x_{2m}}).\nonumber
\end{align}
From \eqref{cma}, we can obtain
\begin{equation}
\label{dbetah}
D_\beta h=F^{ij}\left(D_\beta g^{\bar qj}\right)(\vartheta_{u})_{i\bar q}+
F^{i\bar q}\left(D_\beta\chi_{i\bar q}+D_\beta u_{i\bar q}
+ D_\beta\left(W_{i\bar q}(\mathrm{d}u)\right)\right)
\end{equation}
Thanks to \eqref{dbetafbarq}, \eqref{dbetafbarqi}, \eqref{aiuunderlineu} and \eqref{dbetah}, it follows that
\begin{align}
\label{ldbetauunderlineu}
L\left(D_\beta(u-\underline{u})\right)
=&F^{i\bar q}\left(\left(D_\beta(u-\underline{u})\right)_{i\bar q}+W_{i\bar q}\left(\mathrm{d} \left(D_\beta(u-\underline{u})\right)\right)\right)\\
=&F^{i\bar q}D_\beta\left((u-\underline{u})_{i\bar q}+W_{i\bar q}(\mathrm{d} (u-\underline{u}))\right)
-F^{i\bar q} \left(D_\beta W_{i\bar q}^p\right)(u-\underline{u})_p\nonumber\\
&-F^{i\bar q}\overline{\left(D_\beta W_{q\bar i}^p\right)}(u-\underline{u})_{\bar p}
 +L(a) (u-\underline{u})_{x_{2m-1}}\nonumber\\
&+2F^{i\bar q}(a_i(u-\underline{u})_{m\bar q}+a_{\bar q}(u-\underline{u})_{i\bar m})\nonumber\\
&+\sqrt{-1}F^{i\bar q}\left(a_i((u-\underline{u})_{\bar q })_{x_{2m}}-a_{\bar q}((u-\underline{u})_{i})_{x_{2m}}\right)\nonumber\\
\leq &C\left(\left(1+|\partial u|_g\right)\left(1+ \mathcal{F}\right)+\sum_pf_p|\lambda_p|\right)\nonumber\\
&+ \left|F^{i\bar q}\left(a_i((u-\underline{u})_{\bar q })_{x_{2m}}-a_{\bar q}((u-\underline{u})_{i})_{x_{2m}}\right)\right|,\nonumber
\end{align}
where we denote by $\lambda_1,\cdots,\lambda_m$ the eigenvalues of $\vartheta_{u}^\flat.$

On the other hand, a direct calculation implies that
\begin{align}
\label{luunerlineu2}
L\left((u_{x_{\beta}}-\underline{u}_{x_{\beta}})^2\right)
=&2F^{i\bar q}((u-\underline{u})_{i})_{x_{\beta}}((u-\underline{u})_{\bar q})_{x_{\beta}}\\
&+2(u_{x_{\beta}}-\underline{u}_{x_{\beta}})F^{i\bar q}
\left((u-\underline{u})_{i\bar q x_\beta}+\left(W_{i\bar q}(\mathrm{d} (u-\underline{u}))\right)_{x_{\beta}}\right)\nonumber\\
&-2(u_{x_{\beta}}-\underline{u}_{x_{\beta}})F^{i\bar q} (W_{i\bar q}^p)_{x_{\beta}}(u-\underline{u})_p \nonumber\\
&-2(u_{x_{\beta}}-\underline{u}_{x_{\beta}})F^{i\bar q}  \left(\overline{W_{q\bar i}^{p}}\right)_{x_{\beta}}(u-\underline{u})_{\bar p} \nonumber\\
=&2F^{i\bar q} ((u-\underline{u})_{i})_{x_{\beta}}((u-\underline{u})_{\bar q})_{x_{\beta}}\nonumber\\
&+2(u_{x_{\beta}}-\underline{u}_{x_{\beta}})
\left(h_{x_{\beta}}-F^{i\bar q}\left(\chi_{i\bar q}\right)_{x_{\beta}}+O\left(\mathcal{F}+\sum_if_i|\lambda_i|\right)
\right)\nonumber\\
&-2(u_{x_{\beta}}-\underline{u}_{x_{\beta}})F^{i\bar q}(W_{i\bar q}^p)_{x_{\beta}}(u-\underline{u})_p \nonumber\\
&-2(u_{x_{\beta}}-\underline{u}_{x_{\beta}})F^{i\bar q} \left(\overline{W_{q\bar i}^{p}}\right)_{x_{\beta}}(u-\underline{u})_{\bar p}\nonumber\\
\geq&2F^{i\bar q} ((u-\underline{u})_{i})_{x_{\beta}}((u-\underline{u})_{\bar q})_{x_{\beta}}
-CK\mathcal{F}-CK^{1/2} \sum_if_i|\lambda_i| ,\nonumber
\end{align}
where for the second equality we use the equality
$$
h_{\beta}=F^{ij}(g^{\bar q j})_{x_\beta}(\vartheta_u)_{i\bar q}
+F^{i\bar q}\left(
(\chi_{i\bar q})_{x_{\beta}}+u_{i\bar q x_\beta}+\left(W_{i\bar q}(\mathrm{d}u)\right)_{x_\beta}
\right)
$$
by applying $\partial/\partial x_{\beta}$ to both sides of \eqref{cma}.

The Cauchy-Schwarz inequality yields that
\begin{align}
\label{cs1}
&-\frac{2}{K}F^{i\bar q}((u-\underline{u})_{i})_{x_{2m}}((u-\underline{u})_{\bar q})_{x_{2m}}\\
&+|F^{i\bar q}\left(a_i((u-\underline{u})_{\bar q })_{x_{2m}}-a_{\bar q}((u-\underline{u})_{i})_{x_{2m}}\right)|=O(K\mathcal{F}),\nonumber\\
\label{cs2}
&2\sum_{\beta=1}^{2m-2}F^{i\bar q}((u-\underline{u})_{i})_{x_{\beta}}((u-\underline{u})_{\bar q})_{x_{\beta}}\\
=&4\sum_{k=1}^{m-1}F^{i\bar q}
\left(
(u-\underline{u})_{i  k} (u-\underline{u})_{\bar q \bar k}
+(u-\underline{u})_{i  \bar k} (u-\underline{u})_{k\bar q}
\right)\nonumber\\
\geq &4\sum_{k=1}^{m-1}F^{i\bar q}
(u-\underline{u})_{i \bar k} (u-\underline{u})_{k\bar q}\nonumber\\
\geq&\sum_{k=1}^{m-1}F^{i\bar q} \left(\vartheta_u\right)_{i\bar k}\left(\vartheta_u\right)_{k\bar q}
-CK\mathcal{F},\nonumber\\
\geq&c_0\sum_{i\not=r}f_i\lambda_i^2-CK\mathcal{F},\nonumber
\end{align}
where we use \eqref{defnvartheta}, \eqref{p1m1ij1}, the Cauchy-Schwarz inequality, and $r$ is the index chosen as in \eqref{p1m1ij1}.

It follows from \eqref{filambdai} and \cite[Corollary 2.21]{guan2014} that
\begin{equation}
\label{filambdaileqfilambdai21}
 0\leq\sum_{k=1}^mf_k|\lambda_k|\leq \epsilon\sum_{i\not=r}f_i\lambda_i^2+\frac{C}{\epsilon}\mathcal{F}+P(r),\quad\forall\; r\in\{1,\cdots,m\}
\end{equation}
where $P(r)$ is uniformly bounded, for it given by
\begin{equation*}
 P(r):=\left\{
\begin{array}{rl}
  f(\lambda)-f(\mathbf{1}), &\quad \text{if}\;\lambda_r\geq0; \\
  0, & \quad \text{if}\;\lambda_r<0.
\end{array}
\right.
\end{equation*}
Thanks to \eqref{ddbarrho2}, \eqref{lv}, \eqref{ldbetauunderlineu}, \eqref{luunerlineu2}, \eqref{cs1}, \eqref{cs2} and \eqref{filambdaileqfilambdai21},  we can deduce that
\begin{equation*}
 \left\{
\begin{array}{rl}
  L(Q)\leq 0,&\quad \text{on}\quad M_\delta, \\
  Q\geq 0,&\quad  \text{on}\quad\partial M_\delta,
\end{array}
\right.
\end{equation*}
provided  the positive constants $A_1\gg A_2\gg 1$ independent of $K$ and $\epsilon\in(0,1)$  chosen carefully.
Then the minimum principle yields that $Q\geq 0$ on $M_\delta$, and hence on $M_\delta \cap \partial M$ there holds
\begin{align*}
 \left|D_\alpha D_{2m-1}u\right|
\leq&  \left|D_\alpha D_{2m-1}\underline{u}\right|
+ \left|D_{2m-1}Q_1\right| \\
=&\left|D_\alpha D_{2m-1}\underline{u}\right|
+ A_1K\left|D_{2m-1}v\right|
\leq CK,
\end{align*}
where for the last inequality we use \eqref{equzeroorder} and hence \eqref{tnestimate} follows.

For the normal-normal estimate \eqref{nnestimate}, we will use the method in \cite{trudinger1995,guan2014} in the local case and the Riemannian setup. For this aim, we need use the local unitary frames. We choose smooth orthonormal local frames  $X_1,\cdots,X_{2m}$ near $\mathbf{p}\in\partial M$  with respect to $g$ such that
\begin{equation*}
JX_{2i-1}=X_{2i},\quad 1\leq i\leq m,
\end{equation*}
and that $X_{2m-1}$ is the unit inner normal vector around $\mathbf{0}.$
We define a unitary basis of $(1,0)$ type frames by
\begin{equation*}
e_j:=\frac{1}{\sqrt{2}}\left(X_{2j-1}-\sqrt{-1}X_{2j}\right),\quad 1\leq j\leq m.
\end{equation*}
The linear operator $L$ given by \eqref{defnl} will be rewritten as
\begin{align*}
 L(v)=&F^{i\bar j}\left(e_i\bar e_{ j}(v)-[e_i,\bar e_{j}]^{(0,1)}(v)+W_{i\bar j}(\mathrm{d}v)\right)\\
=&F^{i\bar j}\left(\nabla_i\nabla_{\bar j}v+W_{i\bar j}^p(\nabla_p v)+\overline{W_{j\bar i}^p}(\nabla_{\bar p} v)\right), \quad \forall\;v\in C^2(M,\mathbb{R}).
\end{align*}
Using these unitary frames, we have
\begin{equation*}
 \lambda(\vartheta_u^\flat)=\lambda\left(\left(\vartheta_u)_{i\bar j}\right)_{1\leq i,j\leq m}\right),
\end{equation*}
and hence we can use the method in \cite{trudinger1995}.

Let $A=(a_{i\bar j})$ be $(m-1)\times (m-1)$ or $m\times m$ Hermitian matrix. Then we denote by $\lambda_1(A)\leq \cdots\leq \lambda_m(A)$ the eigenvalues of $A$ if $A$ is $m\times m$ Hermitian matrix, and by $\lambda'_1 (A)\leq\cdots\leq \lambda_{m-1}'(A)$ the eigenvalues of $A$ if it is an $(m-1)\times (m-1)$ Hermitian matrix. We also set $A'=(a_{i\bar j})_{1\leq i,j\leq m-1}$  if $A$ is an $m\times m$ Hermitian matrix.  Cauchy's interlace inequality (see for example \cite{Hwang})
yields that
\begin{equation}
\label{cauchyinterlace}
 \lambda_j(A)\leq \lambda'_{j}(A')\leq \lambda_{j+1}(A),\quad 1\leq j\leq m-1.
\end{equation}
Note that $a_{m\bar m}\in\mathbb{R}.$ It follows from \cite[Lemma 1.2]{cnsacta} that
\begin{align}
\label{cnslem121}
 \lambda_j(A)=&\lambda'_j(A')+o(1),\quad 1\leq j\leq m-1,\\
\label{cnslem122}
a_{m\bar m}\leq \lambda_m(A)=&a_{m\bar m}\left(1+O\left(\frac{1}{a_{m\bar m}}\right)\right),\quad \text{as}\quad|a_{m\bar m}|\to +\infty.
\end{align}
Thanks to \eqref{cone}, \eqref{ttestimate} and \eqref{tnestimate}, it is sufficient to get
\begin{equation}
\label{nnestimate0}
\nabla_m\nabla_{\bar m}u\leq C_K
\end{equation}
for the normal-normal estimate \eqref{nnestimate}.
The argument splits into two cases.

\textbf{Case 1:} there holds
\begin{equation*}
\lim_{\lambda_m\to+\infty}f(\lambda_1,\cdots,\lambda_{m-1},\lambda_m)=+\infty, \quad \lambda'=(\lambda_1,\cdots,\lambda_{m-1})\in\Gamma_\infty.
\end{equation*}
In this case, from \eqref{ttestimate} and \eqref{tnestimate}, it follows that $\lambda'\left(\vartheta_u'\right)$ lies in a compact set $L\subset\Gamma_\infty.$ Hence there exist  uniform positive constants $c_0$ and $R_0$ depending only on the range of $\lambda'\left(\vartheta_u' \right)$ such that for any $R\geq R_0$ one infers
\begin{equation*}
 f\left(\lambda'\left(\vartheta_u'\right),R\right)>\sup_{\mathbf{x}\in M}h(\mathbf{x})+c_0,
\end{equation*}
and hence we get
\begin{equation}
\label{flambda'lambdam}
 f\left(\lambda' ,R \right)>\sup_{\mathbf{x}\in M}h(\mathbf{x})+c_0/2,
\quad \forall\;\lambda'\in U_L,\quad R\geq R_0,
\end{equation}
since $f_m>0,$ where $U_L$ is the neighborhood of $L.$

From \eqref{cauchyinterlace}, \eqref{cnslem121} and \eqref{cnslem122}, one can deduce that there exists a $R_1\geq R_0$ such that if $\left(\vartheta_u\right)_{m\bar m}\geq R_1$
\begin{equation*}
 \lambda_m(\vartheta_u^\flat)\geq \left(\vartheta_u\right)_{m\bar m}\geq R_1\geq R_0,\quad \lambda_j(\vartheta_u^\flat)\in U_L, \quad 1\leq j\leq m-1.
\end{equation*}
This, together with \eqref{flambda'lambdam}, yields that
\begin{equation*}
 F(\vartheta_u^\flat)=f(\lambda(\vartheta_u^\flat))>\sup_{\mathbf{x}\in M}h(\mathbf{x})+c_0/2,
\end{equation*}
a contradiction to \eqref{cma}, and hence \eqref{nnestimate0} as well as \eqref{nnestimate} follows.

\textbf{Case 2:} there holds
\begin{equation*}
f_\infty(\lambda'):=\lim_{\lambda_m\to+\infty}f(\lambda_1,\cdots,\lambda_{m-1},\lambda_m)<+\infty, \quad \lambda'=(\lambda_1,\cdots,\lambda_{m-1})\in\Gamma_\infty.
\end{equation*}
For any $(m-1)\times(m-1)$ Hermitian matrix $E,$ if $\lambda'(E)\in \Gamma_\infty,$ we define
\begin{equation*}
 \tilde F(E)=f_\infty(\lambda'(E)).
\end{equation*}
Note that $\tilde F$ is also a concave function, and that
\begin{equation}
\label{littlecinfty}
 c_\infty:=\inf_{\partial M}\left(f_\infty\left(\lambda'\left(\vartheta_{\underline{u}}'\right)\right)-F(\vartheta_{\underline{u}})\right)>0.
\end{equation}
The concavity of $f$ yields that for any point $\mathbf{x}\in\partial M$, there exists a $(m-1)\times(m-1)$ Hermitian metric $\left(\tilde F^{i\bar j}(\mathbf{x})\right)_{1\leq i,j\leq m-1}$ such that
\begin{equation}
\label{concavityf}
 \sum_{i,j=1}^{m-1}\tilde F^{i\bar j}(\mathbf{x})\left(E_{i\bar j}-\left(\vartheta_u\right)_{i\bar j}(\mathbf{x})\right)
\geq \tilde F(E)-\tilde F(\vartheta_u'(\mathbf{x})),\quad \forall\;E\;\text{with}\;\lambda'(E)\in \Gamma_\infty.
\end{equation}
We assume that
\begin{equation*}
P_\infty:=\min_{\mathbf{x}\in\partial M}\left(\tilde F(\vartheta_u'(\mathbf{x}))-h(\mathbf{x})\right)
=\tilde F(\vartheta_u'(\mathbf{0}))-h(\mathbf{0}).
\end{equation*}
The argument in Case 1 yields that it is sufficient to prove
\begin{equation}
\label{mubiao}
 P_\infty>c_0>0
\end{equation}
for some uniform  constant $c_0.$

Taking $E=\vartheta_u'(\mathbf{x})$ in \eqref{concavityf} gives us that
\begin{align}
\label{ijtildefijo}
 &\sum_{i,j=1}^{m-1}\tilde F^{i\bar j}(\mathbf{0})\left(\vartheta_u(\mathbf{x})\right)_{i\bar j}
-h(\mathbf{x})-\sum_{i,j=1}^{m-1}\tilde F^{i\bar j}(\mathbf{0})\left(\vartheta_u(\mathbf{0})\right)_{i\bar j}
+h(\mathbf{0})\\
\geq& \tilde F(\vartheta_u'(\mathbf{x}))-h(\mathbf{x})-P_\infty\geq 0,\quad \forall\;\mathbf{x}\in\partial M.\nonumber
\end{align}
Note that on $\partial M$ there holds
\begin{align}
\label{varthetauuu}
 (\vartheta_u)_{i\bar j}-(\vartheta_{\underline{u}})_{i\bar j}
=& \nabla_i\nabla_{\bar j}(u-\underline{u})+W_{i\bar j}(\nabla(u-\underline{u}))\\
=&-g\left(X_{2m-1},\nabla_{i}\bar e_j\right)\nabla_{X_{2m-1}}(u-\underline{u})
+W_{i\bar j}(\nabla(u-\underline{u}))\nonumber\\
=&-\nabla_{X_{2m-1}}(u-\underline{u})\left(g\left(X_{2m-1},\nabla_{i}\bar e_j\right)
-\frac{1}{\sqrt{2}} \left(W_{i\bar j}^m+\overline{W_{j\bar i}^{m}}\right)\right) \nonumber
\end{align}
with $1\leq i,j\leq m-1.$ Here we remind that $\nabla_i\bar e_j=[e_i,\bar e_j]^{(0,1)}.$
Hence at $\mathbf{0},$ we can deduce from \eqref{littlecinfty}, \eqref{concavityf} and \eqref{varthetauuu} that
\begin{align}
\label{cinfty-pinfty}
 &\nabla_{X_{2m-1}}(u-\underline{u})\sum_{1\leq i,j\leq m-1}\tilde F^{i\bar j}\left(g\left(X_{2m-1},\nabla_{i}\bar e_j\right)
-\frac{1}{\sqrt{2}} \left(W_{i\bar j}^m+\overline{W_{j\bar i}^{ m}}\right)\right)\\
=&\sum_{1\leq i,j\leq m-1}\tilde F^{i\bar j} (\vartheta_{\underline{u}})_{i\bar j}
-\sum_{1\leq i,j\leq m-1}\tilde F^{i\bar j} (\vartheta_u)_{i\bar j}\nonumber\\
\geq&\tilde F(\vartheta_{\underline{u}}')-\tilde F(\vartheta_u')
= \tilde F(\vartheta_{\underline{u}}')-h(\mathbf{0})-P_\infty\geq c_\infty-P_\infty.\nonumber
\end{align}
We assume that
\begin{equation*}
 \nabla_{X_{2m-1}}(u-\underline{u})\sum_{1\leq i,j\leq m-1}\tilde F^{i\bar j}\left(g\left(X_{2m-1},\nabla_{i}\bar e_j\right)
-\frac{1}{\sqrt{2}} \left(W_{i\bar j}^m+\overline{W_{j\bar i}^{m}}\right)\right)>c_\infty/2,\quad\text{at}\;\mathbf{0};
\end{equation*}
otherwise  the equality \eqref{mubiao} follows from \eqref{cinfty-pinfty} and the conclusion follows. We set
\begin{equation*}
 \eta(\mathbf{x})=\sum_{1\leq i,j\leq m-1}\tilde F^{i\bar j}(\mathbf{0})\left(g\left(X_{2m-1},\nabla_{i}\bar e_j\right)
-\frac{1}{\sqrt{2}} \left(W_{i\bar j}^m+\overline{W_{j\bar i}^{ m}}\right)\right)(\mathbf{x}),\quad \forall\;\mathbf{x}\in\partial M.
\end{equation*}
It follows from \eqref{equzeroorder} that
\begin{equation*}
 \eta(\mathbf{0})\geq 2\epsilon_\infty c_\infty,
\end{equation*}
where $\epsilon_\infty>0$ is a uniform constant. We assume that $\eta\geq \epsilon_\infty c_\infty$ on $M_\delta$ with $\delta>0$ sufficiently small.

We consider the   quantity
\begin{align*}
 \Phi(\mathbf{x}):=&-\left(\nabla_{X_{2m-1}}(u-\underline{u})\right)(\mathbf{x})+
\frac{1}{\eta(\mathbf{x})}\sum_{i,j=1}^{m-1}\tilde F^{i\bar j}(\mathbf{0})\left(\left(\vartheta_{\underline{u}}(\mathbf{x})\right)_{i\bar j}-\left(\vartheta_u\right)_{i\bar j}(\mathbf{0})\right)
-\frac{h(\mathbf{x})-h(\mathbf{0})}{\eta(\mathbf{x})}\\
=:&-\left(\nabla_{X_{2m-1}}(u-\underline{u})\right)(\mathbf{x})+\Phi_1(\mathbf{x}),\quad \forall\;\mathbf{x}\in M_\delta.
\end{align*}
We deduce from \eqref{ijtildefijo} and \eqref{varthetauuu} that
\begin{equation}
\label{phigeq0}
 \Phi(\mathbf{0})=0,\quad \Phi(\mathbf{x})\geq 0\quad \forall\;\mathbf{x}\in M_\delta\cap (\partial M).
\end{equation}
It follows from \eqref{ricciid1} and \eqref{ricciid2} that
\begin{equation}
\label{fibarjnablainablabarj}
 F^{i\bar j}\nabla_i\nabla_{\bar j}\nabla_{X_{\beta}}u=F^{i\bar j}\nabla_{X_{\beta}}\nabla_i\nabla_{\bar j}u
+O\left( |\partial u|_g \mathcal{F}+\sum_{k}f_k|\lambda_k|\right).
\end{equation}
From \eqref{cma}, we get
\begin{align}
\label{xbetah}
 X_\beta h
=&F^{i\bar j}\left(\nabla_{X_{\beta}}\chi_{i\bar j}
+\nabla_{X_{\beta}}\nabla_i\nabla_{\bar j}u\right)\\
&+F^{i\bar j}\left( \left(\left(\nabla_{X_{\beta}} W_{i\bar j}^p\right)\nabla_p u\right)+ \left(\overline{\left(\nabla_{X_{\beta}} W_{j\bar i }^{p}\right)}\nabla_{\bar p}u\right)\right)\nonumber\\
&+F^{i\bar j}  \left(W_{i\bar j}^p\left(\nabla_p\nabla_{X_{\beta}}u-T(X_\beta,e_p)u\right) \right) \nonumber\\
&+ F^{i\bar j} \left(\overline{W_{j\bar i}^{  p}}\left(\nabla_{\bar p}\nabla_{X_{\beta}}u-T(X_\beta,\bar e_p)u\right) \right)\nonumber
\end{align}
From \eqref{fibarjnablainablabarj} and \eqref{xbetah}, it follows that
\begin{align}
\label{betauuu}
 L(\nabla_{X_{\beta}}(u-\underline{u}))
=&F^{i\bar j}\left(\nabla_i\nabla_{\bar j}\nabla_{X_\beta}(u-\underline{u})+W_{i\bar j}(\nabla\nabla_{X_{\beta}}(u-\underline{u}))\right)\\
=&O\left(1+\left(1+|\partial u|_g\right)\mathcal{F}+\sum_kf_k|\lambda_k|\right).\nonumber
\end{align}
Hence we have
\begin{align}
\label{2betauuu}
 &\sum_{\beta=1}^{2m-2}L\left(\left(\nabla_{X_\beta}(u-\underline{u})\right)^2\right)\\
=&2\sum_{\beta=1}^{2m-2}\left(\nabla_{X_\beta}(u-\underline{u})\right)
L\left(\nabla_{X_\beta}(u-\underline{u})\right)\nonumber\\
&+2\sum_{\beta=1}^{2m-2}F^{i\bar j}\left(\nabla_i\nabla_{X_\beta}(u-\underline{u})\right)
\left(\nabla_{\bar j}\nabla_{X_\beta}(u-\underline{u})\right)\nonumber\\
=&2 \sum_{k=1}^{m-1}F^{i\bar j}\left(\left(\nabla_i\nabla_k(u-\underline{u})\right)
\left(\nabla_{\bar j}\nabla_{\bar k}(u-\underline{u})\right)
+\left(\nabla_i\nabla_{\bar k}(u-\underline{u})\right)
\left(\nabla_{k}\nabla_{\bar j}(u-\underline{u})\right)\right)\nonumber\\
&+O\left(K+K\mathcal{F}+K\sum_{k}f_k|\lambda_k|\right)\nonumber\\
\geq&2 \sum_{k=1}^{m-1}F^{i\bar j} \left(\nabla_i\nabla_{\bar k}(u-\underline{u})\right)
\left(\nabla_{k}\nabla_{\bar j}(u-\underline{u})\right)
 -CK\left(1+ \mathcal{F}+ \sum_{k}f_k|\lambda_k|\right)\nonumber\\
\geq& F^{i\bar j}(\vartheta_u)_{i\bar k}(\vartheta_u)_{k \bar j}
-CK\left(1+\mathcal{F}+\sum_{k}f_k|\lambda_k|\right)\nonumber\\
\geq&\frac{1}{2}\sum_{k\not=r}f_k\lambda_k^2
-CK\left(1+ \mathcal{F}+\sum_{k}f_k|\lambda_k|\right),\nonumber
\end{align}
where for last second inequality we use \eqref{varthetav} and the Cauchy-Schwarz inequality, and for the last inequality we use Lemma \ref{guanprop2.19}.

We set
\begin{equation*}
 \Psi:=A_1Kv+A_2K\rho^2-\frac{1}{K}\sum_{\beta=1}^{2m-2}\left(\nabla_\beta(u-\underline{u})\right)^2.
\end{equation*}
Since $\Phi_1$ is a uniform quality, it follows from \eqref{filambdai}, \eqref{lv}, \eqref{filambdaileqfilambdai21}, \eqref{phigeq0}, \eqref{betauuu} and \eqref{2betauuu} that
\begin{equation*}
\left\{\begin{array}{rl}
         L(\Phi+\Psi)\leq 0, & \quad \text{on}\quad M_\delta,\\
           \Phi+\Psi\geq 0,& \quad \text{on}\quad \partial M_\delta,
       \end{array}
\right.
\end{equation*}
with $A_1\gg A_2\gg1$  chosen sufficiently large. The minimum principle yields that $\Phi+\Psi\geq 0$ on $M_\delta.$ This, together with the definition of $\Phi,$ yields that $\nabla_{X_{2m-1}}\nabla_{X_{2m-1}}u(\mathbf{0})\leq CK.$

Now we know that
$\lambda(\vartheta_{u}(\mathbf{0}))$ lies in the compact set by \eqref{cone} and Assumption \ref{assum2} of $f$ in the introduction.
Hence we have
\begin{equation*}
 P_\infty\geq f(\lambda'(\vartheta_u')(\mathbf{0}),R)-h(\mathbf{0})>0
\end{equation*}
for $R$ sufficiently large since $f_m>0,$ which yields \eqref{mubiao}, as desired.
\end{proof}
\subsection*{Remark}Actually our proof needs no restriction on $W_{i\bar j}^p.$

\section{Second Order Interior Estimate}\label{sec2orderin}
In this section, we prove the second order interior estimates
\begin{thm}
\label{thm2orderin}
Let $(M,J,g)$ be a compact Hermitian manifold without boundary and $\dim_{\mathbb{C}}M=m,$ where $g$ is the Hermitian metric with respect to the complex structure $J.$ Suppose that $\underline{u}\in C^4(M,\mathbb{R})$ is a $\mathcal{C}$-subsolution to \eqref{cma} and that $u\in C^4(M,\mathbb{R})$ is a solution to \eqref{cma}. There exists a uniform constant $C$ depending only on background data  such that
\begin{equation*}
\sup_{  M}|\sqrt{-1}\partial\bar\partial u|_g \leq C K,
\end{equation*}
where  $K:=1+\sup_{M}|\partial u|_g^2.$
\end{thm}
This theorem is obtain by \cite{yuanarxiv2020} and \cite{stw1503} (cf.\cite{gaborjdg}). Here we sketch a slightly different proof motivated by \cite{twarxiv1906}.
\begin{proof}[Proof of Theorem \ref{thm2orderin}] For convenience, we assume that $\underline{u}=0;$ otherwise, we can use $\chi_{\underline{u}}$ to replace $\chi.$
Let $\lambda_1\geq\lambda_2\geq\cdots\geq\lambda_m$ be the eigenvalues of $\vartheta_u^\flat,$ i.e., the eigenvalues of $((\vartheta_u)_{i\bar j})$ with respect to $g.$ Then since $\sum_{i=1}^m\lambda_i>0,$ it follows that $\lambda_1>0$ and we consider the quantity
$$
H(\mathbf{x}):= \log \lambda_1(\mathbf{x})+\varsigma(|\partial u|_g^2(\mathbf{x})) +\psi(u(\mathbf{x})),\quad \forall\,\mathbf{x}\in M,
$$
where we define
 $$
\varsigma(s)=-\frac{1}{2}\log\left(1-\frac{s}{2K}\right),\quad\psi(s)=D_1e^{-D_2s},
$$
with sufficiently large uniform constants $D_1,D_2>0$ to be determined later.
Note that
$$
\varsigma\left(|\nabla u|^2_{\alpha}\right)\in[0,\,2\log 2]
$$
and
\begin{align*}
\frac{1}{4K}<\varsigma'<\frac{1}{2K},\quad \varsigma''=2(\varsigma')^2>0.
\end{align*}
We assume that $H$ attains its maximum at the interior point $\mathbf{x}_0\in M.$ It suffices to show that there holds $\lambda_1\leq CK$ at $\mathbf{x}_0$ for some uniform constant $C.$ In what follows we may assume that $\lambda_1\gg K$ at the point $\mathbf{x}_0$ without loss of generality and hence \eqref{ftau} holds.
In the followup, we will calculate at the point $\mathbf{x}_0$ under the local coordinate $(z_1,\cdots,z_m)$ for which $g$ is the identity and $\vartheta_u$ is diagonal with entries $(\vartheta_u)_{i\bar i}=\lambda_i$ for $1\leq i\leq m,$ unless otherwise indicated. Note that $(F^{i\bar j})$ is also diagonal at the point $\mathbf{x}_0$ (see Section \ref{secsubsolution}).

\begin{lem}
\label{lemlpartialug2}
There exists a uniform constant $C>0$ such that
\begin{align}
\label{lpartialug2}
 L(|\partial u|_g^2)
=&\sum_kF^{i\bar i}\left(|\nabla_i\nabla_ku|^2+|\nabla_i\nabla_{\bar k}u|^2 \right)+2\Re\left(\sum_k\left(\nabla_ku\right)\left(\nabla_{\bar k}h\right)\right) \\
&+F^{i\bar i}\left(\nabla_ku\right)\overline{T_{ki}^p}(\nabla_i\nabla_{\bar p}u)
+F^{i\bar i}\left(\nabla_{\bar k}u\right)T_{ki}^p(\nabla_p\nabla_{\bar i}u)
+O(|\partial u|_g^2)\mathcal{F}\nonumber\\
\geq&\sum_kF^{i\bar i}\left(|\nabla_i\nabla_ku|^2+(1-\varepsilon)|\nabla_i\nabla_{\bar k}u|^2 \right)\nonumber\\
&+2\Re\left(\sum_k\left(\nabla_ku\right)\left(\nabla_{\bar k}h\right)\right)
-C\varepsilon^{-1}|\partial u|_g^2 \mathcal{F},\nonumber
\end{align}
where $\varepsilon$ is an arbitrary constant with $ \varepsilon\in (0,1/2].$
\end{lem}
\begin{proof}[Proof of Lemma \ref{lemlpartialug2}]
A direct calculation yields that
\begin{align}
\label{lpartialunorm}
L(|\partial u|_g^2)
=&F^{i\bar i}\big(
(\nabla_i\nabla_{\bar i}\nabla_pu)(\nabla_{\bar p}u)
+|\nabla_i\nabla_pu|^2
+|\nabla_{\bar i}\nabla_pu|^2
+(\nabla_pu)(\nabla_{i}\nabla_{\bar i}\nabla_{\bar p}u)
\big)\\
&+F^{i\bar i}W_{i\bar i}^k\big((\nabla_k\nabla_pu)(\nabla_{\bar p}u)
+(\nabla_pu)(\nabla_k\nabla_{\bar p}u)\big) \nonumber\\
&+F^{i\bar i}\overline{W_{i\bar i}^k}\big((\nabla_{\bar k}\nabla_pu)(\nabla_{\bar p}u)
+(\nabla_pu)(\nabla_{\bar k}\nabla_{\bar p}u)\big).\nonumber
\end{align}
One infers from  \eqref{cma} that
\begin{align}
\label{nablakh}
 \nabla_kh=&F^{i\bar j}\left(\nabla_k\chi_{i\bar j}+\nabla_k\nabla_{\bar j}\nabla_{i}u+\nabla_k(W_{i\bar j}(\mathrm{d}u))\right),\\
\label{nablakbarh}
\nabla_{\bar k}h=&F^{i\bar j}\left(\nabla_{\bar k}\chi_{i\bar j}+\nabla_{\bar k}\nabla_{\bar j}\nabla_iu+\nabla_{\bar k}(W_{i\bar j}(\mathrm{d}u))\right).
\end{align}
It follows from \eqref{ricciid0} that
\begin{align}
\label{nablakwiibardu}
 \nabla_k(W_{i\bar i}(\mathrm{d}u))
=&\nabla_k\left(W_{i\bar i}^p\nabla_pu+\overline{W_{i\bar i}^p}\nabla_{\bar p}u\right)\\
=& \left(\nabla_kW_{i\bar i}^p\right)\nabla_pu
+\overline{\left(\nabla_{\bar k}W_{i\bar i}^p\right)}\nabla_{\bar p}u\nonumber\\
&+W_{i\bar i}^p\nabla_p\nabla_ku
-W_{i\bar i}^pT_{kp}^q\nabla_qu
+\overline{W_{i\bar i}^p}\nabla_{\bar p}\nabla_{ k}u,\nonumber
\end{align}
and
\begin{align}
\label{nablakbarwiibardu}
 \nabla_{\bar k}(W_{i\bar i}(\mathrm{d}u))
=&\nabla_{\bar k}\left(W_{i\bar i}^p\nabla_pu+\overline{W_{i\bar i}^p}\nabla_{\bar p}u\right)\\
=& \left(\nabla_{\bar k}W_{i\bar i}^p\right)\nabla_pu
+\overline{\left(\nabla_{k}W_{i\bar i}^p\right)}\nabla_{\bar p}u\nonumber\\
&+W_{i\bar i}^p\nabla_p\nabla_{\bar k}u
+\overline{W_{i\bar i}^p}\nabla_{\bar p}\nabla_{\bar k}u
-\overline{W_{i\bar i}^pT_{kp}^q}\left(\nabla_{\bar q}u\right).\nonumber
\end{align}
Thanks to \eqref{ricciid0}, \eqref{ricciid1}, \eqref{ricciid2}, \eqref{nablakh}, \eqref{nablakbarh}, \eqref{nablakwiibardu} and \eqref{nablakbarwiibardu}, we can deduce
\begin{align}
\label{fiinablakunablaiibar}
 &F^{i\bar i}\nabla_ku\nabla_i\nabla_{\bar i}\nabla_{\bar k}u\\
=&F^{i\bar i}\nabla_ku\left(\nabla_{\bar k}\nabla_i\nabla_{\bar i}u
-\overline{T_{ik}^q}\nabla_{\bar q}u
+R_{i\bar ip \bar k} g^{\bar qp}\nabla_{\bar q}u\right)\nonumber\\
=&\left(\nabla_ku\right)\left(\nabla_{\bar k}h\right) +O(|\partial u|_g^2)\mathcal{F}
-F^{i\bar i}\left(\nabla_ku\right)\overline{T_{ki}^p}(\nabla_{\bar p}\nabla_iu) \nonumber\\
&-F^{i\bar i}\left(\nabla_ku\right)W_{i\bar i}^p\nabla_{p}\nabla_{\bar k}u
-F^{i\bar i}\left(\nabla_ku\right)\overline{W_{i\bar i}^p}\nabla_{\bar p}\nabla_{\bar k}u\nonumber
\end{align}
and
\begin{align}
\label{fiinablakbarnablainablaibarku}
 &F^{i\bar i}\left(\nabla_{\bar k}u\right)\left(\nabla_i\nabla_{\bar i}\nabla_{ k}u\right)\\
=&F^{i\bar i}\left(\nabla_{\bar k}u\right)\left(\nabla_k\nabla_i\nabla_{\bar i}u+T_{ki}^p\nabla_p\nabla_{\bar i}u\right)\nonumber\\
=&\left(\nabla_{\bar k}u\right)\left(\nabla_{k}h\right)+O(|\partial u|_g^2)\mathcal{F} -F^{i\bar i}\left(\nabla_{\bar k}u\right)W_{i\bar i}^p\nabla_p\nabla_ku\nonumber\\
&-F^{i\bar i}\left(\nabla_{\bar k}u\right)\overline{W_{i\bar i}^p}\nabla_k\nabla_{\bar p}u
+F^{i\bar i}\left(\nabla_{\bar k}u\right)T_{ki}^p(\nabla_p\nabla_{\bar i}u).\nonumber
\end{align}
From  \eqref{defnvartheta}, \eqref{lpartialunorm}, \eqref{nablakwiibardu}, \eqref{nablakbarwiibardu},
\eqref{fiinablakunablaiibar}, \eqref{fiinablakbarnablainablaibarku} and Young's inequality, one infers \eqref{lpartialug2}. This completes the proof of Lemma \ref{lemlpartialug2}.
\end{proof}

Since $\lambda_1$ may not be smooth at $\mathbf{x}_0,$ we define a smooth function $\phi$ on $M$ by (cf. \cite[Lemma 5]{brendleetx2017} and \cite[Proof of Theorem 3.1]{twarxiv1906})
 \begin{equation}
 \label{defnphi}
H(\mathbf{x}_0)\equiv \log \phi(\mathbf{x})+\varsigma(|\partial u|_g^2(\mathbf{x})) +\psi(u(\mathbf{x})),\quad \forall\,\mathbf{x}\in M.
\end{equation}
Note that $\phi$ satisfies
\begin{equation}
\phi(\mathbf{x})\geq \lambda_1(\mathbf{x})\quad \forall\,\mathbf{x}\in M,\quad \phi(\mathbf{x}_0)=\lambda_1(\mathbf{x}_0).
\end{equation}
Applying the operator $L$ defined in \eqref{defnl} to \eqref{defnphi}, one infers
\begin{align}
\label{0lhatx0}
0=&\frac{1}{\lambda_1}L(\phi)-\frac{1}{\lambda_1^2}F^{i\bar i}|\nabla_i\phi|^2+\varsigma'L(|\partial u|_g^2)+\psi'L(u)+\psi''F^{i\bar i}|\nabla_iu|^2\\
&+\varsigma''F^{i\bar i}
\left|\sum_p\big((\nabla_i\nabla_pu)(\nabla_{\bar p}u)
+(\nabla_pu)(\nabla_{i}\nabla_{\bar p}u)\big)\right|^2
 .\nonumber
\end{align}
Differentiating \eqref{defnphi} one can deduce
\begin{equation}
\label{nalbaiphi}
0=\frac{\nabla_i\phi}{\phi}+\varsigma'\left((\nabla_pu)(\nabla_{\bar p}\nabla_iu +(\nabla_{\bar p}u)\nabla_i\nabla_{p}u\right)+\psi'(\nabla_iu).
\end{equation}
\begin{lem}
\label{lemddbarphi}
Let $\mu$ denote the multiplicity of the largest eigenvalue of $\vartheta_u^\flat$ at $\mathbf{x}_0,$ so that
$\lambda_1=\cdots=\lambda_\mu>\lambda_{\mu+1}\geq \cdots\geq \lambda_m.$ Then at $\mathbf{x}_0,$ for each $i$ with $1\leq i\leq m,$ there hold
\begin{align}
\label{nablaiphi}
&\nabla_i(\vartheta_u)_{k\bar\ell}= (\nabla_i\phi) g_{k\bar \ell},\quad \mbox{for}\quad 1\leq k,\,\ell\leq \mu,\\
\label{ddbarphi}
&\nabla_{\bar i}\nabla_i\phi \geq \nabla_{\bar i}\nabla_{i}(\vartheta_u)_{1\bar 1}+\sum_{q>\mu}\frac{\left|\nabla_i(\vartheta_u)_{q\bar 1}\right|^2+\left|\nabla_{\bar i}(\vartheta_u)_{q\bar 1}\right|^2}{\lambda_1-\lambda_q}.
\end{align}
\end{lem}
\begin{proof}[Proof of Lemma \ref{lemddbarphi}]
This is a slight generalization of \cite[Lemma 3.2]{twarxiv1906} since $((\vartheta_u)_{k\bar \ell})$ is not necessarily positive definite and the proof is the same as \cite[Proof of Lemma 3.2]{twarxiv1906} by replacing $\tilde g$ there with $\vartheta_u$ here since it only uses the fact that $\phi$ is smooth and satisfies \eqref{defnphi} (i.e., \cite[Formula (3.4)]{twarxiv1906}). This completes the proof of Lemma \ref{lemddbarphi}.
\end{proof}
It follows from  \eqref{defnvartheta} and \eqref{ricciid2times} that
\begin{align}
\label{fibaribariithetau1bar1}
F^{i\bar i}\nabla_{\bar i}\nabla_{i}(\vartheta_u)_{1\bar 1}
=&F^{i\bar i}\nabla_{\bar i}\nabla_{i}\chi_{1\bar 1}
+F^{i\bar i}\nabla_{\bar i}\nabla_i\left(W_{1\bar1}^p\nabla_pu+\overline{W_{1\bar 1}^q}\nabla_{\bar q}u\right)\\
&+F^{i\bar i}\nabla_{\bar 1}\nabla_{1}\nabla_{\bar i}\nabla_{i}u
-F^{i\bar i}\left(T_{i1}^p\nabla_{\bar i}\nabla_{\bar 1}\nabla_pu
+\overline{T_{i1}^q}\nabla_i\nabla_{\bar q}\nabla_{ 1}u\right)\nonumber\\
&+F^{i\bar i}\left(R_{i\bar i1}{}^p\nabla_{\bar 1}\nabla_pu-R_{1\bar 1i}{}^p\nabla_{\bar i}\nabla_pu
-T_{i1}^p\overline{T_{i1}^q}\nabla_{\bar q}\nabla_pu
\right).\nonumber
\end{align}
Differentiating both sides of \eqref{nablakh} by $\nabla_{\bar \ell}$ gives
\begin{equation}
\label{nablabarknablakh}
\nabla_{\bar \ell}\nabla_{k}h
= F^{i\bar j,p\bar q}\left(\nabla_k(\vartheta_u)_{i\bar j}\right)
\left(\nabla_{\bar\ell}(\vartheta_u)_{p\bar q}\right)
+F^{i\bar j}
\big(
\nabla_{\bar \ell}\nabla_k\chi_{i\bar j}
+\nabla_{\bar \ell}\nabla_k\nabla_{\bar j}\nabla_iu
+\nabla_{\bar \ell}\nabla_k(W_{i\bar j}(\mathrm{d}u))
\big).
\end{equation}
Substituting \eqref{nablabarknablakh} with $k=\ell=1$ into \eqref{fibaribariithetau1bar1} yields
\begin{align}
\label{fibarinalbabarinablarivarthetau1bar1}
&F^{i\bar i}\nabla_{\bar i}\nabla_{i}(\vartheta_u)_{1\bar 1}\\
=&-F^{i\bar j,p\bar q}\left(\nabla_1(\vartheta_u)_{i\bar j}\right)
\left(\nabla_{\bar1}(\vartheta_u)_{p\bar q}\right)\nonumber\\
 &+F^{i\bar i}\left(\nabla_{\bar i}\nabla_{i}\chi_{1\bar 1}
-\nabla_{\bar 1}\nabla_{1}\chi_{i\bar i}\right)
+F^{i\bar i}\left(\nabla_{\bar i}\nabla_i\left(W_{1\bar1}(\mathrm{d}u)\right)-\nabla_{\bar 1}\nabla_1\left(W_{i\bar i}(\mathrm{d}u)\right)\right)\nonumber\\
&+\nabla_{\bar 1}\nabla_1h-F^{i\bar i}\left(T_{i1}^p\nabla_{\bar i}\nabla_{\bar 1}\nabla_pu
+\overline{T_{i1}^q}\nabla_{i}\nabla_{\bar q}\nabla_1u\right)\nonumber\\
&+F^{i\bar i}\left(R_{i\bar i1}{}^p\nabla_{\bar 1}\nabla_pu-R_{1\bar 1i}{}^p\nabla_{\bar i}\nabla_pu
-T_{i1}^p\overline{T_{i1}^q}\nabla_{\bar q}\nabla_pu
\right).\nonumber
\end{align}
It follows from  \eqref{defnvartheta}, \eqref{nablaiphi} and Young's inequality that
\begin{align}
\label{fibariti1pnablainablabar1nablapurealpart}
&F^{i\bar i}\left(T_{i1}^p\nabla_{\bar i}\nabla_{\bar 1}\nabla_pu
+\overline{T_{i1}^q}\nabla_{i}\nabla_{\bar q}\nabla_1u\right)\\
=&2\Re\left(F^{i\bar i}\overline{T_{i1}^q}\nabla_{i}(\vartheta_u)_{1\bar q}
-F^{i\bar i}\overline{T_{i1}^q}\left(
\left(\nabla_iW_{1\bar q}^p\right)\nabla_pu
+W_{1\bar q}^p\nabla_i\nabla_pu
+\overline{\nabla_{\bar i}W_{1\bar q}^p}\nabla_{\bar p}u
+\overline{W_{1\bar q}^p}\nabla_i\nabla_{\bar p}u
\right)\right) \nonumber\\
\geq &2\Re\left(F^{i\bar i}\overline{T_{i1}^1}\nabla_{i}(\vartheta_u)_{1\bar 1}\right)
+2\Re\left(F^{i\bar i}\sum_{q>\mu}\overline{T_{i1}^q}\nabla_{i}(\vartheta_u)_{1\bar q}\right)\nonumber\\
&-C\lambda_1\mathcal{F}-C\sum_{p}F^{i\bar i}|\nabla_i\nabla_pu|\nonumber\\
\geq&-CF^{i\bar i}|\nabla_{i}(\vartheta_u)_{1\bar 1}|-\sum_{q>\mu}F^{i\bar i}\frac{|\nabla_{i}(\vartheta_u)_{1\bar q}|^2}{\lambda_1-\lambda_q}-C\lambda_1\mathcal{F}-C\sum_{p}F^{i\bar i}|\nabla_i\nabla_pu|,\nonumber
\end{align}
where we use the fact that $\lambda_1\gg K>1$ and that both $|u_{i\bar j}|$ and $\lambda_q$ ($q>\mu$) can be controlled by $\lambda_1.$

\textbf{Case 1:} $f$ can be rewritten as \eqref{stwacta2017equ2.8}. In this case, we actually give a slightly different proof of the second order estimates in \cite{stw1503}. Hence we just point out the main differences and sketch the similar part. Note that \begin{equation}
\label{lambdamudefn}
\mu_j=\frac{1}{m-1}\sum_{k\not=j}\lambda_k,\quad \lambda_j=\sum_{k=1}^m\mu_k-(m-1)\mu_j.
\end{equation} We collect some basic properties about $f_i$ and $\tilde f_i$ from \cite{gaborjdg,stw1503}.
If $\lambda\in\Gamma$ with $\lambda_1\geq \cdots\geq \lambda_m,$ then $\mu_1\leq \cdots\leq\mu_m$ and $\tilde f_1\geq \cdots\geq \tilde f_m.$
We have
\begin{equation}
\label{fk}
f_k = \frac{1}{m-1}\sum_{i\ne k} \widetilde{f}_i,
\end{equation}
which yields that $0<f_1  \leq \cdots \leq f_m$.  There also holds
\begin{equation}
 \label{eq:fk1}
0<\frac{\tilde{f}_1}{m-1} \leq f_k \leq \tilde{f}_1,\quad k>1.
\end{equation}
In addition, it follows from  \eqref{fk}  with $k=1$ that
 \begin{equation}
  \label{eq:fk2}
 \widetilde{f}_i \leq (m-1)f_1, \quad i> 1.
\end{equation}
It follows from Condition \eqref{zu1} and Condition \eqref{zu2} of $Z(u)$ given by \eqref{eq:c1}, \eqref{lambdamudefn}, \eqref{fk}, \eqref{eq:fk1}, \eqref{eq:fk2} that (see the argument in \cite{stw1503})
\begin{align}
\label{eq:fk3}
&F^{i\bar i}W_{i\bar i}(\nabla(\vartheta_u)_{1\bar 1})\\
=&\tilde F^{i\bar i}Z_{i\bar i}(\nabla(\vartheta_u)_{1\bar 1})\nonumber\\
=&\tilde F^{1\bar 1}\sum_{p>1}\left(Z_{1\bar1}^p \nabla_p(\vartheta_u)_{1\bar 1}+\overline{Z_{1\bar1}^p}\nabla_{\bar p} (\vartheta_u)_{1\bar 1}\right)
+\sum_{i>1}\tilde F^{i\bar i}Z_{i\bar i}(\nabla(\vartheta_u)_{1\bar 1})\nonumber\\
\leq&C\sum_{p>1}F^{p\bar p}|\nabla_p (\vartheta_u)_{1\bar 1}|
+C F^{1\bar 1} \sum_{q }|\nabla_q (\vartheta_u)_{1\bar 1}|\nonumber\\
=&C\sum_{p>1}F^{p\bar p}|\nabla_p (\vartheta_u)_{1\bar 1}|
+C F^{1\bar 1}  |\nabla_1 (\vartheta_u)_{1\bar 1}|
+C F^{1\bar 1} \sum_{q>1 }|\nabla_q (\vartheta_u)_{1\bar 1}|\nonumber\\
\leq&C\sum_{p>1}F^{p\bar p}|\nabla_p (\vartheta_u)_{1\bar 1}|
+C F^{1\bar 1}  |\nabla_1 (\vartheta_u)_{1\bar 1}|
+C \sum_{q>1 }F^{q\bar q}|\nabla_q (\vartheta_u)_{1\bar 1}|\nonumber\\
\leq&C\sum_pF^{p\bar p}|\nabla_p (\vartheta_u)_{1\bar 1}|,\nonumber
\end{align}
\begin{equation}
\label{stwacta1fqqbar11wqq}
F^{i\bar i}\nabla_{\bar i}\nabla_i(W_{1\bar 1}(\mathrm{d}u))
\geq
-C\left(F^{i\bar i}\sum_p|\nabla_i\nabla_pu|+\lambda_1\mathcal{F}\right)
\end{equation}
and
\begin{equation}
\label{stwacta1fqqqqwbar11}
F^{i\bar i}\nabla_{\bar 1}\nabla_1(W_{i\bar i}(\mathrm{d}u))
\leq
 C\left(F^{i\bar i}\left(|\nabla_i(\vartheta_u)_{1\bar 1}|+\sum_p|\nabla_i\nabla_pu|\right)+\lambda_1\mathcal{F}\right),
\end{equation}
where we also use the fact that $\lambda_1\gg K>1$ and that $|u_{i\bar j}|$ can be controlled by $\lambda_1.$
Note that \eqref{stwacta1fqqbar11wqq} and \eqref{stwacta1fqqqqwbar11} can be found in \cite{stw1503} directly (cf. \cite{zhengimrn}).

Applying the operator $L$ defined in \eqref{defnl} to $\phi,$ we can deduce from \eqref{ftau}, \eqref{nablaiphi}, \eqref{ddbarphi}, \eqref{fibarinalbabarinablarivarthetau1bar1}, \eqref{fibariti1pnablainablabar1nablapurealpart}, \eqref{eq:fk3}, \eqref{stwacta1fqqbar11wqq} and \eqref{stwacta1fqqqqwbar11} that
\begin{align}
L(\phi)
\label{estimateoflphi}
=&F^{i\bar i}\left(\nabla_{\bar i}\nabla_i\phi+W_{i\bar i}(\nabla\phi)\right)\\
\geq&F^{i\bar i}\nabla_{\bar i}\nabla_i(\vartheta_u)_{1\bar1}
+F^{i\bar i}W_{i\bar i}(\nabla (\vartheta_u)_{1\bar1})\nonumber\\
&+\sum_{q>\mu}F^{i\bar i}\frac{|\nabla_i(\vartheta_u)_{q\bar 1}|^2+|\nabla_{\bar i}(\vartheta_u)_{q\bar 1}|^2}{\lambda_1-\lambda_q}\nonumber\\
\geq&-F^{i\bar j,p\bar q}\left(\nabla_1(\vartheta_u)_{i\bar j}\right)
\left(\nabla_{\bar1}(\vartheta_u)_{p\bar q}\right)\nonumber\\
&+2\Re\left(F^{i\bar i}\overline{T_{i1}^q}\nabla_{i}(\vartheta_u)_{1\bar q}\right)
+\sum_{q>\mu}F^{i\bar i}\frac{|\nabla_i(\vartheta_u)_{q\bar 1}|^2+|\nabla_{\bar i}(\vartheta_u)_{q\bar 1}|^2}{\lambda_1-\lambda_q}\nonumber\\
&-C F^{i\bar i}\left(|\nabla_i(\vartheta_u)_{1\bar 1}|+\sum_p|\nabla_i\nabla_pu|\right)-C\lambda_1\mathcal{F}  \nonumber\\
\geq&-F^{i\bar j,p\bar q}\left(\nabla_1(\vartheta_u)_{i\bar j}\right)
\left(\nabla_{\bar1}(\vartheta_u)_{p\bar q}\right)\nonumber\\
&-C F^{i\bar i}\left(|\nabla_i(\vartheta_u)_{1\bar 1}|+\sum_p|\nabla_i\nabla_pu|\right)-C\lambda_1\mathcal{F}  \nonumber
\end{align}
From \eqref{0lhatx0} and \eqref{estimateoflphi}, one can infer that
\begin{align}
\label{0lhatx01}
0\geq&-\frac{1}{\lambda_1}F^{i\bar j,p\bar q}\left(\nabla_1(\vartheta_u)_{i\bar j}\right)
\left(\nabla_{\bar1}(\vartheta_u)_{p\bar q}\right)
-\frac{1}{\lambda_1^2}F^{i\bar i}|\nabla_i\phi|^2  \\
&+\varsigma'L(|\partial u|_g^2)+\varsigma''F^{i\bar i}
\left|\sum_p\big((\nabla_i\nabla_pu)(\partial_{\bar p}u)
+(\nabla_pu)(\nabla_{i}\nabla_{\bar p}u)\big)\right|^2
  \nonumber\\
&+\psi'L(u)+\psi''F^{i\bar i}|\nabla_iu|^2-\frac{C}{\lambda_1} F^{i\bar i}\left(|\nabla_i(\vartheta_u)_{1\bar 1}|+\sum_p|\nabla_i\nabla_pu|\right)-C \mathcal{F}.\nonumber
\end{align}
Since $\varsigma'\geq 1/(4K),$  it follows from \eqref{lpartialug2} with $\varepsilon=1/3$ and Young's inequality that
\begin{align}
\label{lgradientusquareuse2}
&\varsigma'L(|\partial u|_g^2)-\frac{C}{\lambda_1} F^{i\bar i}\sum_p|\nabla_i\nabla_pu|\\
\geq&\frac{1}{6K}\sum_kF^{i\bar i}\left(|\nabla_i\nabla_ku|^2 +|\nabla_i\nabla_{\bar k}u|^2 \right)-C\mathcal{F},\nonumber
\end{align}
where we also use the fact that $\lambda_1\gg K>1.$
Hence we can deduce from \eqref{0lhatx01} and \eqref{lgradientusquareuse2} that
\begin{align}
\label{0lhatx02}
0\geq&-\frac{1}{\lambda_1}F^{i\bar j,p\bar q}\left(\nabla_1(\vartheta_u)_{i\bar j}\right)
\left(\nabla_{\bar1}(\vartheta_u)_{p\bar q}\right)
-\frac{1}{\lambda_1^2}F^{i\bar i}|\nabla_i\phi|^2  \\
&+\frac{1}{6K}\sum_kF^{i\bar i}\left(|\nabla_i\nabla_ku|^2+|\nabla_i\nabla_{\bar k}u|^2 \right)\nonumber\\
&+\varsigma''F^{i\bar i}
\left|\sum_p\big((\nabla_i\nabla_pu)(\nabla_{\bar p}u)
+(\nabla_pu)(\nabla_{i}\nabla_{\bar p}u)\big)\right|^2
  \nonumber\\
&+\psi'L(u)+\psi''F^{i\bar i}|\nabla_iu|^2-\frac{C}{\lambda_1} F^{i\bar i} |\nabla_i(\vartheta_u)_{1\bar 1}| -C \mathcal{F}.\nonumber
\end{align}
From \eqref{nablaiphi}, we know that $\nabla_i\phi=\nabla_i(\vartheta_u)_{1\bar 1}.$ This, together with \cite[Equation (3.7)]{stw1503}, yields that \eqref{0lhatx02} is the same as \cite[Equation (3.28)]{stw1503} essentially after changing  $\nabla_p\nabla_pu$ and $\nabla_i(\vartheta_u)_{1\bar 1}$ into $\partial_i\partial_p u$ and $\partial_i(\vartheta_u)_{1\bar 1}$ respectively, and changing the coefficient of $\mathcal{F}$ into a larger uniform constant (only replacing $\tilde\lambda_{1,k}$ in \cite{stw1503} with $\nabla_k\phi $, $g_{k\bar\ell}$ in \cite{stw1503} by $(\vartheta_u)_{k\bar\ell}$, the metric $\alpha$ in \cite{stw1503} by $g$, $\phi$ in \cite{stw1503} by $\varsigma$ and $F^{k\bar k}u_{k\bar k}$ in \cite{stw1503} by $L(u)$).
After changing these notations, we can repeat the argument in \cite{stw1503} word for word to get
\begin{equation}
\lambda_1\leq CK,
\end{equation}
by replacing $\tilde H_k=0$ in \cite{stw1503} by \eqref{nalbaiphi} and replacing the paragraph between \cite[Inequality (3.53)]{stw1503}(not containing) and \cite[Inequality (3.54)]{stw1503}(containing) by $$\psi'L(u)= \psi'F^{k\bar k}((\vartheta_u)_{k\bar k}-\chi_{k\bar k}).$$
\textbf{Case 2:} $f$ cannot be rewritten as \eqref{stwacta2017equ2.8}. In this case, we actually give a slightly different proof of the second order estimates in \cite{gaborjdg,yuanarxiv2020}. Hence we just point out the main differences and sketch the similar part.

We can deduce from \eqref{ricciid0}, \eqref{ricciid1}, \eqref{ricciid2} and \eqref{nablakwiibardu} that
\begin{align}
\label{nablakwiibardu}
 \nabla_{\bar \ell}\nabla_k(W_{i\bar i}(\mathrm{d}u))
=& \left(\nabla_{\bar \ell}\nabla_kW_{i\bar i}^p\right)\nabla_pu
+\left(\nabla_kW_{i\bar i}^p\right)\left(\nabla_{\bar \ell}\nabla_pu\right)\\
&+\overline{\left(\nabla_\ell\nabla_{\bar k}W_{i\bar i}^p\right)}\nabla_{\bar p}u
+\overline{\left(\nabla_{\bar k}W_{i\bar i}^p\right)}
\nabla_{\bar \ell}\nabla_{\bar p}u\nonumber\\
&+(\nabla_{\bar \ell}W_{i\bar i}^p)\left(\nabla_p\nabla_ku-T_{kp}^q\nabla_qu\right)
+W_{i\bar i}^p\nabla_{\bar \ell}\nabla_k\nabla_pu\nonumber\\
&+\overline{\nabla_{\ell}W_{i\bar i}^p} \nabla_{\bar p}\nabla_ku
+\overline{W_{i\bar i}^p}\nabla_{\bar \ell}\nabla_{\bar p}\nabla_{ k}u\nonumber\\
=& \left(\nabla_{\bar \ell}\nabla_kW_{i\bar i}^p\right)\nabla_pu
+\left(\nabla_kW_{i\bar i}^p\right)\left(\nabla_{\bar \ell}\nabla_pu\right)\nonumber\\
&+\overline{\left(\nabla_\ell\nabla_{\bar k}W_{i\bar i}^p\right)}\nabla_{\bar p}u
+\overline{\left(\nabla_{\bar k}W_{i\bar i}^p\right)}
\left(\nabla_{\bar p}\nabla_{\bar \ell}u-\overline{T_{\ell p}^q}\nabla_{\bar q}u\right)\nonumber\\
&+(\nabla_{\bar \ell}W_{i\bar i}^p)\left(\nabla_p\nabla_ku-T_{kp}^q\nabla_qu\right)\nonumber\\
&+W_{i\bar i}^p\left(\nabla_p\nabla_{\bar \ell}\nabla_ku
-T_{kp}^q\nabla_q\nabla_{\bar\ell}u
+R_{k\bar\ell p}{}^q\nabla_qu\right)\nonumber\\
&+\overline{\nabla_{\ell}W_{i\bar i}^p} \nabla_{\bar p}\nabla_ku
+\overline{W_{i\bar i}^p}\left(\nabla_{\bar p}\nabla_{\bar \ell}\nabla_{ k}u-\overline{T_{\ell p}^q}\nabla_{\bar q}\nabla_ku\right).\nonumber
\end{align}
A direct calculation with \eqref{nablakh} yields that
\begin{align}
\label{fibariwibaripnablapnablabarinablaiu}
&F^{i\bar i}W_{1\bar 1}^p\nabla_p\nabla_{\bar i}\nabla_iu\\
=&W_{1\bar 1}^p\nabla_pu-F^{i\bar i}W_{1\bar 1}^p\left(\nabla_p\chi_{i\bar i}+(\nabla_pW_{i\bar i}^q)\nabla_qu
+ \overline{\nabla_{\bar p}W_{i\bar i}^q} \nabla_{\bar q}u
+W_{i\bar i}^q\nabla_p\nabla_qu
+W_{i\bar i}^q\nabla_p\nabla_{\bar q}u
\right).\nonumber
\end{align}
Substituting \eqref{secondwdu} and \eqref{fibariwibaripnablapnablabarinablaiu}   into \eqref{nablakwiibardu} with $(\ell,k,i,i)=(i,i,1,1)$ yields that
\begin{align}
\label{fibarinalbabarinablaiw1bar1du}
F^{i\bar i}\nabla_{\bar i}\nabla_i(W_{1\bar 1}(\mathrm{d}u))
=&O(|\partial u|_g)\mathcal{F}+O(1)\sum_pF^{i\bar i}(|\nabla_p\nabla_iu|+|\nabla_p\nabla_{\bar i}u|)\\
 &+2\Re\left(F^{i\bar i}W_{1\bar 1}^p\nabla_p\nabla_{\bar i}\nabla_iu\right) \nonumber\\
\geq &-CF^{i\bar i}|\nabla_i\nabla_1u|-C\lambda_1\mathcal{F} ,\nonumber
\end{align}
where we use \eqref{ftau} and the  fact that $\lambda_1\gg K>1.$

It follows from \eqref{nablakwiibardu} with $(\ell,k)=(1,1)$ that
\begin{equation}
\label{fibarinablabar1nabla1wibari}
-F^{i\bar i}\nabla_{\bar 1}\nabla_1(W_{i\bar i}(\mathrm{d}u))
\geq -F^{i\bar i}a_{\bar i}\nabla_i\nabla_{\bar 1}\nabla_1 u
-F^{i\bar i}a_{ i}\nabla_{\bar i}\nabla_{\bar 1}\nabla_1 u
-CF^{i\bar i}|\nabla_i\nabla_1u|-C\lambda_1\mathcal{F},
\end{equation}
where we use the fact  $\lambda_1\gg K>1.$

We can deduce from  \eqref{defnvartheta},  \eqref{secondwdu} and \eqref{ftau} that
\begin{align}
\label{fiithetauibar1ithetau1bar1bari}
&F^{i\bar i}
\left(
W_{i\bar i}^p\nabla_p(\vartheta_u)_{1\bar 1}
+\overline{W_{i\bar i}^p}\nabla_{\bar p}(\vartheta_u)_{1\bar 1} \right)\\
=&F^{i\bar i}W_{i\bar i}^p\left(
\nabla_p\chi_{1\bar 1}
+\nabla_p\nabla_{\bar 1}\nabla_1u
+\nabla_pW_{1\bar 1}^q\nabla_qu
+W_{1\bar 1}^q\nabla_p\nabla_qu
+\overline{\nabla_{\bar p}W_{1\bar 1}^q}\nabla_{\bar q}u
+\overline{W_{1\bar 1}^q}\nabla_{\bar q}\nabla_pu\right)\nonumber\\
&+F^{i\bar i}\overline{W_{i\bar i}^p}\left(
\nabla_{\bar p}\chi_{1\bar 1}
+\nabla_{\bar p}\nabla_{\bar 1}\nabla_1u
+\nabla_{\bar p}W_{1\bar 1}^q\nabla_qu
+W_{1\bar 1}^q\nabla_{\bar p}\nabla_qu
+\overline{\nabla_{ p}W_{1\bar 1}^q}\nabla_{\bar q}u
+\overline{W_{1\bar 1}^q}\nabla_{\bar q}\nabla_{\bar p}u\right)\nonumber\\
\geq&F^{i\bar i}a_{\bar i}\nabla_i\nabla_{\bar 1}\nabla_1u
+F^{i\bar i}a_i\nabla_{\bar i}\nabla_{\bar 1}\nabla_1u
-CF^{i\bar i}|\nabla_i\nabla_1u|-C\lambda_1\mathcal{F},\nonumber
\end{align}
where we also use the fact that $\lambda_1\gg K>1.$

Applying the operator $L$ defined in \eqref{defnl} to $\phi,$ we can deduce from \eqref{nablaiphi}, \eqref{ddbarphi}, \eqref{fibarinalbabarinablarivarthetau1bar1}, \eqref{fibariti1pnablainablabar1nablapurealpart}, \eqref{fibarinalbabarinablaiw1bar1du}, \eqref{fibarinablabar1nabla1wibari} and \eqref{fiithetauibar1ithetau1bar1bari} that
\begin{align}
L(\phi)
\label{estimateoflphicase2}
=&F^{i\bar i}\left(\nabla_{\bar i}\nabla_i\phi+W_{i\bar i}(\nabla\phi)\right)\\
\geq&F^{i\bar i}\nabla_{\bar i}\nabla_i(\vartheta_u)_{1\bar1}
+F^{i\bar i}W_{i\bar i}(\nabla (\vartheta_u)_{1\bar1})\nonumber\\
&+\sum_{q>\mu}F^{i\bar i}\frac{|\nabla_i(\vartheta_u)_{q\bar 1}|^2+|\nabla_{\bar i}(\vartheta_u)_{q\bar 1}|^2}{\lambda_1-\lambda_q}\nonumber\\
\geq&-F^{i\bar j,p\bar q}\left(\nabla_1(\vartheta_u)_{i\bar j}\right)
\left(\nabla_{\bar1}(\vartheta_u)_{p\bar q}\right)\nonumber\\
&-C F^{i\bar i}\left(|\nabla_i(\vartheta_u)_{1\bar 1}|+ |\nabla_i\nabla_1u|\right)-C\lambda_1\mathcal{F}.  \nonumber
\end{align}
It follows from \eqref{estimateoflphicase2} and the arguments in Case 1 that we can also have \eqref{0lhatx02}, which the same as \cite[Equation (3.28)]{stw1503} essentially after changing  $\nabla_p\nabla_pu$ and $\nabla_i(\vartheta_u)_{1\bar 1}$ into $\partial_i\partial_p u$ and $\partial_i(\vartheta_u)_{1\bar 1}$ respectively, and changing the coefficient of $\mathcal{F}$ into a larger uniform constant (only replacing $\tilde\lambda_{1,k}$ in \cite{stw1503} with $\nabla_k\phi $, $g_{k\bar\ell}$ in \cite{stw1503} by $(\vartheta_u)_{k\bar\ell}$, the metric $\alpha$ in \cite{stw1503} by $g$, $\phi$ in \cite{stw1503} by $\varsigma$ and $F^{k\bar k}u_{k\bar k}$ in \cite{stw1503} by $L(u)$).
Note that in this case, $Z(u)$ given by \eqref{eq:c1} also satisfies
Condition \eqref{zu1} and Condition \eqref{zu2}, and  the argument after
\cite[Equation (3.28)]{stw1503}  does not use \eqref{lambdamudefn}, \eqref{fk}, \eqref{eq:fk1} or \eqref{eq:fk2}.
Hence after changing these notations, we can repeat the argument after
\cite[Equation (3.28)]{stw1503}  word for word to get
\begin{equation}
\lambda_1\leq CK,
\end{equation}
by replacing $\tilde H_k=0$ in \cite{stw1503} by \eqref{nalbaiphi} and replacing the paragraph between \cite[Inequality (3.53)]{stw1503}(not containing) and \cite[Inequality (3.54)]{stw1503}(containing) by $$\psi'L(u)= \psi'F^{k\bar k}((\vartheta_u)_{k\bar k}-\chi_{k\bar k}).$$
This completes the proof of Theorem \ref{thm2orderin}.
\end{proof}

\subsection*{Remark} In Case 2 above, we can choose different auxiliary functions in \cite{gaborjdg} to simplify the argument. Here considering the length of the paper, we just use the same auxiliary functions as the ones used in Case 1.

In the proof above, we used a viscosity type argument to deal with the non-differentiability of the largest eigenvalue $\lambda_1 $ as in \cite{twarxiv1906} (cf.\cite{brendleetx2017}). The authors of \cite{gaborjdg,stw1503} use a perturbation argument to deal with this issue.
As pointed out in \cite{stw1503}, maybe we can overcome the  non-differentiability of the largest eigenvalue $\lambda_1 $ by a carefully chosen quadratic function $(\vartheta_u)_{i\bar j}$ as in \cite{twjams,twcrelle}.

Thanks to \eqref{cone}, Theorem \ref{thm2ndonbd} and Theorem \ref{thm2orderin}, we can deduce
\begin{cor}
Let $(M,J,g)$ be a compact Hermitian manifold with smooth boundary, $\dim_{\mathbb{C}}M=m$, and the canonical complex structure $J$, where $g$ is the Hermitian metric. Suppose that $\underline{u}\in C^4(M,\mathbb{R})$ is an admissible  subsolution to\eqref{cma}-\eqref{cmabv} and that $u\in C^4(M,\mathbb{R})$ is a solution to \eqref{cma}-\eqref{cmabv}. There exists a uniform constant $C_K$ depending only on background data  and $K$such that
\begin{align}
\sup_M\left|\lambda_j(\vartheta_u^\flat) \right|\leq &C_K;\nonumber\\
\label{fjbd}
C_K^{-1}\leq f_j\leq& C_K,\quad j=1,\cdots,m.
\end{align}
\end{cor}
\begin{proof}
From Theorem \ref{thm2ndonbd} and Theorem \ref{thm2orderin} it follows that
\begin{equation*}
 |\sqrt{-1}\partial\bar\partial u|_M\leq C_K.
\end{equation*}
This, together with  and Assumption \ref{assum2} of $f$ in the Introduction, yields that $\lambda\left(\vartheta_u^\flat\right)$ lies in the compact set
\begin{equation*}
 \left\{f\geq \inf_Mh>\sup_{\partial \Gamma}f\right\}\cap \bar B_R(\mathbf{0})\subset \Gamma,
\end{equation*}
for some uniform $R>0,$ where $B_R(\mathbf{0})\subset \mathbb{R}^m$ is a ball centered at the origin with radius $R,$ as desired.
\end{proof}
\section{First Order Estimate}\label{sec1st}
In this section,  prove the first order estimates of the solutions to the Dirichlet problem \eqref{cma}-\eqref{cmabv}   with $ f=\left(\sigma_m/\sigma_\ell\right)^{1/(m-\ell)},$  $\Gamma=\Gamma_m$ and $1\leq \ell\leq m-1.$

\begin{proof}
[Proof of Theorem \ref{thmhessianquotient}] It is sufficient to prove the upper bound of the gradient of the solution.
 It is equivalent to obtain the upper bound of $|\partial u|_g$, where $u$ is the solution to the Dirichlet problem
 \begin{equation}
\label{cma0}
\left\{\begin{array}{rl}
          F\left(\vartheta_u^{\flat}\right)=f\left(\lambda\left(\vartheta_u^{\flat}\right)\right)=h, &\quad\text{on}\quad M, \\
         u=0, & \quad\text{on}\quad \partial M,
       \end{array}
\right.
\end{equation}
where $0$ is an admissible subsolution to \eqref{cma0}. Note that $\chi$ plays the role of $\vartheta_{\underline{u}}$ in Theorem \ref{mainthm}.
Recall from \cite{cnsacta} that a cone $\Gamma$ is of type 1 if
all positive $\lambda_i$ axis belong to $\partial \Gamma$ (e.g., all the $\Gamma_k$ with $2\leq k\leq m$). Note that the cone in our theorem is also of type 1.
For the cone of type 1, it follows from \cite{guanarxiv2014} that the notion of $\mathcal{C}$-subsolution in \cite{gaborjdg} and the  notion of subsolution in \cite{guan2014} are equivalent. Hence by
\cite[Theorem 2.18]{guan2014} (see  \cite[Lemma 2.2]{guanshisui2015} for a refined version) there exist  constants $R>0$ and $\kappa>0$ such that if $|\lambda(\vartheta_{u}^\flat)|\geq R$, then there holds
\begin{equation}
\label{ijfijetaij}
\sum_{i,j}F^{ij}(\vartheta_{u}^\flat)\left((\chi^\flat)_i{}^j-(\vartheta_{u}^\flat)_i{}^j\right)\geq \kappa+\kappa\mathcal{F}.
\end{equation}
We assume that $\sup_M|\partial u|_g\gg 1;$ otherwise the conclusion follows.

Thanks to Assertion \eqref{gaborlem92} of Lemma \ref{lem9}, we deduce that
\begin{equation}
\label{ftau6}
\mathcal{F}\geq \tau.
\end{equation}
We use the auxiliary function in \cite{ctwjems} in the almost complex case
\begin{equation*}
 P:=e^{\rho(\eta)}|\partial u|_g^2,
\end{equation*}
where $\eta= -u+\sup_Mu+1$ and $\rho\in C^\infty(\mathbb{R},\mathbb{R})$ will be determined later.

We assume that $P$ attains its maximum at the interior point $\mathbf{x}_0;$ otherwise the conclusion follows from \eqref{equzeroorder}.
Let $e_1,\cdots,e_m$ be the local unitary frame with respect to $g$ near $\mathbf{x}_0$ such that $\left((\vartheta_u)_{i\bar j}\right)$ and $\left(F^{i\bar j}\right)$  are diagonal  by \eqref{fijexpression}.
At $\mathbf{x}_0,$ 
we have
\begin{equation}
 \label{0geqlp}
0\geq L(P)
=e^{\rho}L\left(|\partial u|_g^2\right)+|\partial u|_g^2L(e^\rho)+2\Re\left(\sum_iF^{i\bar i}\left(\nabla_i\left(|\partial u|_g^2\right)\right)\left(\nabla_{\bar i}\left(e^\rho\right)\right)\right),
\end{equation}
where $L$ is defined by \eqref{defnl}.

From \eqref{defnvartheta}, \eqref{defnl} and \eqref{ijfijaij}, it follows that
\begin{align}
\label{levarphi}
 L(e^\rho)
=&e^\rho\left(\rho''+(\rho')^2\right)F^{i\bar i}|\nabla_i \eta|^2+e^\rho \rho'\sum_{i}F^{i\bar i}\left(\nabla_i\nabla_{\bar i}\eta+W_{i\bar i}(\mathrm{d}\eta)\right)\\
=&e^\rho\left(\rho''+(\rho')^2\right)F^{i\bar i}|\nabla_i \eta|^2+e^\rho \rho'\sum_{i}F^{i\bar i}\left(\chi_{i\bar i}-(\vartheta_{u})_{i\bar i}\right).\nonumber
\end{align}

A direct calculation yields that
 \begin{align}
 \label{refii}
 &2\Re\left(F^{i\bar i}\nabla_i(e^\rho)\nabla_{\bar i}(|\partial u|_g^2)\right)\\
=&2\Re\left(\sum_{k}e^\rho \rho'F^{i\bar i}\left(\nabla_i\eta\right)\left(\nabla_k\nabla_{\bar i}u\right)\left(\nabla_{\bar k}u\right) \right)
+2\Re\left(\sum_{k}e^\rho\rho'F^{i\bar i}\left(\nabla_i\eta\right)\left(\nabla_{\bar i}\nabla_{\bar k}u\right)\left(\nabla_{k}u\right) \right),\nonumber
\end{align}
where
\begin{align}
\label{refii1}
 &2\Re\left(\sum_{k}e^\rho\rho'F^{i\bar i}\left(\nabla_i\eta\right)\left(\nabla_k\nabla_{\bar i}u\right)\left(\nabla_{\bar k}u\right) \right)\\
=&2\Re\left(\sum_{k}e^\rho\rho'F^{i\bar i}\left(\nabla_i\eta\right)\left((\vartheta_u)_{k\bar i}-\chi_{k\bar i}-W_{k\bar i}(\mathrm{d}u) \right)\left(\nabla_{\bar k}u\right) \right)\nonumber\\
\geq&2 \sum_{k}e^\rho\rho'F^{i\bar i}\left(\nabla_i\eta\right) (\vartheta_u)_{k\bar i} \left(\nabla_{\bar k}u\right)
-2 \sum_{k}e^\rho\rho'F^{i\bar i}\left(\nabla_i\eta\right) \chi_{k\bar i} \left(\nabla_{\bar k}u\right)  \nonumber\\
&-\varepsilon e^\rho(\rho')^2|\partial u|_g^2\sum_{i}F^{i\bar i}\left(\nabla_i\eta\right)\left(\nabla_{\bar i}\eta\right)
 -C\varepsilon^{-1}e^\rho|\partial u|_g^2\mathcal{F},\nonumber
\end{align}
and
\begin{align}
\label{refii2}
 &2\Re\left(\sum_{k}e^\rho\rho'F^{i\bar i}\left(\nabla_i\eta\right)\left(\nabla_{\bar i}\nabla_{\bar k}u\right)\left(\nabla_{k}u\right) \right)\\
\geq&-(1-\varepsilon)\sum_kF^{i\bar i}e^\rho|\nabla_i\nabla_ku|^2
-(1+2\varepsilon)e^\rho(\rho')^2|\partial u|_g^2\sum_iF^{ii}\left(\nabla_i\eta\right)\left(\nabla_{\bar i}\eta\right),\nonumber
\end{align}
where we use
\begin{equation*}
  2ab\geq -(1-\varepsilon)a^2-(1+2\varepsilon)b^2,\quad \forall \varepsilon\in(0,1/2].
\end{equation*}
Substituting \eqref{refii1} and \eqref{refii2} into \eqref{refii} implies that
\begin{align}
\label{refiiesti}
  &2\Re\left(F^{i\bar i}\nabla_i(e^\rho)\nabla_{\bar i}(|\partial u|_g^2)\right)\\
\geq&2 \sum_{k}e^\rho\rho'F^{i\bar i}\left(\nabla_iu\right) (\vartheta_u)_{k\bar i} \left(\nabla_{\bar k}u\right)
 -C_0\varepsilon^{-1}e^\rho |\partial u|_g^2\mathcal{F}-(1-\varepsilon)\sum_iF^{ii}e^\rho|\nabla_i\nabla_ku|^2\nonumber\\
&
-(1+3\varepsilon)e^\rho(\rho')^2|\partial u|_g^2\sum_iF^{ii}\left(\nabla_i\eta\right)\left(\nabla_{\bar i}\eta\right)
-2 \sum_{k}e^\rho\rho'F^{i\bar i}\left(\nabla_i\eta\right) \chi_{k\bar i} \left(\nabla_{\bar k}u\right)  .\nonumber
\end{align}
Thanks to \eqref{0geqlp}, \eqref{levarphi}, \eqref{lpartialug2} and \eqref{refiiesti}, one can deduce that
\begin{align}
\label{0geqevarphileft}
0\geq & \left(\rho''-3\varepsilon(\rho')^2\right)|\partial u|_g^2\sum_iF^{i\bar i}|\nabla_i \eta|^2
+ \rho'|\partial u|_g^2\sum_{i}F^{i\bar i}  \left(\chi_{i\bar i}-(\vartheta_{u})_{i\bar i}\right)
-C_0\varepsilon^{-1} |\partial u|_g^2\mathcal{F}\\
&+2 \sum_{k} \rho'F^{i\bar i}\left(\nabla_i\eta\right)(\vartheta_u)_{k\bar i} \left(\nabla_{\bar k}u\right)
+2 \Re\left(\sum_k\left(\nabla_ku\right)\left(\nabla_{\bar k}h\right)\right)\nonumber\\
&-2 \sum_{k} \rho'F^{i\bar i}\left(\nabla_i\eta\right) \chi_{k\bar i} \left(\nabla_{\bar k}u\right).  \nonumber
\end{align}
We take
\begin{equation*}
 \rho(\eta):=\frac{e^{A\eta}}{A},\quad \varepsilon:=\frac{Ae^{-A\eta(\mathbf{x}_0)}}{6}.
\end{equation*}
Then it follows from \eqref{0geqevarphileft} that
\begin{align}
\label{0geqevarphileftn}
0\geq & \frac{A}{2}\sum_iF^{i\bar i}|\nabla_i u|^2
+  \sum_{i}F^{i\bar i}  \left(\chi_{i\bar i}-(\vartheta_{u})_{i\bar i}\right)
-\frac{6C_0}{A}\mathcal{F}\\
&-\frac{2}{|\partial u|_g^2} \sum_{i} F^{i\bar i}\left(\nabla_iu\right)(\vartheta_u)_{i\bar i} \left(\nabla_{\bar k}u\right)
+\frac{2}{e^{A\eta(\mathbf{x}_0)}|\partial u|_g^2} \Re\left(\sum_k\left(\nabla_ku\right)\left(\nabla_{\bar k}h\right)\right)\nonumber\\
&+\frac{2}{|\partial u|_g^2} \sum_{k} F^{i\bar i}\left(\nabla_iu\right) \chi_{k\bar i} \left(\nabla_{\bar k}u\right).  \nonumber
\end{align}
The following argument from \cite{guansun2015} splits into two cases.

\textbf{Case 1:}$|\lambda(\vartheta_{u}^\flat|\geq R.$
Note that  $(\chi_{k\bar i})$ is positive definite and that
\begin{align}
\label{onekindestimate}
&-2 \sum_{k}  F^{i\bar i}\left(\nabla_iu\right)(\vartheta_u)_{k\bar i} \left(\nabla_{\bar k}u\right)\\
\geq &-2 |\partial u|_g^2\sum_iF^{i\bar i}(\vartheta_u)_{i\bar i}
=-2 |\partial u|_g^2 f(\lambda(\vartheta_u^\flat))
 \geq  -C |\partial u|_g^2,\nonumber
\end{align}
where we use the fact that
\begin{equation}
\label{filambdai2}
\sum_iF^{i\bar i}(\vartheta_u)_{i\bar i}= f(\lambda(\vartheta_u^\flat)).
\end{equation}
\textbf{Subcase 1.1:} there holds $F^{i\bar i}\geq K$ for some $i$ and $K$ sufficiently large.
Proof by contradiction, together with \eqref{ijfijetaij}, \eqref{0geqevarphileftn} and \eqref{onekindestimate}, yields that $|\partial u|_g\leq C.$

\textbf{Subcase 1.2:} there holds $F^{i\bar i}\leq K$ for all $i,$ then we assume that $(\vartheta_u)_{1\bar 1}\leq \cdots\leq (\vartheta_u)_{m\bar m}.$ For convenience, we set $$\mu_i=\frac{1}{\lambda_i}=\frac{1}{(\vartheta_u)_{i\bar i}},\quad \mbox{for}\quad 1\leq i\leq m.$$
Then we have
\begin{equation*}
f(\lambda)=(\sigma_{m-\ell}(\mu))^{-1/(m-\ell)}=(hC_m^{\ell})^{-1/(m-\ell)},
\end{equation*}
and
\begin{equation}
F^{i\bar i}=f_i=\frac{1}{m-\ell}f^{m-\ell+1}\mu_i^2\sigma_{m-\ell-1;i}(\mu),
\end{equation}
where $\sigma_{m-\ell-1;i_1\cdots i_s}(\mu):=\sigma_{m-\ell-1}(\mu_{|\mu_{i_1}=\cdots=\mu_{i_s}=0}).$
Note that
\begin{equation*}
  \Pi_{i=1}^{m-\ell}\mu_i\geq \frac{\sigma_{m-\ell}(\mu)}{C^\ell_m}= h,
\end{equation*}
which yields that
\begin{equation}
\label{mu1estimate}
\mu_1\leq (\mu_1)^2h^{-1}\Pi_{i=2}^{m-\ell}\mu_i\leq (\mu_1)^2h^{-1}\sigma_{m-\ell-1;1}(\mu)
\leq \frac{(m-\ell)K}{f^{m-\ell+1}h^{-1}}.
\end{equation}
It follows from \eqref{mu1estimate} that
\begin{equation}
\label{sigmamell}
\sigma_{m-\ell-1;i}(\mu)\leq C_m^{\ell+1}(\mu_1)^{m-\ell-1}\leq K',\quad \mbox{for}\quad 1\leq i\leq m.
\end{equation}
The Cauchy-Schwarz inequality, together with \eqref{sigmamell}, yields that
\begin{align}
\label{anotherkindestimate}
&-2 \sum_{k}  F^{i\bar i}\left(\nabla_iu\right)(\vartheta_u)_{k\bar i} \left(\nabla_{\bar k}u\right)\\
\geq &-\frac{1}{m-\ell}f^{m-\ell+1}\sum_i\sigma_{m-\ell-1;i}(\mu)-\frac{1}{m-\ell}f^{m-\ell+1}|\partial u|_g^2 \sum_i\sigma_{m-\ell-1;i}(\mu)\mu_i^2|\nabla_iu|^2\nonumber\\
\geq&-C-|\partial u|_g^2 \sum_iF^{i\bar i}|\nabla_iu|^2.\nonumber
\end{align}
Since $(\chi_{k\bar i})$ is positive definite, it follows from \eqref{ijfijetaij}, \eqref{0geqevarphileftn} and \eqref{anotherkindestimate} that
\begin{equation*}
0\geq   \left(\frac{A}{2}-1\right)\sum_iF^{i\bar i}|\nabla_i u|^2
+  \kappa\mathcal{F}+\kappa
-\frac{6C_0}{A}\mathcal{F}
 -\frac{C}{|\partial u|_g^2}  -\frac{C}{e^{A\eta(\mathbf{x}_0)}},
\end{equation*}
which yields $|\partial u|_g\leq C$ with $A$ sufficiently large.

\textbf{Case 2:} $|\lambda(\vartheta_u^\flat)|\leq R.$ One can deduce that
\begin{equation*}
\left\{\lambda\in\Gamma:\;f(\lambda)\geq \inf_{M}h>\sup_{\partial\Gamma} f\right\}\cap \overline{B_R(\mathbf{0})}\subset\Gamma
\end{equation*}
is a compact set, and hence there exists a constant $C_3$ depending on the background data such that
\begin{equation*}
C_3\geq f_i\geq C_3^{-1}>0,\quad1\leq j\leq m,\quad\text{on}\;M.
\end{equation*}
Similar arguments as in Subcase 1.2 yield that $|\partial u|_g\leq C$ provided that $A$ (resp. $\varepsilon$)  is chosen sufficiently large (resp. small).
\end{proof}
\section{Examples on Compact Hermitian Manifolds without Boundary}\label{secwubian}
In this section, we study Equation \eqref{cma} on compact Hermitian manifolds without boundary and mainly prove Theorem \ref{thmkellhessianquotient}.
\begin{proof}
[Proof of Theorem \ref{thmkellhessianquotient}]
We rewrite the $(k,\ell)$-Hessian quotient equation \eqref{hessianquotient}  and the $(k,\ell)$-$(m-1,m-1)$-Hessian quotient equation \eqref{tklcma} ($1\leq \ell<k\leq m$) as
\begin{equation}
\label{sec7kellhessianquotient}
f(\lambda)
=-\frac{\vartheta_{u}^\ell\wedge\omega^k}
{\vartheta_{u}^k\wedge\omega^k}
=-\frac{\binom{m}{\ell}^{-1}\sigma_\ell(\lambda)}{\binom{m}{k}^{-1}\sigma_k(\lambda)}
\end{equation}
and
\begin{equation}
\label{sec7kellm-1hessionquotient}
f(\lambda)
=-\frac{\left(P_\omega(\vartheta_u)\right)^\ell\wedge\omega^k}
{\left(P_\omega(\vartheta_u)\right)^k\wedge\omega^k}
=-\frac{\binom{m}{\ell}^{-1}\sigma_\ell(T(\lambda))}{\binom{m}{k}^{-1}\sigma_k(T(\lambda))} \end{equation}
where for simplicity we assume that $\chi$ in \eqref{sec7kellhessianquotient} and $P_\omega(\chi)$ in \eqref{sec7kellm-1hessionquotient} are $k$-positive forms. It follows that $g=\left(\frac{\sigma_k(\lambda)}{\sigma_\ell(\lambda)}\right)^{1/(k-\ell)}$ satisfies our structural assumption (i.e., Assumption \eqref{assum1}, Assumption \eqref{assum2} and Assumption \eqref{assum3} in Section \ref{secintro}) by \cite{spruck}. So does $g=\left(\frac{\sigma_k(T(\lambda))}{\sigma_\ell(T(\lambda))}\right)^{1/(k-\ell)}$ by the arguments in \cite{stw1503} (see also Section \ref{secintro}). Then our $f$ in \eqref{sec7kellhessianquotient} and  \eqref{sec7kellm-1hessionquotient} also satisfies our structural assumption since they are in the form of $f=-g^{-(k-\ell)}.$
Hence we will deal with \eqref{sec7kellhessianquotient} and  \eqref{sec7kellm-1hessionquotient} uniformly.
\subsection*{The Method of Continuity} We set
$$
h_0:=\left\{
       \begin{array}{ll}
         \frac{\vartheta_{0}^\ell\wedge\omega^k}
{\vartheta_{0}^k\wedge\omega^k}, & \mbox{for \eqref{sec7kellhessianquotient}},  \\
         \frac{\left(P_\omega(\vartheta_0)\right)^\ell\wedge\omega^k}
{\left(P_\omega(\vartheta_0)\right)^k\wedge\omega^k}, & \mbox{for \eqref{sec7kellm-1hessionquotient}}.\\
       \end{array}
      \right.
$$
We study a family of equations for $(u_t,b_t)\in C^{2,\alpha}(M,\mathbb{R})\times\mathbb{R}$
\begin{equation}
\label{sec7gutt}
G(u_t,b_t):= F(\vartheta_{u_t})+h_0^{1-t}h^te^{b_t}=0,\quad\lambda((\vartheta_{u_t})^\flat)\in\Gamma,\quad \sup_M u_t=0,\quad t\in[0,1].
\end{equation}
We consider
\begin{equation*}
\mathscr{T}:=\Big\{t\in [0,1]:\,\mbox{there exists}\, (u_t,b_t)\in C^{2,\alpha}(M,\mathbb{R})\times\mathbb{R}\,\mbox{solves \eqref{sec7gutt}}  \Big\}.
\end{equation*}
Note that $0\in \mathscr{T}.$ We wish to show that $\mathscr{T}$ is open. Assume that $\hat{t}\in\mathscr{T}.$  It suffices to show that, for some small $\varepsilon>0,$ there exists $v_t\in C^{2,\alpha}(M,\mathbb{R})$ for $[\hat t,\hat t+\varepsilon]$ with $v_{\hat t}=0$ and
$$
\frac{F(\hat{\chi}_{v_t})}{F(\vartheta_{u_{\hat t}})}-h_0^{\hat t-t}h^{t-\hat{t}}e^{b_t-b_{\hat{t}}} =0,\quad \lambda((\hat{\chi}_{v_t})^\flat)\in\Gamma,
$$
where
$$
\hat{\chi}_{v_t}:=\vartheta_{u_{\hat t}}+\sqrt{-1}\partial\bar\partial v_t+W(\mathrm{d}v_t).
$$
Indeed, if we can find such a $v_t$ then $u_t:=u_{\hat{t}}+v_t$ solves \eqref{sec7gutt}, and up to adding a time-dependent constant we can also arrange that $\sup_M u_t=0,$ as desired. We define
$$
\hat\omega=-\frac{\sqrt{-1}}{F(\vartheta_{u_{\hat t}})}\sum_{i,j=1}^mF_{i\bar j}(\vartheta_{u_{\hat t}})\mathrm{d}z_i\wedge\mathrm{d}\bar z_j,
$$
where $(F_{i\bar j}(\vartheta_{u_{\hat t}}))$ is the transport of the inverse of $(F^{i\bar j}(\vartheta_{u_{\hat t}})),$ i.e.,
$$\sum_{q=1}^mF_{i\bar q}(\vartheta_{u_{\hat t}})F^{j\bar q}(\vartheta_{u_{\hat t}})=\delta_{ij},\quad 1\leq i,\,j\leq m.$$
We consider the linear differential operator
$$
A(v)=-\frac{1}{F(\vartheta_{u_{\hat t}})}\sum_{i,j=1}^mF^{i\bar j}(\vartheta_{u_{\hat t}})(\partial_{\bar j}\partial_iv+W_{i\bar j}(\mathrm{d}v)),\quad \forall\,v\in C^2(M,\mathbb{R}).
$$
It is elliptic and its kernel are the constant functions. Denote by $A^*$ the adjoint of $A$ with respect to the $L^2$ inner product with volume form $\hat\omega^m.$
We need argue as  \cite{gauduchon1}(cf. \cite[Theorem 2.2]{ctwjems}). The index of $L$ is zero and hence the kernel of $A^*$ is one dimensional which is spanned by a smooth function $\sigma'$. The maximum principle yields that teach nonzero function in  the image of $A$ must change sign. Since $\sigma'$ is orthogonal to the image of $A,$ it follows that $\sigma'$ must have constant sign. We assume that $f\geq 0.$ The strong maximum principle yields that $\sigma'>0,$ and hence we can write $\sigma'=e^\sigma$ for some $\sigma\in C^\infty(M,\mathbb{R}).$ By adding a constant to $\sigma,$ we may assume that
\begin{equation}
\label{sigmaequ2}
\int_Me^{\sigma}\hat\omega^m=1.
\end{equation}
We will show that we can find $v_t\in C^{2,\alpha}(M,\mathbb{R})$ for $[\hat t,\hat t+\varepsilon)$ such that
\begin{equation*}
\frac{F(\hat{\chi}_{v_t})}{F(\vartheta_{u_{\hat{t}}})}=  \left(\int_M\frac{F(\hat{\chi}_{v_t})}{F(\vartheta_{u_{\hat{t}}})}e^{\sigma}  \hat\omega^m\right)
h_0^{\hat t-t}h^{t-\hat{t}}e^{c_t}
\end{equation*}
where $c_t$ is the normalization constant given by
$$
\int_Mh_0^{\hat t-t}h^{t-\hat{t}}e^{c_t}\hat\omega^m=1.
$$
We set $B_1$ and $B_2$ given by
\begin{align*}
B_1:=&\left\{v\in C^{2,\alpha}(M,\mathbb{R}):\,\int_Mve^{\sigma}\hat{\omega}^m=0,\quad\lambda((\hat{\chi}_{v})^\flat)\in\Gamma\right\},\\
B_2:=&\left\{w\in C^{\alpha}(M,\mathbb{R}):\,\int_Me^we^{\sigma}\hat{\omega}^m=1 \right\}.
\end{align*}
We define $\Psi:\,B_1\to B_2$ by
\begin{equation*}
\Psi(v):=\log \frac{F(\hat{\chi}_{v })}{F(\vartheta_{u_{\hat{t}}})}-\log\left(\int_M\frac{F(\hat{\chi}_{v })}{F(\vartheta_{u_{\hat{t}}})}e^{\sigma}  \hat\omega^m\right).
\end{equation*}
Our goal is to find $v_t$ solving $\Psi(v_t)=(t-\hat{t})(h-h_0)+c_t$ for $t\in[\hat{t},\hat{t}+\varepsilon).$ Note that $\Psi(0)=0 $ by \eqref{sigmaequ2}. It follows from the Inverse Function Theorem that it suffices to show the invertibility of
$$
(D\Psi)_0:\,T_0B_1\to T_0B_2,
$$
where
$$
T_0B_1:=\left\{\zeta\in C^{2,\alpha}(M,\mathbb{R}):\,\int_M\zeta e^{\sigma}\hat\omega^m=0\right\}
$$
and
$$
T_0B_2:=\left\{\rho\in C^{\alpha}(M,\mathbb{R}):\,\int_M\rho e^{\sigma}\hat\omega^m=0\right\}
$$
denote the tangent spaces to $B_1$ and $B_2$ at $0.$ Note that $T_0B_2$ consists of $C^\alpha$ functions orthogonal to the kernel of $L^*.$  A direct calculation yields that
\begin{align*}
(D\Psi)_0(\zeta)=&-\frac{1}{F(\vartheta_{u_{\hat t}})}F^{i\bar j}(\vartheta_{u_{\hat t}})(\partial_{\bar j}\partial_i\zeta+W_{i\bar j}(\mathrm{d}\zeta))-\int_M
\frac{ e^{\sigma}}{F(\vartheta_{u_{\hat t}})} F^{i\bar j}(\vartheta_{u_{\hat t}})(\partial_{\bar j}\partial_i\zeta+W_{i\bar j}(\mathrm{d}\zeta))\hat\omega^m \\
=& A(\zeta)+\int_M  e^\sigma A(\zeta) \hat\omega^m
= A(\zeta)+\int_M  A^*(e^\sigma)\zeta\hat\omega^m
= A(\zeta).
\end{align*}
It follows from the Fredholm alternative that $A$ and hence $(D\Psi)_0$ is an isomorphism, as required.
Assumption \eqref{assum1} of $f$ yields that
at the point $\mathbf{x}_{\mathrm{max}}$ where $u_t$ attains its maximum, there holds
\begin{equation}
\label{btlowerbd}
e^{b_t}\geq \left(\frac{h_0}{h}\right)^t(\mathbf{x}_{\mathrm{max}}).
\end{equation}
and at the point $\mathbf{x}_{\mathrm{min}}$ where $u_t$ attains its minimum, there holds
\begin{equation*}
e^{b_t}\leq \left(\frac{h_0}{h}\right)^t(\mathbf{x}_{\mathrm{min}}).
\end{equation*}
Hence there exists a uniform constant $C>0$ such that
\begin{equation}
\label{btbd}
|b_t|\leq C.
\end{equation}

\subsection*{A Priori Estimates} Given the openness of $\mathscr{T}$ and \eqref{btbd}, it suffices to deduce a priori estimates for \eqref{sec7gutt} to solve \eqref{sec7kellhessianquotient} and \eqref{sec7kellm-1hessionquotient}. For this aim, we need find $\mathcal{C}$-subsolutions of \eqref{sec7gutt} for all $t\in [0,1]$.
\begin{prop}
\label{propkellhessianquotient}
Let $(M,J,g)$ be a compact Hermitian manifold without boundary and $\dim_{\mathbb{C}}M=m,$ where $g$ is the Hermitian metric with respect to the complex structure $J.$ Suppose that $\chi$ is a $k$-positive form  with
\begin{align}
\label{kellconhermitian1}
 &h_0\geq h,\quad\mbox{on}\quad M,\\
 \label{kellconhermitian2}
 &k h  \chi^{k-1}\wedge\omega^{m-k}- \ell\chi^{\ell-1}\wedge\omega^{m-\ell}>0
\end{align}
for $(m-1,m-1)$ form.
The function $\underline{u}=0$ is a $\mathcal{C}$-subsolution of \eqref{sec7gutt} of  the $(k,\ell)$-Hessian quotient equation \eqref{sec7kellhessianquotient} for all $t\in[0,1].$

Suppose that $\chi$ is a real $(1,1)$ form such that $P_\omega(\chi)$ is a $k$-positive form  with
\begin{align}
\label{kellm-1conhermitian1}
 &h_0\geq h,\quad\mbox{on}\quad M,\\
 \label{kellm-1conhermitian2}
 &k h  P_\omega(\chi)^{k-1}\wedge\omega^{m-k}- \ell P_\omega(\chi)^{\ell-1}\wedge\omega^{m-\ell}>0
\end{align}
for $(m-1,m-1)$ form.
The function $\underline{u}=0$ is a $\mathcal{C}$-subsolution of \eqref{sec7gutt} of  the $(k,\ell)$-$(m-1,m-1)$-Hessian quotient equation \eqref{sec7kellhessianquotient}  for all $t\in[0,1].$
\end{prop}
\begin{proof}[Proof of Proposition \ref{propkellhessianquotient}]
We follow the arguments of \cite[Proof of Proposition 22]{gaborjdg}. For the first conclusion, we just need to check that if $\mu'$ denotes any $(m-1)$-tuple from $\lambda(\chi^\flat),$ then
\begin{equation*}
-\frac{\binom{m}{\ell}^{-1}\sigma_{\ell-1}(\mu')}{\binom{m}{k}^{-1}\sigma_{k-1}(\mu')}
\geq-h_0^{1-t}h^te^{b_t},
\end{equation*}
which is equivalent to
\begin{equation}
\label{equivalentform}
k(h_0^{1-t}h^te^{b_t})\chi^{k-1}\wedge\omega^{m-k}-\ell\chi^{\ell-1}\wedge\omega^{m-\ell}>0
\end{equation}
for $(m-1,m-1)$ form.
Note that $G(u_t,b_t)$ satisfies our structural assumption (i.e., Assumption \eqref{assum1}, Assumption \eqref{assum2} and Assumption \eqref{assum3} in Section \ref{secintro}). It follows from \eqref{btlowerbd} and \eqref{kellconhermitian1} that
\begin{equation*}
  e^{b_t}\geq \left(\frac{h_0}{h}\right)^t(\mathbf{x}_{\mathrm{max}})\geq 1,
\end{equation*}
which, together with  \eqref{kellconhermitian2}, yields \eqref{equivalentform}, as desired.

The second conclusion follows similarly. This completes the proof of Proposition \ref{propkellhessianquotient}
\end{proof}
\subsubsection*{Zero Order Estimates}
We need prove
\begin{equation}
\label{thmzeroorderwubian}
\sup_M|u|\leq C.
\end{equation}
We modify the arguments in \cite{twjams,gaborjdg}. We define a second order elliptic operator $B$ by
$$
B(v):=\frac{m(\sqrt{-1}\partial\bar\partial v+W(\mathrm{d}v))\wedge\omega^{m-1}}{\omega^m},\quad \forall\,v\in C^2(M,\mathbb{R}).
$$
It is elliptic and its kernel are the constant functions. Denote by $B^*$ the adjoint of $B$ with respect to the $L^2$ inner product with volume form $\omega^m.$ It follows from the argument for $A$ to get $\sigma$ above that there exists a function $\tau\in C^\infty(M,\mathbb{R})$ such that $\ker B^*=\{r e^\tau:\,\forall\,r\in\mathbb{R}\}$ and
\begin{equation}
\label{regutau}
\int_Me^\tau\omega^m=1.
\end{equation}
Standard PDE theory (see, e.g., \cite[Appendix A]{as2}) yields that there exists a Green function $G$ for $B$ which satisfies
\begin{align}
\label{gxylowerbd}
&G(\mathbf{x},\mathbf{y})\geq -C,\\
\label{gxyl1bd}
&\|G(\mathbf{x},\cdot)\|_{L^1(M,e^\tau\omega^m)}\leq C,\quad \forall\,\mathbf{x}\in M,
\end{align} for a uniform constant $C>0,$
and
\begin{equation}
\label{greenformula}
u(\mathbf{x})=\int_Mu(\mathbf{y})e^{\tau(\mathbf{y})}\omega^{m}(\mathbf{y})
+\int_MG(\mathbf{x},\mathbf{y})(Bu)(\mathbf{y})e^{\tau(\mathbf{y})}\omega^m(\mathbf{y}),
\end{equation}
where we use \eqref{regutau}.
Since $e^\tau\in \ker B^*,$ we have
\begin{equation}
\label{buetauomegam=0}
\int_M(Bu)(\mathbf{y})e^{\tau(\mathbf{y})}\omega^m(\mathbf{y})
=\int_Mu (\mathbf{y})(B^*(e^{\tau}))(\mathbf{y}) \omega^m(\mathbf{y})=0.
\end{equation}
It follows from \eqref{gxylowerbd}, \eqref{greenformula} and \eqref{buetauomegam=0} that we can assume that \begin{equation}
\label{gxynonnegative}
G(x,y)\geq 0.
\end{equation}
Since $\sum_{i=1}^m\lambda((\vartheta_u)^\flat)>0,$ we have
\begin{equation}
\label{bulowerbd}
B(u)\geq -C
\end{equation}
for some uniform constant $C>0.$

Assume that $u(\mathbf{p})=0=\sup_Mu.$ We can deduce from \eqref{gxyl1bd}, \eqref{greenformula}, \eqref{gxynonnegative} and \eqref{bulowerbd} that
\begin{align*}
\int_{M}(-u)e^{\tau(\mathbf{y})}\omega^m(\mathbf{y})
=&u(\mathbf{p})-\int_MG(\mathbf{p},\mathbf{y})(Bu)(\mathbf{y})
e^{\tau(\mathbf{y})}\omega^m(\mathbf{y})\\
\leq&C\int_MG(\mathbf{p},\mathbf{y}) e^{\tau(\mathbf{y})}\omega^m(\mathbf{y})
\leq C,
\end{align*}
which, together with \eqref{regutau}, yields that
\begin{equation}
\label{ul1bd}
\int_M(-u)\omega^m\leq C.
\end{equation}
It suffices to obtain the lower bound of $L:=\inf_Mu<0$ to get the zero order estimate. Since $\underline{u}=0$ is the  $\mathcal{C}$-subsolution, it follows from Definition \ref{csubsol} that the set
\begin{equation*}
\left(\lambda\left(\vartheta_{0}^\flat\right)+\Gamma_m\right) \cap \Gamma^{h(\mathbf{x})},\quad
\forall\;\mathbf{x}\in M
\end{equation*}
is uniformly bounded. There exist $\delta>0$ and $R>0$ such that there holds
\begin{equation}
\label{eigenbd}
\left(\lambda\left(\vartheta_{0}^\flat\right)-\delta\mathbf{1}+\Gamma_m \right)\cap \Gamma^{h(\mathbf{x})}\subset B_R(\mathbf{0}),\quad\forall\;\mathbf{x}\in M.
\end{equation}
We assume that $u$ attains its minimum at the origin of the local coordinate  chart $ B_2(\mathbf{0})$; otherwise we can get this by a biholomorphic map. Let us work in $B_1(\mathbf{0})$. For $\epsilon>0$ sufficiently small, we set $v=u+\epsilon|\mathbf{z}|^2$. We have
\begin{equation*}
\inf_{|\mathbf{w}|\leq 1}v(\mathbf{w})=v(\mathbf{0})= L,\quad v (\mathbf{z})\geq L+\epsilon,\quad \forall\;\mathbf{z}\in \partial B_1(\mathbf{0}).
\end{equation*}
It follows from \cite[Proposition 11]{gaborjdg} that
\begin{equation}
\label{plebesgue}
c_0\epsilon^{2n}\leq \int_P\det(D^2v),
\end{equation}
where the integration is respect to the Lebesgue measure, and the set $P$ is given by
\begin{equation}
 \label{setp}
 P:=\left\{\mathbf{x}\in B_1(\mathbf{0}):|Dv(\mathbf{x})|\leq \frac{\epsilon}{2},\quad v(\mathbf{y})\geq v(\mathbf{x})+Dv(\mathbf{x})\cdot(\mathbf{y}-\mathbf{x}),\quad\forall\, \mathbf{y}\in B_1(\mathbf{0})\right\}.
 \end{equation}
For any $\mathbf{x}\in P$,  we have $|Dv(\mathbf{x})|\leq \frac{\epsilon}{2}$ and $D^2v(\mathbf{x})\geq 0$ which shows $\partial_i\partial_{\overline{j}}u(\mathbf{x})\geq -\epsilon \delta_{ij}$ and that
 \begin{equation}
   \label{blockiinequality}
 \det(D^2v)\leq 2^{2n}\left(\det(\partial_i\partial_{\overline{j}}v)\right)^2
 \end{equation}
 from the argument in \cite{blockiscience}. We choose $\epsilon$ sufficiently small depending only on $\delta$ and $\omega$ such that
 \begin{equation}
 \label{eigenbd1}
 \lambda\left(\vartheta_{u}^\flat\right)\in
 \lambda\left(\vartheta_{0}^\flat\right)-\delta\mathbf{1}+\Gamma_m,\quad \forall\;\mathbf{x}\in P.
 \end{equation}
 On the other hand, since $u$ is a solution to \eqref{sec7gutt}, we have
 \begin{equation}
 \label{eigenbd2}
 \lambda\left(\vartheta_{u}^\flat\right)\in\partial \Gamma^{h(\mathbf{x})},\quad \forall\;\mathbf{x}\in P.
 \end{equation}
 From \eqref{eigenbd}, \eqref{eigenbd1} and \eqref{eigenbd2}, we deduce that
 $|u_{i\bar j}|$ and hence $|v_{i\bar j}|$ is bounded from above at any point $\mathbf{x}\in P$.
 This, together with \eqref{plebesgue} and \eqref{blockiinequality}, yields that
 \begin{equation}
 \label{volpomegalower}
 c_0\epsilon^{2n}\leq C'\mathrm{Vol}_{\omega}(P).
 \end{equation}
 From \eqref{setp}, we have
 \begin{equation*}
 v(\mathbf{x})<L+\epsilon/2<0,
 \end{equation*}
 where without loss of generality we assume that $L\ll-1$, from which we have
 \begin{equation}
 \label{volpomegaupper}
 \mathrm{Vol}_{\omega }(P)\leq C''\frac{\int_M(-v) \omega^m}{|L+\epsilon/2|}.
 \end{equation}
 Thanks to \eqref{ul1bd}, \eqref{volpomegalower} and \eqref{volpomegaupper}, we get that $L$ is uniformly bounded from below, as required.

\subsubsection*{Second Order Estimates}
We need prove
\begin{equation}
\label{thmsecondorderwubian}
\sup_M|\sqrt{-1}\partial\bar\partial u|\leq C\left(1+\sup_M|\partial u|_g^2\right).
\end{equation}
See Theorem \ref{thm2orderin}.

\subsubsection*{First Order Estimates}
We need prove
\begin{equation}
\label{thmfirstorderwubian}
\sup_M|\partial u|_g\leq C.
\end{equation}
 We use the blowup argument in \cite{gaborjdg,twjams} originated from \cite{dkajm}.
Assume for a contradiction that \eqref{thmfirstorderwubian} does not hold. There exist a sequence of basic real $(1,1)$ form $\chi_{j}$ and basic smooth functions $h_j$ and $u_j$ such that
\begin{align}
&\|\chi_{j}\|_{C^2(M,g)}+\|h_{j}\|_{C^2(M,g)}\leq C\nonumber\\
\label{1ordercon}&\left[\inf_{(\mathbf{x},j)\in M\times\mathbb{N}^*}h_j,\sup_{(\mathbf{x},j)\in M\times\mathbb{N}^*}h_j\right]\subset\left(\sup_{\partial \Gamma}f,\sup_\Gamma f\right),\\
\label{fequj}
F(\vartheta_{j})=&h_j,\quad \text{with}\quad C_j:=\sup_M|\partial u_j|_{g}\to+\infty,\quad \text{as}\quad j\to +\infty,
\end{align}
where \begin{equation*}
\sup_Mu_j=0,\quad\vartheta_j=\chi_j+\sqrt{-1}\partial\bar\partial u_j+W(\mathrm{d}u_j).
\end{equation*}
We also assume that $0$ is a  $\mathcal{C}$-subsolution of  equations given by \eqref{fequj}.

From  \eqref{thmzeroorderwubian}, it follows that $\sup_M|u_j|\leq C$. For each $j$, there exists $\mathbf{p}_j\in M$ such that $\sup_M|\partial u|_g=|\partial u_j|(\mathbf{p}_j)=C_j\to+\infty$ as $j\to +\infty$.
Without loss of generality, we assume that $\lim_{j\to \infty}\mathbf{p}_j=\mathbf{p}_0\in M$ and that $\mathbf{p}_0$ is the center of the local coordinate chart $ B_{2}(0)\subset \mathbb{R}\times \mathbb{C}^n$ with all the points $\mathbf{p}_j\in   B_{1}(\mathbf{0})$.  Now we just need consider $\omega ,\,\chi_{ j},\,u_j,\,h_j,$ as quantities on $B_2(\mathbf{0})$. We also assume that $\mathbf{z}=(z_1,\cdots,z_n)$ is the coordinates on $\mathbb{C}^n$ and that $\omega(\mathbf{0})=\gamma$, where $\gamma$ is the standard Hermitian metric on $\mathbb{C}^n$.

We define
\begin{equation*}
\hat u_j(\mathbf{z}):=u_j(\mathbf{z}(\mathbf{p}_j)+\mathbf{z}/C_j),\quad \text{on}\quad B_{C_j}(\mathbf{0}),
\end{equation*}
which satisfies
\begin{equation*}
\sup_{B_{C_j}(\mathbf{0})}|\hat u_j|\leq C,\quad \text{and}\quad \sup_{B_{C_j}(\mathbf{0})}|\partial \hat u_j|\leq C,
\end{equation*}
where the gradient  and norm are the Euclidean ones. Moreover, from the definition of these functions, it follows that
\begin{equation*}
 \partial_k \hat u_j (\mathbf{0})=C_j^{-1} \partial_k u_j(\mathbf{z}(\mathbf{p}_j)) ,\quad k=1,\cdots,n,
\end{equation*}
which yields that
\begin{equation*}
C\geq |\partial \hat u_j| (\mathbf{0})>c>0.
\end{equation*}
From \eqref{thmsecondorderwubian}, it follows that
\begin{equation}
\sup_{B_{C_j}(\mathbf{0})}|\sqrt{-1}\partial \overline{\partial} \hat u_{j}|_{\gamma}\leq CC_j^{-2}\sup_M|\sqrt{-1}\partial \overline{\partial} u_j|_{g}\leq C.
\end{equation}
Thanks to the elliptic estimates for $\Delta_{\gamma}$ and the Sobolev embedding, we see that for each given compact set $K\subset \mathbb{C}^n$, each $\alpha\in (0,1)$ and $p>1$, there exists a constant $C$ such that
\begin{equation*}
\|\hat u_{j}\|_{C^{1,\alpha}(K)}+\|\hat u_{j}\|_{W^{2,p}(K)}\leq C.
\end{equation*}
This yields that there exists a subsequence of $\hat u_j$ that converges strongly in $C_{\mathrm{loc}}^{1,\alpha}(\mathbb{C}^n)$ as well as weakly in $W^{2,p}_{\mathrm{loc}}(\mathbb{C}^n)$ to a function $u\in C_{\mathrm{loc}}^{1,\alpha}(\mathbb{C}^n)\cap W^{2,p}_{\mathrm{loc}}(\mathbb{C}^n)$ with $\sup_{\mathbb{C}^n}(|u|+|\nabla u|)\leq C$ and $|\nabla u|(\mathbf{0})\geq c>0$. In particular, $u$ is non-constant.

We set
\begin{equation*}
\Phi_j:\,\mathbb{C}^n\to\mathbb{C}^n,\quad \mathbf{z}\mapsto C_j^{-1}\mathbf{z}+\mathbf{x}_j,\quad \mathbf{x}_j:=\mathbf{z}(\mathbf{p}_j).
\end{equation*}
Then we have
\begin{align*}
\hat u_j=&u_j\circ\Phi_j,\quad\text{on}\quad B_{C_j}(\mathbf{0}),\\
\gamma_j:=&C_j^2\Phi_j^*\omega \to \gamma,\quad\text{smoothly on compact set of $\mathbb{C}^n$ as}\quad j\to \infty.
\end{align*}
In particular, $\Phi_j^*\omega  \to 0$ smoothly. Similarly, $\beta_j:=\Phi_j^*\chi_{j}\to 0$ and $\hat W(\mathrm{d}\hat u_j)=\Phi_j^*(W(\mathrm{d}u_j))\to 0$ smoothly. This also shows that
\begin{equation}
\label{eigenvaluematrix}
  \left|\lambda((\gamma_j)^{\bar q s}\left((\beta_j)_{r\overline{q}}+(\hat u_j)_{r\overline{q}}+\hat W_{r\bar q}\right))-\lambda((\hat u_j)_{r\overline{s}})\right|\to 0,\quad\text{smoothly}.
\end{equation}
We rewrite \eqref{fequj} as
\begin{equation}
\label{fequj2}
F\left(C_j^{2}(\gamma_j)^{\bar q s}\left((\beta_j)_{r\overline{q}}+(\hat u_j)_{r\overline{q}}+\hat W_{r\bar q}(\mathrm{d}\hat u_j)\right)\right)=h_j.
\end{equation}
We claim that $u$ is a $\Gamma$-solution (see \cite[Definition 15]{gaborjdg}).
Indeed, we first suppose that there exists a $C^2$ function $v$ such that $v\geq u$ and $v(\mathbf{z}_0)=u(\mathbf{z}_0)$ for some point $\mathbf{z}_0$. By the construction of $u$, for any $\epsilon>0$, there exists a large $N\in\mathbb{N}$ such that if $j\geq N$, then there exist $a_j,\,\mathbf{z}_j$ with $|a_j|<\epsilon,\,|\mathbf{z}_j-\mathbf{z}_0|<\epsilon$ such that
\begin{equation*}
v+\epsilon|\mathbf{z}-\mathbf{z}_0|+a_j\geq \hat u_j,\quad\text{on}\;B_1(\mathbf{z}_0),\;\text{with equality at $\mathbf{z}_j$},
\end{equation*}
and that $\lambda( (\hat u_j)_{k\overline{\ell}})$ lies in the $2\epsilon$ neighborhood of $\Gamma\supset\Gamma_n$ by \eqref{eigenvaluematrix} and \eqref{fequj2}.
This means that $v_{k\overline{\ell}}(\mathbf{z}_j)+\epsilon\delta_{k\ell}\geq (\hat u_j)_{k\overline{\ell}}(\mathbf{z}_j)$ and hence  $v_{k\overline{\ell}}(\mathbf{z}_j)+\epsilon\delta_{k\ell}$ lies in the $2\epsilon$ neighborhood of $\Gamma$, from which we can deduce that $\lambda((v_{k\overline{\ell}}(\mathbf{z}_0)))\in\overline{\Gamma}$ by letting $\epsilon\to0$.

We second suppose that $v$ is a $C^2$ function such that $v\leq u$ and $v(\mathbf{z}_0)=u(\mathbf{z}_0)$. As above, for any $\epsilon>0$, there exists $N_1\in\mathbb{N}$ sufficiently large such that for any $j>N_1$, we can find $a_j,\,\mathbf{z}_j$ with $|a_j|<\epsilon,\,|\mathbf{z}_j-\mathbf{z}_0|<\epsilon$ satisfying
\begin{equation*}
v-\epsilon|\mathbf{z}-\mathbf{z}_0|+a_j\leq \hat u_j,\quad\text{on}\;B_1(\mathbf{z}_0),\;\text{with equality at $\mathbf{z}_j$},
\end{equation*}
This yields that $\left(v_{k\overline{\ell}}(\mathbf{z}_j)-\epsilon\delta_{k\ell}\right)\leq \left((u_j)_{k\overline{\ell}}(\mathbf{z}_j)\right)$. If $\lambda\left(\left(v_{k\overline{\ell}}(\mathbf{z}_j)-3\epsilon\delta_{k\ell}\right)\right)\in \Gamma$, then
$\lambda\left(\left((u_j)_{k\overline{\ell}}(\mathbf{z}_j)\right)\right)\in \Gamma+2\epsilon\mathbf{1}$.

From \eqref{eigenvaluematrix}, it follows that
\begin{equation}
\lambda\left((\gamma_j)^{\bar q s}\left((\beta_j)_{r\overline{q}}+(\hat u_j)_{r\overline{q}}+\hat W_{r\bar q}(\mathrm{d}\hat u_j)\right)\right)\in \Gamma+\epsilon\mathbf{1},
\end{equation}
 and hence
\begin{equation*}
f\left(C_j^2(\gamma_j)^{\bar q s}\left((\beta_j)_{r\overline{q}}+(\hat u_j)_{r\overline{q}}++\hat W_{r\bar q}(\mathrm{d}\hat u_j)\right)\right)>\sigma
\end{equation*}
for any $j>N_1$ ($N_1$ may be chosen larger if necessary), where  $\sigma\in\left(\sup_{(\mathbf{x},j)\in M\times\mathbb{N}^*}h_j ,\sup_\Gamma f\right)$ by Assertion \eqref{gaborlem91} in Lemma \ref{lem9}.
This is a contradiction to \eqref{fequj2}. Finally, we deduce that
$\lambda((v_{k\overline{\ell}}(\mathbf{z}_j)))\in \mathbb{R}^n\setminus(\Gamma+3\epsilon\mathbf{1})$
 and hence $\lambda((v_{k\overline{\ell}}(\mathbf{z}_0)))\in \mathbb{R}^n\setminus \Gamma$
by letting $\epsilon\to0.$

Now we get a non-constant Lipschitz $\Gamma$-solution $u$ since $|\nabla u|(\mathbf{0})\geq c>0$, which is a contradiction to Theorem \cite[Theorem 20]{gaborjdg}. This contradiction yields the desired \eqref{thmfirstorderwubian}.
\subsubsection*{$C^{2,\alpha}$-Estimates and Higher Order Estimates} Given \eqref{thmzeroorderwubian}, \eqref{thmsecondorderwubian} and \eqref{thmfirstorderwubian}, the $C^{2,\alpha}$ estimate for some $0<\alpha<1$ follows from the Evans-Krylov theory (see for example \cite{twwycvpde,chucvpde}). Differentiating the equations and using the Schauder theory (see for example \cite{gt1998}), we then deduce uniform a priori $C^k$ estimates for all $k\geq 0.$

This completes the proof Theorem \ref{thmkellhessianquotient}.
\end{proof}

\end{document}